\documentclass[a4paper, 12pt, twoside, reqno, openright]{amsart}
\usepackage{amsmath}
\usepackage{amssymb}
\usepackage{amsthm}
\usepackage{amscd}
\usepackage{amsopn}
\usepackage{paralist}
\usepackage[UKenglish]{babel}
\usepackage[T1]{fontenc}
\usepackage{multicol}
\usepackage{pifont}
\usepackage[top=1.5in, bottom=1.2in, left=1.2in, right=1.2in]{geometry}
\usepackage[cal=boondox, scr=euler, bb=ams]{mathalfa}
\usepackage{tikz}
\usepackage{wrapfig}
\usepackage{etex}
\usepackage{pictex}
\usepackage[all,cmtip]{xy}

\newenvironment{altenumerate}
   {\begin{list}
      {\textup{(\theenumi)} }
      {\usecounter{enumi}
       \setlength{\labelwidth}{0pt}
       \setlength{\labelsep}{2pt}
       \setlength{\leftmargin}{0pt}
       \setlength{\itemsep}{\the\smallskipamount}
       \renewcommand{\theenumi}{\roman{enumi}}
      }}
   {\end{list}}
\newenvironment{altitemize}
   {\begin{list}
      {$\bullet$ }
      {\setlength{\labelwidth}{0pt}
       \setlength{\labelsep}{2pt}
       \setlength{\leftmargin}{0pt}
       \setlength{\itemsep}{\the\smallskipamount}
      }}
   {\end{list}}

\newtheoremstyle{Teorema}{5pt}{5pt}{\it}{}{\bf}{.}{ }{}
\theoremstyle{Teorema}
\newtheorem{Theorem}{Theorem}[section]
\newtheorem{Corollary}[Theorem]{Corollary}
\newtheorem{Proposition}[Theorem]{Proposition}
\newtheorem{Definition}[Theorem]{Definition}
\newtheorem{PropDef}[Theorem]{Proposition and Definition}
\newtheorem{Lemma}[Theorem]{Lemma}
\newtheoremstyle{Annotazione}{5pt}{5pt}{\rm}{}{\it}{.}{ }{}
\theoremstyle{Annotazione}
\newtheorem{Remark}[Theorem]{Remark}

\newtheorem{Examples}[Theorem]{Examples}
\newtheorem{Question}[Theorem]{Question}

\def\GL{\mathrm{GL}}

\def\Sh{\mathrm{Sh}}
\def\bba{\mathbb{A}}
\def\bbz{\mathbb{Z}}
\def\bbr{\mathbb{R}}
\def\bbq{\mathbb{Q}}
\def\bbc{\mathbb{C}}

\def\bbp{\mathbb{P}}
\def\bbf{\mathbb{F}}
\def\bbn{\mathbb{N}}

\def\bbs{\mathbb{S}}
\def\bbg{\mathbb{G}}

\def\adf{\mathbb{A}_{\mathrm{f}}}

\def\id{\mathrm{id}}

\def\Pic{\operatorname{Pic}}
\def\End{\operatorname{End}}
\def\Hom{\operatorname{Hom}}
\def\Ext{\operatorname{Ext}}
\def\Aut{\operatorname{Aut}}

\def\Gal{\operatorname{Gal}}

\def\Spec{\operatorname{Spec}}

\def\defined{\overset{\mathrm{def}}{=}}
\def\Hup{\mathrm{H}}
\def\Rup{\mathrm{R}}
\def\MK{\mathrm{K}^{\mathrm{M}}}

\def\Sh{\operatorname{Sh}}
\def\Wrat{\mathrm{W}_{\mathrm{rat}}}
\def\pioneet{\pi_1^{\text{\textup{\'{e}t}}}}
\def\pionepath{\pi_1^{\text{\textup{path}}}}
\def\pionehat{\hat{\pi}_1^{\text{\textup{path}}}}
\def\pizeropath{\pi_0^{\text{\textup{path}}}}
\def\lpc{\mathrm{lpc}}
\def\GMod{G\text{-\textbf{\textup{Mod}}}}
\def\GammaFSet{\varGamma\text{-\textbf{\textup{FSet}}}}
\def\UellnSet{U_{\ell^n}\text{-\textbf{\textup{FSet}}}}
\def\UBrellnSet{U_{(\ell^n)}\text{-\textbf{\textup{FSet}}}}

\def\et{\text{\rm \'{e}t}}

\makeatletter
\newcommand*\rel@kern[1]{\kern#1\dimexpr\macc@kerna}
\newcommand*\widebar[1]{%
  \begingroup
  \def\mathaccent##1##2{%
    \rel@kern{0.8}%
    \overline{\rel@kern{-0.8}\macc@nucleus\rel@kern{0.2}}%
    \rel@kern{-0.2}%
  }%
  \macc@depth\@ne
  \let\math@bgroup\@empty \let\math@egroup\macc@set@skewchar
  \mathsurround\z@ \frozen@everymath{\mathgroup\macc@group\relax}%
  \macc@set@skewchar\relax
  \let\mathaccentV\macc@nested@a
  \macc@nested@a\relax111{#1}%
  \endgroup
}

\makeatother

\def\qbar{\widebar{\bbq }}
\def\Fbar{\widebar{F}}
\def\Ebar{\widebar{E}}

\begin{document}

\title{Topological~realisations of~absolute~Galois~groups}
\author{Robert A. Kucharczyk}
\address{D\'{e}partement Mathematik\\ ETH Z\"urich\\ R\"amistrasse 101\\ 8092 Z\"urich\\ Switzerland}
\email{robert.kucharczyk@math.ethz.ch}

\author{Peter Scholze}
\address{Mathematisches Institut\\ Rheinische Friedrich-Wilhelms-Universit\"at Bonn\\ Endenicher Allee 60\\ 53115 Bonn\\ Germany}
\email{scholze@math.uni-bonn.de}
\begin{abstract}
Let $F$ be a field of characteristic $0$ containing all roots of unity. We construct a functorial compact Hausdorff space $X_F$ whose profinite fundamental group agrees with the absolute Galois group of $F$, i.e.~the category of finite covering spaces of $X_F$ is equivalent to the category of finite extensions of $F$.

The construction is based on the ring of rational Witt vectors of $F$. In the case of the cyclotomic extension of $\bbq$, the classical fundamental group of $X_F$ is a (proper) dense subgroup of the absolute Galois group of $F$. We also discuss a variant of this construction when the field is not required to contain all roots of unity, in which case there are natural Frobenius-type automorphisms which encode the descent along the cyclotomic extension.
\end{abstract}

\maketitle

\thispagestyle{empty}

\tableofcontents
\section{Introduction}

\noindent This paper grew out of an attempt to understand whether certain constructions in $p$-adic Hodge theory could potentially have analogues over number fields. One important technique in $p$-adic Hodge theory is the possibility to relate Galois groups of $p$-adic fields with Galois groups or fundamental groups of more geometric objects. Some sample results of this type are the following.

\begin{Theorem}[Fontaine--Wintenberger, {\cite{FontaineWintenberger}}] Let $K$ be the cyclotomic extension $\bbq_p(\zeta_{p^\infty})$ of $\bbq_p$. Then the absolute Galois group of $K$ is isomorphic to the absolute Galois group of $\bbf_p(\! (t)\! )$.
\end{Theorem}

\begin{Theorem}[Fargues--Fontaine, Weinstein, {\cite{WeinsteinGeomGal}}] There is a natural `space' $Z$ defined over $\bbc_p$ whose geometric fundamental group is the absolute Galois group of $\bbq_p$. Formally, $Z$ is the quotient of a $1$-dimensional punctured perfectoid open unit disc by a natural action of $\bbq_p^\times$.
\end{Theorem}

One can regard both of these theorems as instances of the general `tilting' philosophy, \cite{scholzethesis}, which relates objects of mixed characteristic with objects of equal characteristic, the latter of which have a more geometric flavour. An important feature of the tilting procedure is that it only works for `perfectoid' objects; in the case of fields, this is related to the need to pass to the cyclotomic extension, or a similar `big' field. Another common feature is the critical use of ($p$-typical) Witt vectors.

In looking for a global version of these results, one is thus led to consider a `global' version of the Witt vectors, and the standard objects to consider are the big Witt vectors. Recall that for any commutative ring $A$, the ring of big Witt vectors $\mathrm{W}(A)$ can be identified with the set $1+tA[\! [t]\! ]$ of power series with constant coefficient $1$, where addition of Witt vectors corresponds to multiplication of power series. The multiplication is subtler to write down, and is essentially determined by the rule that the product of $1-at$ and $1-bt$ is given by $1-abt$. In particular, there is a multiplicative map $A\to \mathrm{W}(A)$, $a\mapsto [a]=1-at$, called the Teichm\"uller map.

In general, there is a map of algebras $\mathrm{W}(A)\to \prod_n A$ called the ghost map, where the product runs over all integers $n\geq 1$. If $A$ is a $\bbq$-algebra, the ghost map is an isomorphism, so that in particular for a field $F$ of characteristic $0$, $\mathrm{W}(F) = \prod_n F$ is just an infinite product of copies of $F$.

One thus cannot expect $\mathrm{W}(F)$ to have a rich structure. However, work on the K-theory of endomorphisms, \cite{MR0424786}, suggested to look at the following subring of $\mathrm{W}(F)$, called the ring of rational Witt vectors.\footnote{It is this connection, as well as the observation that the Dennis trace map from K-theory to topological Hochschild homology factors canonically over the K-theory of endomorphisms, that led the second author to consider the rational Witt vectors.}

\begin{Definition} Let $A$ be a commutative ring. The \emph{rational Witt vectors} over $A$ are the elements of
\[
\Wrat(A) = \left\{\frac{1+a_1t+\ldots+a_n t^n}{1+b_1t+\ldots+b_m t^m}\,\middle\lvert\, a_i, b_j\in A\right\}\ \subset\mathrm{W}(A).
\]
\end{Definition}

It is not hard to see that $\Wrat(A)$ actually forms a subring of $\mathrm{W}(A)$. The Teichm\"uller map $A\to \mathrm{W}(A)$ factors over $\Wrat(A)$.

Now let $F$ be a field of characteristic $0$ containing all roots of unity, and fix once and for all an embedding $\iota\colon \bbq /\bbz\hookrightarrow F^\times$; this `bigness' hypothesis will be important for the construction, and all constructions will depend on $\iota$. We also fix the standard embedding
\[
\exp\colon\bbq /\bbz\hookrightarrow\bbc^\times ,\quad x\mapsto e^{2\pi ix}.
\]

\begin{Definition} Let $X_F$ be the set of ring maps $\Wrat(F)\to \bbc$ whose restriction along $\bbq /\bbz\buildrel\iota\over\hookrightarrow F^\times\buildrel{[\cdot]}\over\longrightarrow \Wrat(F)^\times$ gives the standard embedding $\exp\colon \bbq /\bbz \to\bbc$. We endow $X_F$ with its natural complex topology, cf.~Definition~\ref{def:complextop}.
\end{Definition}

One can check that $X_F$ is one connected component of the complex points of the scheme $\Spec\Wrat(F)$. Actually, in the paper, $X_F$ will denote a closely related space which is a deformation retract of the space considered here. This variant will be a compact Hausdorff space.

\begin{Theorem}[{Theorem \ref{Theorem:TopologicalEtaleFundamentalGroupOfCHSIsGalois}}]\label{thm:Main} The functor taking a finite extension $E$ of $F$ to $X_E\to X_F$ induces an equivalence of categories between the category of finite extensions of $F$, and the category of connected finite covering spaces of $X_F$. In particular, the absolute Galois group of $F$ agrees with the \'etale fundamental group of $X_F$.
\end{Theorem}

Here, the \'etale fundamental group of a connected topological space classifies, by definition, the finite covering spaces of the latter, cf.~Definition~\ref{def:etfundgroup}. It is in general not directly related to the classical fundamental group defined in terms of paths. We also prove a version of this theorem in the world of schemes, replacing $X_F$ by one connected component of $\Spec (\Wrat(F)\otimes \bbc )$, cf.~Theorem~\ref{thm:MainScheme}.

Contrary to the results in $p$-adic Hodge theory cited above which reflect deep properties about ramification of local fields, this theorem is rather formal. In fact, the proof of the theorem is essentially an application of Hilbert's Satz 90 in its original form. Also, we cannot currently state a precise relationship between this theorem and the results in $p$-adic Hodge theory stated above. Still, we believe that there is such a relation, and that the theorem indicates that the ring of rational Witt vectors is an interesting object; in fact, we would go so far as to suggest to replace all occurences of the big Witt vectors by the rational Witt vectors.\footnote{An instance is the definition of a $\Lambda$-ring, which can be regarded as a commutative ring $A$ with a map $A\to \mathrm{W}(A)$ satisfying certain properties. In most natural examples, including $\mathrm{K}_0$ of a commutative ring, the map $A\to \mathrm{W}(A)$ actually factors through a map $A\to \Wrat(A)$.}

We warn the reader that the space $X_F$ is highly infinite-dimensional, and in general far from path-connected. For example, if $F$ is algebraically closed, its compact Hausdorff version can be (non-canonically) identified with an infinite product of copies of the solenoid (cf.~Proposition/Definition~\ref{propdef:solenoid})
\[
\mathcal S=\varprojlim_{n\in \bbn} \bbs^1 = \Hom(\bbq ,\bbs^1) = \bba /\bbq .
\]

Abstractly, it is clear that any group can be realised as the fundamental group of a topological space, by using the theory of classifying spaces. One may thus wonder what extra content Theorem~\ref{thm:Main} carries. We give several answers to this question. All are variants on the observation that our construction gives an actual topological space, as opposed to a topological space up to homotopy; and in fact, it is not just any space, but a compact Hausdorff space. For example, in Proposition~\ref{Prop:CompactCSImpliesTorsionFree}, we observe that properties of $X_F$ formally imply that the absolute Galois group of $F$ is torsion-free. Also, a compact Hausdorff space has certain finer homotopical and (co)homological invariants which give rise to non-profinitely completed structures on natural arithmetic invariants. From now on, let $X_F$ denote the compact Hausdorff space defined in Section~\ref{sec:GaloisTop} below, which is a deformation retract of the space considered above.

\subsubsection*{\textbf{\textup{Fundamental group.}}} By design, the \'etale fundamental group of $X_F$ agrees with the absolute Galois group of $F$. However, as a topological space, $X_F$ also has a classical fundamental group, given by homotopy classes of loops; we denote it by $\pionepath(X_F)$ (suppressing the choice of base point in the introduction). In general, $\pionepath(X_F)$ could be trivial even when $F$ is not algebraically closed; this happens whenever $F$ is `too big'.

However, for many examples of interest, the situation is better.

\begin{Theorem}[{Section~\ref{subsec:FundGroup}}]\label{thm:PathFund} Assume that $F$ is an abelian extension of a finite extension of $\bbq$. Then $X_F$ is path-connected, and the map $\pionepath(X_F)\to \pioneet(X_F)\cong \Gal(\widebar{F}/F)$ is injective with dense image. Moreover, $\pionepath(X_F)$ carries a natural topology, making it a complete topological group which can be written as an inverse limit of discrete infinite groups. The map $\pionepath(X_F)\to \Gal(\widebar{F}/F)$ is continuous for this topology, but $\pionepath(X_F)$ does not carry the subspace topology.
\end{Theorem}

\begin{Remark} More precisely, $\pionepath(X_F)$ is an inverse limit of discrete groups, each of which is an extension of a finite group by a free abelian group of finite rank. The essential difference between $\pionepath(X_F)$ and $\Gal(\widebar{F}/F)$ is that the Kummer map
\[
F^\times\to \Hom(\Gal(\widebar{F}/F),\hat{\bbz })
\]
lifts to a map
\[
F^\times\to \Hom(\pionepath(X_F),\bbz )\ .
\]
One can characterise the image of $\pionepath(X_F)\to \Gal(\widebar{F}/F)$ as the stabiliser of the class in $\Ext(\widebar{F}^\times,\bbz)$ coming by pullback along a fixed inclusion $\Fbar^\times\hookrightarrow \bbc^\times$ from the exponential sequence
\[
0\to \bbz\to \bbc\buildrel\exp\over\longrightarrow \bbc^\times\to 0,
\]
cf.~Proposition~\ref{prop:descrpionepath}. In particular, the group $\pionepath(X_{\bbq (\zeta_\infty)})\subset \Gal(\qbar/\bbq (\zeta_\infty))$ acts naturally on the group $\log(\widebar{\bbq })\subset \bbc$ of logarithms of algebraic numbers.\footnote{The induced action on $2\pi\mathrm{i}\bbz\subset \bbc$ is trivial, as we are working over the cyclotomic extension. In fact, there can not be an action (except for complex conjugation) on $2\pi\mathrm{i}\bbz$, which presents an obstruction to extending this action beyond the cyclotomic extension.}
\end{Remark}

\subsubsection*{\textbf{\textup{Cohomology.}}} In general, the singular cohomology groups of $X_F$ do not agree with the sheaf cohomology groups (as, e.g., path-connected and connected components do not agree), and sheaf cohomology behaves better. Thus, let $\Hup^i(X_F,A)$ denote the sheaf cohomology with coefficients in the constant sheaf $A$, for any abelian group $A$. The second part of the following theorem is a consequence of the Bloch--Kato conjecture, proved by Voevodsky, \cite[Theorem~6.1]{MR2811603}.

\begin{Theorem}[{Theorem \ref{Thm:EtaleCohomologyOfXFIsGaloisCohomology}, Proposition \ref{prop:corblochkato}}]\label{thm:Cohom} Let $i\geq 0$ and $n\geq 1$.
\begin{enumerate}
\item[{\rm (i)}] There is a natural isomorphism
\[
\Hup^i(X_F,\bbz /n\bbz )\cong \Hup^i(\Gal(\widebar{F}/F),\bbz /n\bbz )\ .
\]
\item[{\rm (ii)}] The cohomology group $\Hup^i(X_F,\bbz )$ is torsion-free. In particular, using (i), there is a canonical isomorphism
\[
\Hup^i(X_F,\bbz )/n\cong  \Hup^i(\Gal(\widebar{F}/F),\bbz /n\bbz )\ .
\]
\end{enumerate}
\end{Theorem}

Thus, one gets natural $\bbz$-structures on the Galois cohomology groups. Note that we regard the choice $\iota$ of roots of unity as fixed throughout; in particular, all Tate twists are trivialised.

\begin{Remark} Recall that by the Bloch--Kato conjecture,
\[
\Hup^i(\Gal(\widebar{F}/F),\bbz /n\bbz )\cong \MK_i(F)/n,
\]
where $\MK_i(F)$ denotes the Milnor K-groups of $F$. One might thus wonder whether $\Hup^i(X_F,\bbz )=\MK_i(F)$. This cannot be true, as the latter contains torsion. However, it is known that all torsion in $\MK_i(F)$ comes via cup product by roots of unity
\[
\bbq /\bbz\otimes \MK_{i-1}(F)\to \MK_i(F),
\]
so that $\MK_i(F)_{\mathrm{tf}} \defined \MK_i(F) / (\bbq /\bbz\cup \MK_{i-1}(F))$ is torsion-free. Also, as we are taking the quotient by a divisible subgroup, one still has
\[
\Hup^i(\Gal(\Fbar /F),\bbz /n\bbz )\cong \MK_i(F)_{\mathrm{tf}}/n.
\]
One could then wonder whether
\[
\Hup^i(X_F,\bbz )=\MK_i(F)_{\mathrm{tf}}.
\]
This is true for $i=0,1$, but not for $i>1$, as the Steinberg relation $x\cup (1-x)=0$ for $x\neq 0,1$ does not hold in $\Hup^2(X_F,\bbz )$. However, we regard this as a defect of $X_F$ that should be repaired:
\end{Remark}

\begin{Question} Does there exist a topological space $X^{\mathrm{M}}_F$ mapping to $X_F$ such that there are isomorphisms
\[
\Hup^i(X^{\mathrm{M}}_F,\bbz )\cong \MK_i(F)_{\mathrm{tf}}
\]
for all $i\geq 0$, which are compatible with the isomorphisms in degrees $i=0,1$ for~$X_F$?
\end{Question}

For algebraically closed fields $F$, the space $X^{\mathrm{M}}_F$ would have to be constructed in such a way as to freely adjoin the Steinberg relation on its cohomology groups; the general case should reduce to this case by descent.

\subsubsection*{\textbf{\textup{Descent along the cyclotomic extension.}}}

So far, all of our results were assuming that $F$ contains all roots of unity. One may wonder whether the general case can be handled by a descent technique. This is, unfortunately, not automatic, as the construction for $F$ involved the choice of roots of unity, so one cannot na\"{\i}vely impose a descent datum. However, there are certain structures on $X_F$ that we have not made use of so far.

First, $X_F$ was defined as (one connected component of) the $\bbc$-valued points of some scheme defined over $\bbq$ (or even $\bbz$). In particular, by the machinery of \'etale homotopy types, its profinite homotopy type acquires an action of (a subgroup of) the absolute Galois group of $\bbq$. This action should, in fact, factor over the Galois group of the cyclotomic extension of $\bbq$, and allow one to define the descent datum. Unfortunately, this requires heavy technology, and does not play well with the purely topological considerations on cohomology and fundamental groups above; however, we record a version of this relationship on the level of cohomology as part of Theorem~\ref{Thm:CompatibilityThreeActionsCohomology} below.

Second, $X_F$ was defined in terms of the rational Witt vectors, and the rational Witt vectors carry extra endomorphisms, given by Frobenius operators.\footnote{In fact, one can combine the first and second observation, which leads to the observation that $\Wrat(F)$ is a $\Lambda$-ring; in fact, (almost tautologically) one for which the map $\Wrat(F)\to \mathrm{W}(\Wrat(F))$ factors over $\Wrat(F)\to \Wrat(\Wrat(F))$.} Thus, one would expect to have Frobenius operators on $X_F$; however, the Frobenius operators exchange connected components, and it turns out that on the connected component $X_F$ there are no remaining operators.\footnote{In fancy language, the `dynamical system' of the connected components of $\Spec (\Wrat(\bbq(\zeta_\infty))\otimes \bbc )$ with its Frobenius operators is one form of the Bost--Connes system, \cite{bostconnes}.} However, one can use a different connected component instead, at least in some situations. In this respect, we prove the following result.

\begin{Theorem}[{Proposition~\ref{Prop:EtaleFGYellF}.(ii), Theorem~\ref{thm:cohomelladic}, Theorem~\ref{Thm:CompatibilityThreeActionsCohomology}}]\label{thm:Descent} Let $\ell$ be a fixed prime, and let $F$ be a perfect field of characteristic different from $\ell$ (but possibly positive) whose absolute Galois group $\Gal(\widebar{F}/F)$ is pro-$\ell$. Let $n\leq \infty$ be maximal such that $\mu_{\ell^n}\subset F$; for simplicity, we assume $n\geq 2$ in case $\ell=2$. Then there is a compact Hausdorff space $Y_{\ell^n,F}$ with an action of $U_{(\ell^n)} = 1+\ell^n\bbz_{(\ell)}$, with the following properties.
\begin{enumerate}
\item[{\rm (i)}] Let $F(\zeta_{\ell^\infty})/F$ be the $\ell$-cyclotomic extension. Then there is a natural isomorphism $\pioneet(Y_{\ell^n,F})\cong \Gal(\widebar{F}/F(\zeta_{\ell^\infty}))$.
\item[{\rm (ii)}] There is a natural isomorphism
\[
\Hup^i(Y_{\ell^n,F},\bbz /\ell^m\bbz )\cong \Hup^i(\Gal(\widebar{F}/F(\zeta_{\ell^\infty})),\bbz /\ell^m\bbz )\ .
\]
Under this isomorphism, the action of $U_{(\ell^n)}$ on the left corresponds to the action of $1+\ell^n\bbz_\ell\cong \Gal(F(\zeta_{\ell^\infty})/F)$ on the right via the tautological embedding $U_{(\ell^n)}\hookrightarrow 1+\ell^n\bbz_\ell$.
\end{enumerate}
\end{Theorem}

We note that there is again an interesting difference between discrete and profinite groups: The Galois group of the cyclotomic extension is profinite, but the Frobenius operators live in a discrete subgroup. This is necessary, as the Frobenius operators will also act on $\Hup^i(Y_{\ell^n,F},\bbz )$.

Finally, let us give a brief summary of the different sections. In Sections 2 and 3, we recall various basic facts about topological fundamental groups, and Pontrjagin duals, respectively. The material here is standard, but not always well-known. In Section 4, we prove Theorem~\ref{thm:Main} in the world of schemes, and in Section 5 we prove the version for topological spaces. Next, in Section 6, we prove Theorem~\ref{thm:PathFund}; this relies on a careful analysis of the path-connected components of $X_F$ and an analysis of the multiplicative groups of large extensions of number fields. In Section 7, we prove Theorem~\ref{thm:Cohom}. Finally, in Section 8, we prove Theorem~\ref{thm:Descent}.

\subsubsection*{\textbf{\textup{Acknowledgements.}}} Part of this work was done while the second author was a Clay Research Fellow. All of it was done while the first author was supported by the Swiss National Science Foundation. The  authors wish to thank Lennart Meier for asking a very helpful question, Markus Land and Thomas Nikolaus for a discussion about Proposition~\ref{Prop:CompactCSImpliesTorsionFree}, and Eric Leichtnam for pointing out some typographical errors in an earlier version.

\subsubsection*{\textbf{\textup{Notation.}}} We denote the profinite completion of the integers by $\hat{\bbz }$, the ring of finite (rational) ad\`eles by $\adf =\hat{\bbz}\otimes_{\bbz }\bbq$ and the full ring of ad\`eles by $\bba =\adf\times\bbr$.

\section{Preliminaries on fundamental groups}

\noindent In this section we assemble a number of results, some well-known, some less so, on different concepts of fundamental groups.

\subsection{Classical fundamental groups and path components} For a topological space $X$ with a base point $x\in X$ we let $\pionepath (X,x)$ be the usual fundamental group defined in terms of loops. To be precise, a loop in $X$ based at $x$ is a continuous map $\gamma\colon [0,1]\to X$ with $\gamma (0)=\gamma (1)$, and for two loops $\gamma$, $\delta$ the product $\gamma\ast\delta$ is defined as `run first through $\delta$, then through $\gamma$',\footnote{Note that this convention is reverse to that prevalent in algebraic topology, but it is common in algebraic geometry and is more convenient when working with categories of covering spaces. Of course, these two conventions yield groups which are opposite groups of one another, hence related by a canonical isomorphism $[\gamma ]\mapsto [\gamma]^{-1}$.} i.e.\ as
\begin{equation*}
\gamma\ast\delta\colon [0,1]\to X,\quad t\mapsto\begin{cases}
\delta (2t)&\text{for }0\le t\le\frac 12,\\
\gamma (2t-1)&\text{for }\frac 12\le 1.
\end{cases}
\end{equation*}
This composition induces a group structure on the set $\pionepath (X,x)$ of homotopy classes of loops in $X$ based at~$x$; we call $\pionepath (X,x)$ the \emph{classical fundamental group} of $X$ at~$x$.

\subsubsection*{Path components} Recall that a space $X$ is \emph{path-connected} if for every two points $x,y\in X$ there is a path in $X$ from $x$ to $y$, i.e.\ a continuous map $\gamma\colon [0,1]\to X$ with $\gamma (0)=x$ and $\gamma (1)=y$. More generally, introduce an equivalence relation on the points of a space $X$ by declaring $x$ and $y$ equivalent if there is a path from $x$ to $y$ in~$X$. Then the equivalence classes of this relation are called the \emph{path components} of $X$; they can be characterised as the maximal path-connected subspaces of~$X$. The set $\pizeropath (X)$ of path components will be equipped with the quotient topology induced by the given topology on~$X$.

Finally note that if $x\in X$ and $X^{\circ }\subseteq X$ is the path component containing $x$, then every loop in $X$ based at $x$ lies in $X^{\circ }$, and similarly for every homotopy of loops. Hence the inclusion $X^{\circ }\hookrightarrow X$ induces an isomorphism of fundamental groups $\pionepath (X^{\circ },x)\to\pionepath (X,x)$.

Since the interval $[0,1]$ is connected, every path-connected space is connected, but the converse does not hold. The most well-known counterexample seems to be the \emph{topologist's sine curve}
\begin{equation*}
T=\{ (0,y)\mid -1\le y\le 1 \}\cup \{ (x,\sin\tfrac 1x)\mid x>0 \}\subset\bbr^2
\end{equation*}
which is connected but has two path components, cf.\ Figure~\ref{Fig:TopologistsSineCurve}. See \cite[Example~117]{MR507446} for more details.
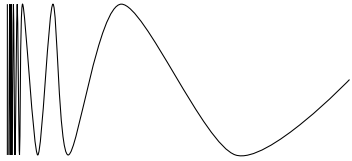
\begin{figure}
\begin{tikzpicture}
\draw plot [smooth] coordinates {(0.02,-1) (0.03,1) (0.04,-1) (0.05,1) (0.06,-1) (0.08,1) (0.1,-1) (0.13,1) (0.16,-1) (0.2,1) (0.4,-1) (0.6,1) (0.8,-1) (1.5,1) (3,-1) (4.5,0)};
\draw (0,1) -- (0,-1);
\end{tikzpicture}
	\caption{The topologist's sine curve}\label{Fig:TopologistsSineCurve}
\end{figure}

 More instructive for our purposes is the following example, to which we will return several times in this section.

\begin{PropDef}\label{propdef:solenoid}
The following topological groups are all canonically isomorphic; each of them is called a (one-dimensional) \emph{solenoid}.
\begin{enumerate}
	\item The Pontryagin dual $\bbq^{\vee }$, i.e.\ the set of group homomorphisms $\bbq\to\bbs^1$ endowed with the compact-open topology, where $\bbq$ carries the discrete topology and $\bbs^1\subset\bbc^{\times }$ is the unit circle;	
	\item the inverse limit $\varprojlim_{n\in\bbn }\bbs^1$, where the set $\bbn$ is partially ordered by divisibility, and for $m\mid n$ the transition map from the $n$-th to the $m$-th copy of $\bbs^1$ is $z\mapsto z^{n/m}$;
	\item the inverse limit $\varprojlim_{n\in\bbn }\bbr /\frac 1n\bbz$, where the transition maps are induced by the identity $\bbr\to\bbr$;
	\item the ad\`ele class group $\bba /\bbq$, where $\bbq$ is diagonally embedded in~$\bba$.
\end{enumerate}
\end{PropDef}
\begin{proof}
	We can write each of these groups as an inverse limit of certain topological groups, indexed by the partially ordered set $\bbn$. The constituents of index $n$ are, respectively:
	\begin{enumerate}
		\item the quotient $A_n=(\frac 1n\bbz )^{\vee }$ of $\bbq^{\vee }$ corresponding to the subgroup $\frac 1n\bbz\subset\bbq$;
		\item the $n$-th copy of $\bbs^1$, denoted by $B_n$;
		\item the quotient $C_n=\bbr /\frac 1n\bbz$; 
		\item the double quotient $D_n=(\frac 1n\hat{\bbz })\backslash\bba /\bbq  =\bba / (\bbq +\frac 1n\hat{\bbz })$.
	\end{enumerate}
	We can write down some isomorphisms between these constituents:
	\begin{itemize}
		\item $B_n\to A_n$, $s\mapsto (\frac 1n\bbz\to\bbs^1,\, q\mapsto s^{nq})$;		
		\item $C_n\to B_n$, $t+\frac 1n\bbz\mapsto\mathrm{e}^{2\pi\mathrm{i}nt}$;
		\item $C_n\to D_n$ induced by the inclusion $\bbr\hookrightarrow\bba =\adf\times\bbr$, $t\mapsto (0,t)$.
	\end{itemize}
	It is easy to see that these give three families of isomorphisms of topological groups $A_n\leftarrow B_n\leftarrow C_n\to D_n$, commuting with the structure maps of the inverse systems, hence defining isomorphisms between the limits.
\end{proof}
\begin{Proposition}\label{Prop:PiZeroPathSolenoid}
	The solenoid $\mathcal{S}$ is a commutative compact Hausdorff group. It is connected, but not path-connected. The path component $\mathcal{S}^{\circ }$ containing the identity is a subgroup, hence the path components of $\mathcal{S}$ are precisely the cosets of~$\mathcal{S}^{\circ }$. Therefore $\pizeropath (\mathcal{S})\cong\mathcal{S}/\mathcal{S}^{\circ }$ acquires the structure of a topological group.
	
	As such it is canonically isomorphic to $\adf /\bbq\cong\hat{\bbz }/\bbz$ with the quotient topology (which is indiscrete), where $\bbq$ is embedded diagonally in~$\adf$.
\end{Proposition}
\begin{proof}
	The claims in the third and fourth sentences of the proposition follow formally from $\mathcal{S}$ being a topological group. 
	
	We next show that $\mathcal{S}$ is connected, using the description (ii): $\mathcal{S}\cong\varprojlim B_n$ with each $B_n\cong\bbs^1$. If $f\colon\varprojlim B_n\to\{ 0,1\}$ is continuous, then by construction of the inverse limit topology $f$ must factor through some $B_n$, hence be constant.
	
	It remains to determine $\pizeropath (\mathcal{S})$. It is most convenient to do this using description (iv): $\mathcal{S}\cong\bba /\bbq$. It is not hard to see that the quotient map $\bba\to\bba /\bbq$ has unique lifting of paths, i.e.\ if $\gamma\colon [0,1]\to\bba /\bbq$ is continuous and $a\in\bba$ is such that $\gamma (0)=a+\bbq$, then there is a unique continuous $\tilde{\gamma }\colon [0,1]\to\bba$ inducing $\gamma$ and satisfying $\tilde{\gamma }(0)=a$. From this we see that the neutral path component of $\bba /\bbq$ is precisely the image in $\bba /\bbq$ of $\{ 0 \}\times\bbr\subset\adf\times\bbr =\bba$. Note that $\bbq$ is embedded diagonally, hence it intersects trivially with $\{ 0 \}\times\bbr$, and we obtain a group isomorphism (but not a homeomorphism, see below!) $\bbr\to\mathcal{S}^0$. Hence the topological group $\pizeropath (\mathcal{S})\cong\pizeropath (\bba /\bbq)$ is isomorphic to $(\bba /\bbr )/\bbq\cong\adf /\bbq$. Since $\bbq$ is dense in $\adf$, this carries the indiscrete topology. Note that $\bbq +\hat{\bbz} =\adf$ and $\bbq\cap\hat{\bbz }=\bbz$, so we may also identify $\pizeropath (\mathcal{S})$ with $\hat{\bbz }/\bbz$.
\end{proof}

\subsubsection*{Locally path-connected spaces} A topological space $X$ is \emph{locally path-connected} if for every $x\in X$ and every open neighbourhood $V\subseteq X$ of $x$ there exists an open neighbourhood $U\subseteq V\subseteq X$ of $x$ which is path-connected.

Let $X$ be a topological space, and denote the given topology by $\mathfrak{O}=\{ V\subseteq X\mid V\text{ open} \}\subseteq\mathfrak{P}(X)$. For an open subset $V\in\mathfrak{O}$ and $x\in V$ set
\begin{equation*}
U(V,x)=\{ y\in X\mid\text{ there is a path }\gamma\colon [0,1]\to V\text{ with }\gamma (0)=x\text{ and }\gamma (1)=y \} ;
\end{equation*}
i.e., $U(V,x)$ is the path component of $V$ containing~$x$. It is then clear that the $U(V,x)$ for varying $x\in X$ and $x\in V\in\mathfrak{O}$ form a basis of a topology $\mathfrak{O}^{\lpc}$ on the set~$X$. We let $X^{\lpc }$ be the topological space with underlying space $X$ and topology $\mathfrak{O}^{\lpc }$. Hence we obtain a continuous but not necessarily open bijection $X^{\lpc }\to X$. The following properties are easily checked:
\begin{Lemma}\label{Lem:CheapPropertiesOfLPCFunctor}
	Let $X$ be a topological space, with topology $\mathfrak{O}$.
	\begin{enumerate}
		\item For every point $x\in X$ the sets $U(V,x)$, where $V$ runs through all elements of $\mathfrak{O}$ containing $x$, is a basis of neighbourhoods of $x$ in $X^{\lpc}$.
		\item The space $X^{\lpc }$ is locally path-connected.
		\item The topology $\mathfrak{O}^{\lpc}$ is the coarsest topology on the set $X$ which is finer than $\mathfrak{O}$ and locally path-connected.
		\item The construction is functorial: if $f\colon Y\to X$ is continuous, then so is $f\colon Y^{\lpc }\to X^{\lpc }$.
		\item If $Y$ is a locally path-connected space, then any continuous map $Y\to X$ factors uniquely as $Y\to X^{\lpc }\to X$. In other words, $X^{\lpc }\to X$ is universal among continuous maps from locally path-connected spaces to~$X$.
		
		This may be rephrased as follows: if $\mathbf{Top}$ denotes the category of topological spaces with continuous maps and $\mathbf{LPC}\subset\mathbf{Top}$ denotes the full subcategory of locally path-connected spaces then the functor $(-)^{\lpc }\colon\mathbf{Top}\to\mathbf{LPC}$ is right adjoint to the inclusion functor $\mathbf{LPC}\hookrightarrow\mathbf{Top}$.\hfill $\square$
	\end{enumerate}
\end{Lemma}
\begin{Examples}\label{Epl:lpcFoliationsSolenoid}
	\begin{altenumerate}
		\item If $X$ is totally disconnected then $X^{\lpc}$ is discrete.
		\item Let $M$ be a smooth manifold and let $\mathcal{F}$ be a foliation on $M$, defined by a vector subbundle $F$ of the tangent bundle $\mathrm{T}M$ such that the sections of $F$ are stable under the Lie bracket (also known as an \emph{involutive} or \emph{integrable} subbundle). Recall that a \emph{leaf} of the foliation is a smooth manifold $L$ together with an injective immersion $i\colon L\hookrightarrow M$ which, for every $p\in L$, induces an identification of $\mathrm{T}_pL$ with $F_p\subseteq\mathrm{T}_{i(p)}M$, and which is maximal with respect to this property.
		
		Since $i$ is injective, we may identify $L$ with a subset of~$M$. In general, however, the topology of $L$ will not be the subspace topology inherited from~$M$. It is not too hard to show that if $i(L)$ is endowed with this subspace topology, then
		\begin{equation*}
		i(L)^{\lpc}\cong L.
		\end{equation*}
		For instance, let $\vartheta\in\bbr$ and consider the \emph{Kronecker foliation of slope~$\vartheta$}. This is the foliation $\mathcal{F}_{\vartheta}$ on the torus $M=\bbr^2/\bbz^2$ given by the subbundle $F_{\vartheta }\subset\mathrm{T}M\cong \bbr^2\times M$ with $F_{\vartheta ,p} =\bbr\cdot {\vartheta\choose 1}$ for every $p\in M$. Each leaf of $\mathcal{F}_{\vartheta}$ is then the image in $M$ of an affine subspace in $\bbr^2$ parallel to $\bbr\cdot {\vartheta\choose 1}$. If $\vartheta\in\bbq$ then every leaf $L$ is homeomorphic to $\bbs^1$, and $L\to i(L)$ is a homeomorphism. If $\vartheta$ is irrational then all leaves are homeomorphic to $\bbr$ and have \emph{dense image} in $M$. They are all translates, via the group structure on $M$, of the leaf through $0$:
		\begin{equation*}
		i\colon\bbr\to M,\quad t\mapsto\begin{pmatrix}
		t\vartheta \\ t
		\end{pmatrix}\bmod\bbz^2.
		\end{equation*}
		The topology on $i(\bbr )$ inherited from $M$ defines a topology $\mathfrak{O}_{\vartheta}$ on $\bbr$. A basis of neighbourhoods of $0$ for this topology is given by the sets
		\begin{equation*}
		\{ t\in\bbr\mid \text{both $t$ and $t\vartheta$ differ by less than $\varepsilon$ from an integer}\}
		\end{equation*}
		for varying $\varepsilon>0$. Hence every $\mathfrak{O}_{\vartheta}$-neighbourhood of $0$ is unbounded. Still, $\mathfrak{O}^{\lpc}_{\vartheta}$ is the Euclidean topology on~$\bbr$.
		
		Note that $M$ is the completion of the topological group $(\bbr ,\mathfrak{O}_{\vartheta })$, and using this it is not hard to see that $(\bbr ,\mathfrak{O}_{\vartheta_1})\simeq (\bbr ,\mathfrak{O}_{\vartheta_2})$ as topological groups if and only if $\vartheta_1$ and $\vartheta_2$ are in the same $\GL_2(\bbz )$-orbit in $\bbp^1(\bbr )\smallsetminus\bbp^1(\bbq )$.
		
		\item There is a similar description of $\mathcal{S}^{\lpc}$, where $\mathcal{S}=\bba /\bbq$ is the solenoid. We have already determined the path components of $\mathcal{S}$ in the proof of Proposition~\ref{Prop:PiZeroPathSolenoid}. Since $\mathcal{S}$ is a topological group, they are all homeomorphic, and they are all dense in $\mathcal{S}$ by Proposition~\ref{Prop:PiZeroPathSolenoid}. One of them is the image of $\bbr$ under the obvious homomorphism $i\colon \bbr\to\bba\to\mathcal{S}$. Again, $i$ is injective and continuous, but not a homeomorphism onto its image. If $\mathfrak{O}$ denotes the topology on $\bbr$ corresponding to the subspace topology on $i(\bbr )\subset\mathcal{S}$, then a basis for $\mathfrak{O}$ is given by the open sets
		\begin{equation*}
		x+\bigcup_{n\in\bbz }\left] kn-\varepsilon ,kn+\varepsilon \right[
		\end{equation*}
		for $x\in\bbr$, $k\in\bbn$ and $\varepsilon >0$. This topology is not locally path-connected, but $\mathfrak{O}^{\lpc }$ is the Euclidean topology on~$\bbr$.
		
		From this we can determine the topology on $\mathcal{S}^{\lpc }$: it is the unique topology on $\mathcal{S}$ for which $\mathcal{S}$ is a topological group and $i\colon\bbr\to\mathcal{S}$ is a homeomorphism onto an open subgroup, where $\bbr$ has the Euclidean topology. Hence $\mathcal{S}^{\lpc}$ is (non-canonically) isomorphic to $\hat{\bbz }/\bbz \times\bbr$, with the discrete topology on $\hat{\bbz }/\bbz$.
	\end{altenumerate}
\end{Examples}
\begin{Corollary}\label{Cor:LPCPreservesManyInvariants}
	Let $X$ be a topological space.
	\begin{enumerate}
		\item The canonical map $\pizeropath(X^{\lpc })\to\pizeropath (X)$ is a bijection.
		\item For any $x\in X$ the canonical map $\pionepath (X^{\lpc },x)\to\pionepath (X,x)$ is a group isomorphism, and similarly for higher homotopy groups defined in the usual way using spheres.
		\item The complex $\mathrm{C}_{\bullet }(X)$ of singular simplices in $X$ with integral coefficients is canonically isomorphic to $\mathrm{C}_{\bullet }(X^{\lpc })$. In particular $X^{\lpc }\to X$ induces isomorphisms on singular homology and cohomology, for any abelian coefficient group.
	\end{enumerate}
\end{Corollary}
\begin{proof}
	Lemma~\ref{Lem:CheapPropertiesOfLPCFunctor} implies that a map of sets $[0,1]\to X$ is continuous if and only if it is continuous when viewed as a map $[0,1]\to X^{\lpc }$, and similarly for homotopies and singular simplices.
\end{proof}

\subsubsection*{Topologies on the classical fundamental group} Let $(X,x)$ be a pointed topological space. Then there exist several natural topologies on $\pionepath (X,x)$.
\begin{altitemize}
	\item The \emph{loop topology} on $\pionepath (X,x)$ is the quotient topology defined by the surjective map $\Omega (X,x)\to\pionepath (X,x)$, where $\Omega (X,x)$ is the loop space of $(X,x)$, i.e.\ the set of all continuous pointed maps $(\bbs^1 ,1)\to (X,x)$ endowed with the compact-open topology. While the loop topology is defined in a very natural way it does not always turn $\pionepath (X,x)$ into a topological group, only into a \emph{quasi-topological group}.
	
	Here a quasi-topological group is a group $G$ together with a topology such that the inverse map $G\to G$, $g\mapsto g^{-1}$, and all multiplication maps $G\to G$, $g\mapsto gh$, and $G\to G$, $g\mapsto hg$, for $h\in G$ are continuous. These conditions do not imply that the multiplication map $G\times G\to G$, $(g,h)\mapsto gh$, is continuous (which would turn $G$ into a topological group).
	
	For instance for the \emph{Hawaiian earrings}
	\begin{equation*}
	H=\bigcup_{n\in\bbn }C_n
	\end{equation*}
	where $C_n\subset\bbr^2$ is a circle of radius $\frac 1n$ centered at $(0,\frac 1n)$, the fundamental group $\pionepath (H,0)$ with the loop topology is a quasi-topological group but not a topological group, see~\cite{MR2810974}.
	\item Brazas \cite{MR2995090} showed that for any pointed space $(X,x)$ there is a finest topology on $\pionepath (X,x)$ such that $\pionepath (X,x)$ becomes a topological group and $\Omega (X,x)\to\pionepath (X,x)$ is continuous. This topology is known as the \emph{$\tau$-topology}. Clearly it agrees with the loop topology if and only if the latter already turns $\pionepath (X,x)$ into a topological group.

	Brazas introduces a generalised notion of covering spaces called \emph{semicovering spaces}, cf.~\cite{MR2954666}. For a semicovering space $p\colon Y\to X$ the subspace topology on the fibre $p^{-1}(x)$ is discrete, and the monodromy action of $\pionepath (X,x)$ on $p^{-1}(x)$ is continuous for the $\tau$-topology. If $X$ is path-connected and locally path-connected this construction provides an equivalence of categories between semicoverings of $X$ and discrete sets with continuous $\pionepath (X,x)$-action.
\item For any monodromy action defined by a semicovering space the point stabilisers will be open subgroups of $\pionepath (X,x)$ for the $\tau$-topology. Hence it makes sense to define a new topology called the \emph{$\sigma$-topology} on $\pionepath (X,x)$ where a neighbourhood basis of the identity is given by the $\tau$-open subgroups (rather than all $\tau$-open neighbourhoods) of $\pionepath (X,x)$.
\item Finally we may consider the completion $\pi_1^{\Gal }(X,x)$ of $\pionepath (X,x)$ with respect to the $\sigma$-topology (more precisely, with respect to the two-sided uniformity defined by the $\sigma$-topology). This group  is complete and has a basis of open neighbourhoods of the identity given by open subgroups. By \cite[Proposition~7.1.5]{MR3379634} it is then a \emph{Noohi group}, i.e.\ the tautological map from $\pi_1^{\Gal }(X,x)$ to the automorphism group of the forgetful functor
\begin{equation*}
\pi_1^{\Gal }(X,x)\text{-\textbf{Sets}}\to\mathbf{Sets}
\end{equation*}
is an isomorphism. Here $\mathbf{Sets}$ is the category of sets and $\pi_1^{\Gal}(X,x)$-\textbf{Sets} is the category of (discrete) sets with a continuous left $\pi_1^{\Gal }(X,x)$-action.

 In the case where $X$ is path-connected and locally path-connected the category $\pi_1^{\Gal}(X,x)$-\textbf{Sets} is again equivalent to the category of semicovering spaces of $X$, and $\pi_1^{\Gal }(X,x)$ can be constructed from that category as the automorphism group of a fibre functor,  see \cite{Klevdal2015} for details.
\end{altitemize}

\subsection{Etale fundamental groups of topological spaces} 
Let $X$ be a connected (but not necessarily path-connected!) topological space and $x\in X$. We shall construct a profinite group $\pioneet (X,x)$ which classifies pointed finite coverings of $(X,x)$, much like the \'etale fundamental group in algebraic geometry does. To do so we proceed analogously to the usual construction for schemes.

\subsubsection*{Categories of finite covering spaces} Recall that a continuous map of topological spaces $p\colon Y\to X$ is a trivial finite covering if there is a finite discrete space $D$ and a homeomorphism $X\times D\to Y$ making the obvious diagram commute; more generally, $p\colon Y\to X$ is a \emph{finite covering} if every point in $X$ has an open neighbourhood $U\subseteq X$ such that the base change $p_U\colon Y_U=p^{-1}(U)\to U$ is a trivial finite covering. The map
\begin{equation*}
X\to\bbn_0,\quad x\mapsto\lvert p^{-1}(x)\rvert ,
\end{equation*}
is continuous. If $X$ is connected, it is therefore constant; the unique value it assumes is called the \emph{degree} of the covering.
\begin{Definition}
	Let $X$ be a topological space. The category $\mathbf{FCov}(X)$ has as objects the pairs $(Y,p)$ where $Y$ is a topological space and $p\colon Y\to X$ is a finite covering, and as morphisms from $(Y_1,p_1)$ to $(Y_2,p_2)$ the continuous maps $f\colon Y_1\to Y_2$ such that $p_1=p_2\circ f$.
\end{Definition}
For every point $x\in X$ we define a functor $\Phi_x\colon\mathbf{FCov}(X)\to\mathbf{FSet}$ (the target being the category of finite sets) by sending $p\colon Y\to X$ to the fibre $p^{-1}(x)$, with the obvious action on morphisms. For a continuous map $f\colon X\to X'$ we obtain a functor $f^{\ast }\colon\mathbf{FCov}(X')\to\mathbf{FCov}(X)$ by pullback. For $x\in X$ there is then a canonical isomorphism $\Phi_{f(x)}\cong\Phi_{x}\circ f^{\ast }$ of functors $\mathbf{FCov}(X')\to\mathbf{FSet}$.
\begin{Proposition}\label{Prop:DecompositionForFCov}
	Let $X$ be a connected topological space and $x\in X$. Let $p\colon Y\to X$ be a finite covering of degree~$d$. Then $Y$ splits into finitely many connected components as $Y=Y_1\coprod\dotsb\coprod Y_n$, each $Y_i\to X$ is a finite covering of $X$ of some degree $d_i$, and $d=d_1+\dotsb +d_n$.
\end{Proposition}
\begin{proof}
	To each open and closed subset $Z\subseteq Y$ we assign a counting function
	\begin{equation*}
	c_Z\colon X\to\{ 0,1,\dotsc ,d \} ,\quad x\mapsto\lvert p^{-1}(x)\cap Z \rvert .
	\end{equation*}
	We claim this is continuous: let $U\subseteq X$ be an open subset over which $p$ becomes trivial. Note we cannot assume $U$ to be connected itself because we have not assumed $X$ to be locally connected. Still, $Z\cap p^{-1}(U)$ is both open and closed in $p^{-1}(U)$, and we may assume the latter to be $U\times D$, where $D$ is a discrete set of cardinality~$d$. Hence for each $\delta\in D$ the locus of $u\in U$ with $(u,\delta )\in Z$ is both open and closed in $U$. Therefore $c_Z$ is continuous on~$U$. But any point in $X$ is contained in a suitable $U$, therefore $c_Z$ is continuous on $X$. But as $X$ is connected, $c_Z$ must be constant, equal to some $0\le d_Z\le d$.
	
	From this argument we also see that $Z\to X$ is a finite covering. The same applies to $Y\smallsetminus Z$. The degrees of the two coverings thus obtained must be strictly smaller than $d$, hence after finitely many steps we arrive at a decomposition into connected finite coverings.
\end{proof}
\begin{Lemma}\label{Lemma:AutomorphismForGaloisCat}
	Let $p\colon Y\to X$ be a finite covering, where $X$ and $Y$ are connected topological spaces. Let $g\in\Aut (Y/X)$, i.e.\ $g$ is a homeomorphism $Y\to Y$ with $p\circ g=g$. If $g$ has a fixed point, then it is the identity.
\end{Lemma}
\begin{proof}
	Similarly to the proof of Proposition~\ref{Prop:DecompositionForFCov}, we show that the set $\{ y\in Y\mid g(y)=y \}$ is both open and closed in~$Y$.
\end{proof}
\begin{Proposition}\label{Prop:FCovIsGaloisCategory}
	Let $X$ be a connected topological space. Then $\mathbf{FCov}(X)$ is a Galois category in the sense of SGA~1, and for every $x\in X$ the functor $\Phi_x$ may serve as a fibre functor.
\end{Proposition}
\begin{proof}
	There are several equivalent characterisations of Galois categories, one being given as follows: an essentially small category $\mathcal{C}$ that admits a functor $\Phi\colon\mathcal{C}\to\mathbf{FSet}$ (called `fibre functor') satisfying the following set of axioms (reproduced from  cf.\ \cite[Expos\'e V.4]{MR2017446}).
	\begin{itemize}
		\item[(G1)] $\mathcal{C}$ has a final object, and arbitrary fibre products exist.
		\item[(G2)] $\mathcal{C}$ has finite coproducts and categorical quotients of objects by finite groups of automorphisms.
		\item[(G3)] Every morphism in $\mathcal{C}$ factors as $\iota\circ\pi$ where $\iota$ is the inclusion of a direct summand in a coproduct and $\pi$ is a strict epimorphism.
		\item[(G4)] $\Phi$ commutes with fibre products and sends right units to right units.
		\item[(G5)] $\Phi$ commutes with finite coproducts, sends strict epimorphisms to strict epimorphisms and sends categorical quotients by finite groups to categorical quotients by finite groups.
		\item[(G6)] If $\Phi (f)$ is an isomorphism, then so is~$f$.
	\end{itemize}
	To show that $\mathbf{FCov}(X)$ and $\Phi_x$ satisfy these axioms is mostly straightforward. The nontrivial parts are the existence of quotients in (G2), and (G6).
	
	For the former let $p\colon Y\to X$ be an object in $\mathbf{FCov}(X)$ and let $G\subseteq\Aut (Y/X)$ be a finite subgroup. Endow $G\backslash Y$ with the quotient topology; we claim that $G\backslash Y\to X$ is an object of $\mathbf{FCov}(X)$, and it will follow formally that it is a categorical quotient for the group action. Take an open subset $U\subseteq X$ over which $Y$ is trivialised; it suffices to show that $G\backslash p^{-1}(U)\to U$ is a finite covering. We may assume that $p^{-1}(U)=U\times\{ 1,2,\dotsc ,d \}$. Then we obtain a continuous, hence locally constant, map
	\begin{equation*}
	\alpha\colon U\to\Hom (G,\mathfrak{S}_d)
	\end{equation*}
	where $\alpha (u)\colon G\to\mathfrak{S}_d$ is the permutation action of $G$ on the fibre $p^{-1}(u)\cong \{ 1,2,\dotsc ,d \}$. From the fact that $\alpha$ is locally constant we deduce that the restriction of $G\backslash Y$ to $U$ is a finite covering of~$U$, as desired.
	
	As for (G6) let $f\colon Y_1\to Y_2$ be a morphism in $\mathbf{FCov}(X)$ which induces a bijection on the fibres over some $x\in X$. We need to show that $f$ is a homeomorphism. First, by an argument analogous to the preceding, we show that $f$ is in fact bijective. Then on any open subset $U\subseteq X$ trivialising both coverings we may assume that $f$ takes the form
	\begin{equation*}
	U\times\{ 1,2,\dotsc ,d \}\to U\times \{  1,2,\dotsc ,d \},\quad (u,\delta)\mapsto (u,\beta (u)(\delta))
	\end{equation*}
	for some finite sets $D_i$ and some continuous map $\beta\colon U\to \mathfrak{S}_d$. It is then clear that $f$ is also open.
\end{proof}
For a Galois category $\mathcal{C}$ with a fibre functor $\Phi$ the group $\pi =\Aut\Phi$ acquires a natural structure of a profinite group, as a projective limit over all the images of $\pi$ in $\Aut (\Phi (Y))$ for $Y\in\operatorname{Ob}\mathcal{C}$. The functor $\Phi$ then factors through the category $\pi$-$\mathbf{FSet}$ of finite sets with a continuous left action by $\pi$, and in fact induces an equivalence between $\mathcal{C}$ and $\pi$-$\mathbf{FSet}$ by \cite[Expos\'e~V, Th\'eor\`eme~4.1]{MR2017446}. In fact, Galois categories are precisely those that are equivalent to $\pi$-$\mathbf{FSet}$ for some profinite group~$\pi$, cf.\ the remarks after \cite[Expos\'e~V, D\'efinition~5.1]{MR2017446}.
\begin{Definition}\label{def:etfundgroup}
	Let $X$ be a connected topological space and $x\in X$. The automorphism group of the fibre functor $\Phi_x\colon\mathbf{FCov}(X)\to\mathbf{FSet}$ is called the \emph{\'etale fundamental group} of $X$ at $x$ and denoted by~$\pioneet (X,x)$.
\end{Definition}
It follows from the formalism of Galois categories that for two points $x,x'\in X$ the groups $\pioneet (X,x)$ and $\pioneet (X,x')$ are isomorphic, the isomorphism being canonical up to inner automorphisms, cf.\ \cite[Expos\'e~V, Corollaire~5.7]{MR2017446}.

In a similar vein, let $X$ be a connected topological space and $\gamma\colon [0,1]\to X$ a path; write $x_0=\gamma (0)$ and $x_1=\gamma (1)$. Then $\gamma$ induces an isomorphism of functors $\varphi_{\gamma }\colon\Phi_{x_0}\to\Phi_{x_1}$ as follows: for any finite covering $p\colon Y\to X$ the pullback $\gamma^{\ast }p\colon \gamma^{\ast }Y=Y\times_{X,\gamma }[0,1]\to [0,1]$ trivialises canonically, i.e.\ for any $t\in [0,1]$ the composition
\begin{equation*}
\Phi_{x_t}(Y)=p^{-1}(\gamma (t))= (\gamma^{\ast }p)^{-1}(t)\hookrightarrow\gamma^{\ast }Y\to\pi_0(\gamma^{\ast }Y)
\end{equation*}
is a bijection, hence there is a canonical identification
\begin{equation}\label{eqn:IsoOfFibreFunctorsInducedByPath}
\Phi_{x_0}(Y)\cong\pi_0(\gamma^{\ast }Y)\cong\Phi_{x_1}(Y).
\end{equation}
We define the isomorphism of functors $\varphi_{\gamma }\colon\Phi_{x_0}\to\Phi_{x_1}$ applied to the object $p\colon Y\to X$ to be (\ref{eqn:IsoOfFibreFunctorsInducedByPath}). By conjugation it induces an isomorphism of \'etale fundamental groups
\begin{equation}\label{eqn:PathInducesCanonicalIsoBetweenEtaleFundamentalGroups}
\tau_{\gamma }\colon\pioneet (X,x_0)=\Aut\Phi_{x_0}\to\Aut\Phi_{x_1}=\pioneet (X,x_1),\quad \alpha\mapsto\varphi_{\gamma }\circ\alpha\circ\varphi_{\gamma}^{-1}.
\end{equation}
The class of this isomorphism up to inner automorphisms is precisely the canonical class of isomorphisms for any two points of $X$ mentioned above.

\subsubsection*{Continuity properties}

Another important property of \'etale fundamental groups is their compatibility with cofiltered projective limits.

\begin{Proposition}\label{Prop:FiniteCoveringDefinedOnFiniteLevel}
	Let $(X_{\alpha })$ be a cofiltered projective system of compact Hausdorff spaces and let $X=\varprojlim_{\alpha }X_{\alpha }$.
	\begin{enumerate}
	\item Let $p\colon Y\to X$ be a finite covering space. Then there exists some $\alpha_0$, a finite covering $p_0\colon Y_0\to X_{\alpha_0}$ and a pullback diagram
	\begin{equation*}
	\xymatrix{
		Y \ar[r] \ar[d]_-p	& Y_0\ar[d]^-{p_0}\\
		X \ar[r] & X_{\alpha_0}.
	}
	\end{equation*}
	\item Let $p_1\colon Y_1\to X_{\alpha_1}$ and $p_2\colon Y_2\to X_{\alpha_2}$ be finite covering spaces, and let $f\colon Y_1\times_{X_{\alpha_1}}X\to Y_2\times_{X_{\alpha_2}}X$ be a continuous map commuting with the projections to~$X$. Then there exists some $\alpha_0\ge\alpha_1,\alpha_2$ such that $f$ is the base change along $X\to X_{\alpha_0}$ of a continuous map $Y_1\times_{X_{\alpha_1}}X_{\alpha_0}\to Y_2\times_{X_{\alpha_2}}X_{\alpha_0}$ commuting with the projections to~$X_{\alpha_0}$.
	\end{enumerate}
\end{Proposition}
\begin{proof} We will only prove (i), the proof for (ii) being very similar.
	
	Call a subset of $X$ \emph{basis-open} if it is the preimage of an open set in some $X_{\alpha }$. As the name suggests, these form a basis of the topology on~$X$.
	
	Let $\mathfrak{U}$ be the set of basis-open subsets of $X$ on which $p$ becomes trivial. Then $\mathfrak{U}$ is an open cover of $X$, and since $X$ is compact there exists a finite subcover, say $\{ U_1,\dotsc ,U_n \}$. Then there exist finite sets $D_1,\dotsc ,D_n$ and continuous functions $\varphi_{ij}\colon U_{ij}=U_i\cap U_j\to\operatorname{Isom}(D_i,D_j)$ satisfying the cocycle condition
	\begin{equation*}
	\varphi_{jk}(u)\circ\varphi_{ij}(u)=\varphi_{ik}(u)\qquad\text{for all }u\in U_{ijk}=U_i\cap U_j\cap U_k
	\end{equation*}
	such that $Y$ is isomorphic to the union of the spaces $U_i\times D_i$, glued along the~$\varphi_{ij}$.
	
	Next, let $\mathfrak{V}$ be the set of all basis-open subsets $V\subseteq X$ such that $V$ is contained in some $U_{ij}$, and for every $i,j$ with $V\subseteq U_{ij}$ the restriction $\varphi_{ij}\rvert_V$ is constant. Again, $\mathfrak{V}$ is an open cover of $X$ and hence has a finite subcover $\{ V_1,\dotsc ,V_m \}$. Since the projective system $(X_{\alpha })$ is cofiltered and the $U_i$ and $V_j$ are basis-open sets, there exists some $\alpha_0$ such that all $U_i$ and $V_j$ are preimages of open sets in~$X_{\alpha_0}$. Then also the $\varphi_{ij}$ are compositions with functions on $X_{\alpha_0}$, and we see that $p\colon Y\to X$ is the pullback of a finite covering defined on~$X_{\alpha_0}$.
\end{proof}
\begin{Remark}
	With a little more effort we can show that in the situation of Proposition~\ref{Prop:FiniteCoveringDefinedOnFiniteLevel}, for a compatible system of basepoints and under the assumption that the $X_{\alpha }$ are connected (and hence so is $X$), the natural map
	\begin{equation*}
	\pioneet (X,x)\to\varprojlim_{\alpha }\pioneet (X_{\alpha },x_{\alpha })
	\end{equation*}
	is an isomorphism.
\end{Remark}

\begin{Lemma}\label{Lem:SimplyConnectedProfiniteCovIsUniversal}
	Let $X$ be a connected compact Hausdorff space, and let $(p_{\alpha }\colon Y_{\alpha }\to X)_{\alpha }$ be a cofiltered projective system of connected finite covering spaces such that $\tilde{Y}=\varprojlim_{\alpha }Y_{\alpha }$ has trivial \'etale fundamental group.
	
	Then every connected finite covering space of $X$ is dominated by some~$Y_{\alpha }$.
\end{Lemma}
\begin{proof}
	Let $Z\to X$ be a connected finite covering. By assumption, the pullback covering $Z\times_X\tilde{Y}\to\tilde{Y}$ splits, i.e.\ $Z\times_X\tilde{Y}$ is isomorphic to $\tilde{Y}\times D$ as a $\tilde{Y}$-space, for some finite discrete set~$D$. By Proposition~\ref{Prop:FiniteCoveringDefinedOnFiniteLevel} this splitting has to occur at a finite level, i.e.\ there has to be some $\alpha_0$ such that $Z\times_XY_{\alpha_0}\simeq Y_{\alpha_0}\times D$ as a $Y_{\alpha_0}$-space. Choosing some $d\in D$ we obtain a commutative diagram
	\begin{equation*}
	\xymatrix{
		Y_{\alpha_0}\ar@{^{(}->}[r]^-{\id\times d}\ar@{=}[d] & Y_{\alpha_0}\times D \ar[r]^-{\simeq }\ar[d] & Z\times_XY_{\alpha_0} \ar[r]\ar[d] & Z\ar[d]\\
		Y_{\alpha_0}\ar@{=}[r] & Y_{\alpha_0}\ar@{=}[r] & Y_{\alpha_0}\ar[r] & X
	}
	\end{equation*}
	The composition of the upper horizontal maps $Y_{\alpha }\to Z$ is a continuous map between finite covering spaces of $X$, respecting the projections to $X$, hence itself a finite covering, as desired.
\end{proof}

\subsubsection*{Homotopy invariance} We now show that \'etale fundamental groups are homotopy invariant. We will make extensive use of the following classical result:
\begin{Proposition}[Unique Homotopy Lifting Property]\label{Prop:UHLP}
	Let $X$ be a topological space, $p\colon Y\to X$ a finite covering, and $S$ another topological space. Assume we are given a homotopy, i.e.\ a continuous map $H\colon S\times [0,1]\to X$, together with a lift of $H(-,0)$ to $Y$, i.e.\ a commutative diagram of continuous maps
	\begin{equation*}
	\xymatrix{
		S\times \{0 \} \ar[r] \ar@{^{(}->}[d] & Y\ar[d]^-{p} \\
		S\times [0,1] \ar[r]_-H & X.
	}
	\end{equation*}
	Then there exists a unique continuous map $S\times [0,1]\to Y$ making the resulting diagram
	\begin{equation*}
	\xymatrix{
		S\times \{0 \} \ar[r] \ar@{^{(}->}[d] & Y\ar[d]^-{p} \\
		S\times [0,1] \ar[r]_-H \ar@{-->}[ur] & X.
	}
	\end{equation*}
	commute.\hfill $\square$
\end{Proposition}
This is of course well-known, see e.g.\ \cite[Proposition~1.30]{MR1867354}; however, we wish to explicitly stress that no (local) connectivity assumptions about $X$ are made.

\begin{Proposition}\label{Prop:HomotopicMapsInduceAlmostSameMapsOnPioneet}
	Let $X$ and $Z$ be connected topological spaces, let $f,g\colon X\to Z$ be continuous, let $H$ be a homotopy between them (i.e.\ a continuous map $H\colon X\times [0,1]\to Z$ such that $H(\xi ,0)=f(\xi )$ and $H(\xi ,1)=g(\xi )$ for all $\xi\in X$), and let $x\in X$ be a basepoint. These data determine a path $\gamma\colon [0,1]\to Z$ by $\gamma (t)=H(x,t)$.
	
Then the diagram
\begin{equation*}
\xymatrix{
	&\pioneet (Z,f(x))\ar[dd]^{\tau_{\gamma }}_{\cong }\\
	\pioneet (X,x)\ar[ur]^-{f_{\ast }}\ar[dr]_-{g_{\ast }}\\
	&\pioneet (Z,g(x)),
	}
\end{equation*}
where $\tau_{\gamma}$ is the map in (\ref{eqn:PathInducesCanonicalIsoBetweenEtaleFundamentalGroups}), commutes.
\end{Proposition}
To show this we first need a lemma.
\begin{Lemma}\label{Lem:LemmaForHomotopyInvarianceOfPioneet}
	Let $X$ be a topological space and let $q\colon W\to X\times [0,1]$ be a finite covering. For any $t\in [0,1]$ consider the restriction $q_t\colon W_t=q^{-1}(X\times\{ t \} )\to X\times\{ t \}\cong X$. Then there is a canonical isomorphism $W_0\cong W_1$ in $\mathbf{FCov}(X)$; the construction of this isomorphism is functorial in~$W$. 
\end{Lemma}
\begin{proof}
	Applying Proposition~\ref{Prop:UHLP} to the diagram
	\begin{equation*}
	\xymatrixcolsep{4pc}
	\xymatrix{
		W_0\ar@{^{(}->}[r]\ar@{^{(}->}[d]_-{\id_{W_0}\times 0}& W\ar[d]_-q\\
		W_0\times [0,1]\ar[r]_-{q_0\times\id_{[0,1]}}& X\times [0,1],
		}
	\end{equation*}
	we deduce the existence of a unique continuous map $W_0\times [0,1]\to W$ making the resulting diagram commute. In particular this map induces an isomorphism on the fibres over any point in $X\times [0,1]$ of the form $(x,0)$. By Proposition~\ref{Prop:FCovIsGaloisCategory} (or more precisely by axiom (G6) for Galois categories, mentioned in the proof thereof) it must be a homeomorphism $W_0\times [0,1]\to W$. Functoriality is straightforward.
\end{proof}
\begin{proof}[Proof of Proposition~\ref{Prop:HomotopicMapsInduceAlmostSameMapsOnPioneet}]
	We first note that $H$ induces an isomorphism of functors $f^{\ast }\Rightarrow g^{\ast }\colon\mathbf{FCov}(Z)\to\mathbf{FCov}(X)$, which by abuse of notation we call $H^{\ast }$, in the following way: for every finite covering $p\colon Y\to Z$ we consider the pullback $q=H^{\ast }p\colon H^{\ast }Y\to X\times [0,1]$. The natural isomorphism from Lemma~\ref{Lem:LemmaForHomotopyInvarianceOfPioneet} can then be rewritten as $f^{\ast }Y\overset{\simeq }{\to } g^{\ast }Y$, and it is easy to check that this indeed defines an isomorphism of functors $H^{\ast }\colon f^{\ast }\Rightarrow g^{\ast }$. The `horizontal' composition of $H^{\ast }$ with the identity on the fibre functor $\Phi_x$ induces an isomorphism of functors $\Phi_x\circ f^{\ast }\Rightarrow\Phi_x\circ g^{\ast }$; this isomorphism can be identified with
	\begin{equation*}
	\Phi_x\circ f^{\ast }\cong\Phi_{f(x)}=\Phi_{\gamma (0)}\overset{\varphi_{\gamma }}{\Rightarrow }\Phi_{\gamma (1)}=\Phi_{g(x)}\cong\Phi_x\circ g^{\ast },
	\end{equation*}
	where $\varphi_{\gamma }$ is the isomorphism of functors from (\ref{eqn:IsoOfFibreFunctorsInducedByPath}). This identity of isomorphisms between functors can be translated into the identity $\tau_y\circ f_{\ast }=g_{\ast }$ of maps between automorphism groups of fibre functors, i.e.\ the commutativity of the diagram under consideration.
\end{proof}
Two consequences are easily drawn:
\begin{Corollary}\label{Cor:HomotopicPointPreservingInduceIdenticalMapsOnPioneet}
	Let $(X,x)$ and $(Z,z)$ be connected pointed topological spaces, and let $f,g\colon (X,x)\to (Z,z)$ be homotopic continuous maps in the pointed sense, that is, assume there exists a continuous map $H\colon X\times [0,1]\to Z$ with $H(\xi ,0)=f(\xi )$ and $H(\xi ,1)=g(\xi )$ for all $\xi\in X$, and also $H(x,t)=z$ for all $t\in [0,1]$.
	
	Then the group homomorphisms $f_{\ast },g_{\ast }\colon\pioneet (X,x)\to\pioneet (Z,z)$ are equal.
\end{Corollary}
\begin{proof}
	In this case the path $\gamma\colon [0,1]\to Z$ as in Proposition~\ref{Prop:HomotopicMapsInduceAlmostSameMapsOnPioneet} is constant, hence induces the identity automorphism of $\pioneet (Z,z)$.
\end{proof}
\begin{Corollary}\label{Cor:HomotopyEquivalenceInducesIsoOnPioneet}
	Let $(X,x)$ and $(Y,y)$ be pointed topological spaces, and let $f\colon (X,x)$ $\to (Y,y)$ be a pointed homotopy equivalence, that is, assume there exists a pointed map $g\colon (Y,y)\to (X,x)$ such that $f\circ g$ and $g\circ f$ are homotopic in the pointed sense to the respective identities.
	
	Then $X$ is connected if and only if $Y$ is connected. Assuming this to be the case, $f^{\ast }\colon\mathbf{FCov}(Y)\to\mathbf{FCov}(X)$ is an equivalence of categories and $f_{\ast }\colon\pioneet (X,x)\to\pioneet (Y,y)$ is an isomorphism of topological groups.
\end{Corollary}
\begin{proof}
	It suffices to show that connectedness is preserved under homotopy equivalence, then the remainder will follow formally from Corollary~\ref{Cor:HomotopicPointPreservingInduceIdenticalMapsOnPioneet}. So, let $f\colon X\to Y$ be a homotopy equivalence with quasi-inverse $g\colon Y\to X$. If $Y$ is disconnected there exists a continuous surjection $c\colon Y\to\{ 0,1\}$. Then $c\circ f\circ g$ is homotopic to $c$; but since these two maps have discrete image, they must then be identical. Hence $c\circ f$ must be surjective, hence $X$ is disconnected as well.
\end{proof}
The assumptions in Corollary~\ref{Cor:HomotopyEquivalenceInducesIsoOnPioneet} are met for the inclusion of a \emph{deformation retract}. Recall that a subspace $A$ of a topological space $X$ is called a deformation retract if there exists a continuous map $H\colon X\times [0,1]\to X$ with the following properties:
\begin{enumerate}
	\item $H(x,0)=x$ for all $x\in X$;
	\item $H(x,1)\in A$ for all $x\in X$;
	\item $H(a,t)=a$ for all $a\in A$ and $t\in [0,1]$.
\end{enumerate}
Such a map $H$ is then called a \emph{defining homotopy} for the deformation retract $A\subseteq X$, and the map $r\colon X\to A$ sending $x$ to $H(x,1)$ is called a \emph{deformation retraction}. Note that some authors do not require condition~(iii), and call a deformation retraction in our sense a `strong deformation retraction'.

\subsubsection*{The \'etale fundamental group as a limit of deck transformation groups} There is another way to view the \'etale fundamental group which will be useful; for proofs cf.\ \cite[Expos\'e~V.4]{MR2017446}. A general fact about Galois categories is the existence of a fundamental pro-object representing a given fibre functor $\Phi\colon\mathcal{C}\to\mathbf{FSet}$; this is a cofiltered projective system $(Y_{\alpha })$ of objects together with a functorial isomorphism
\begin{equation*}
\Phi (T)\cong \varinjlim_{\alpha }\Hom_{\mathcal{C}}(Y_{\alpha },T)
\end{equation*}
for objects $T$ of~$\mathcal{C}$. By passing to a cofinal subsystem we may assume that all $Y_{\alpha }$ are Galois objects, i.e.\ $\Aut (Y_{\alpha })$ operates simply transitively on $\Phi (Y_{\alpha })$. We then obtain identifications $\Aut (Y_{\alpha })\cong\operatorname{Im} (\pi\to\Aut\Phi (Y_{\alpha}))$, and by passage to the limit
\begin{equation*}
\varprojlim_{\alpha }\Aut (Y_{\alpha})\cong\pi.
\end{equation*}
For $\mathcal{C}=\mathbf{FCov}(X)$ and $\Phi =\Phi_x$ a fundamental pro-object is a cofiltered projective system of connected finite coverings of $X$ such that every connected finite covering is dominated by one of them. It serves as a replacement for a universal covering space of $X$, which may not even exist as a topological space. The Galois objects in $\mathbf{FCov}(X)$ are precisely the normal finite connected coverings, and so we obtain:
\begin{Proposition}\label{Prop:ExistenceOfAFundamentalProObjectInFCovX}
	Let $X$ be a connected topological space and $x\in X$. Then there exists a cofiltered projective system $(p_{\alpha}\colon Y_{\alpha }\to X)$ of finite connected normal coverings of $X$ such that every finite connected covering of $X$ is dominated by some $Y_{\alpha }$, together with an isomorphism of functors
	\begin{equation*}
	\Phi_x\cong\varinjlim_{\alpha }\Hom_X(Y_{\alpha },-).
	\end{equation*}
	For such a system there is a canonical isomorphism of profinite groups
	\begin{equation*}
	\pioneet (X,x)\cong\varprojlim_{\alpha }\Aut_X(Y_{\alpha }).
	\end{equation*}
\end{Proposition}
Given $x\in X$ there is a simple natural construction of this fundamental pro-object. We define a \emph{pointed finite covering space} of $(X,x)$ to be a continuous map of pointed spaces $p\colon (Y,y)\to (X,x)$ such that $p\colon Y\to X$ is an object of $\mathbf{FCov}(X)$; to put it another way, this is an object $Y$ of $\mathbf{FCov}(X)$ together with an element of $\Phi_x(Y)$. A morphism of pointed finite covering spaces, say from $p_1\colon (Y_1,y_1)\to (X,x)$ to $p_2\colon (Y_2,y_2)\to (X,x)$ is a continuous map of pointed spaces $(Y_1,y_1)\to (Y_2,y_2)$ that makes the obvious diagram commute. Then for each two given pointed finite covering spaces there is \emph{at most one} morphism from the first to the second. In particular, if two pointed finite covering spaces are isomorphic, the isomorphism is unique.

It is easily seen that the isomorphism classes of pointed finite covering spaces of $(X,x)$ form a set $P=P(X,x)$; it becomes a directed set when we define $(Y_1,y_1)\ge (Y_2,y_2)$ to mean that there exists a (necessarily unique) morphism of pointed direct covering spaces $(Y_1,y_1)\to (Y_2,y_2)$. We then define the \emph{universal profinite covering space} of $(X,x)$ to be the pair
\begin{equation*}
(\tilde{X},\tilde{x})=\varprojlim_{(Y,y)\in P}(Y,y).
\end{equation*}
This is a pointed topological space coming with a continuous map $p\colon (\tilde{X},\tilde{x})\to (X,x)$, and (by Proposition~\ref{Prop:ExistenceOfAFundamentalProObjectInFCovX}) also with a continuous action by $\pioneet (X,x)$ which preserves~$p$; moreover, $p$ is the quotient map for this action. The fibre $p^{-1}(x)\subseteq\tilde{X}$ is a principal homogeneous space for $\pioneet (X,x)$, and the point $\tilde{x}\in p^{-1}(x)$ defines a canonical trivialisation. 

We also note the following for later use:
\begin{Proposition}\label{Prop:UHLPforProfiniteCovering}
	Let $(X,x)$ be a pointed connected topological space, and let $p\colon (\tilde{X},\tilde{x})\to (X,x)$ be its universal profinite covering space. Then $p\colon \tilde{X}\to X$ satisfies the unique homotopy lifting property.
\end{Proposition}
\begin{proof}
	This follows formally from Proposition~\ref{Prop:UHLP} and the universal property of projective limits.
\end{proof}

\subsubsection*{Equivariant \'etale fundamental groups} Let $X$ be a connected topological space and let $\varGamma$ be a group acting on $X$ from the left by homeomorphisms. Then we define a \emph{$\varGamma$-equivariant finite covering space} of $X$ as a finite covering $p\colon Y\to X$ together with a lift of the $\varGamma$-action to $Y$, i.e.\ an action of $\varGamma$ by homeomorphisms on $Y$ such that $p$ becomes $\varGamma$-equivariant. If $p_i\colon Y_i\to X$ are $\varGamma$-equivariant finite covering spaces for $i=1,2$ then a morphism from $Y_1$ to $Y_2$ is a continuous $\varGamma$-equivariant map $f\colon Y_1\to Y_2$ such that $p_2\circ f=p_1$. We obtain a category $\mathbf{FCov}_{\varGamma }(X)$ of $\varGamma$-equivariant finite covering spaces of $X$. Essentially repeating the proof of Proposition~\ref{Prop:FCovIsGaloisCategory} we see that $\mathbf{FCov}_{\varGamma }(X)$ is a Galois category, and for every $x\in X$ the functor $\Phi_x\colon\mathbf{FCov}_{\varGamma }(X)\to\mathbf{FSet}$ with $\Phi_x(p\colon Y\to X)=p^{-1}(Y)$ is a fibre functor.
\begin{Definition}
	Let $X$ be a connected topological space endowed with a left action of a group $\varGamma$ by homeomorphisms, and let $x\in X$. The automorphism group of the fibre functor $\Phi_x\colon\mathbf{FCov}_{\varGamma }(X)\to\mathbf{FSet}$ is called the \emph{$\varGamma$-equivariant \'etale fundamental group} of $X$ at $x$ and denoted by $\pioneet ([\varGamma\backslash X],x)$.
\end{Definition}
The notation is purely symbolic at this point, though it is possible to define a stack $[\varGamma\backslash X]$ on a suitable site and extend the theory of \'etale fundamental groups to such stacks. For our purposes, however, the definition of $\pioneet ([\varGamma\backslash X],x)$ given above will suffice.

There is a forgetful functor\label{page:SesFCovForStackyQuotient} $F\colon\mathbf{FCov}_{\varGamma }(X)\to\mathbf{FCov}(X)$ and also a functor $I\colon\GammaFSet\to\mathbf{FCov}_{\varGamma }(X)$ which is some sort of induction: it sends a finite $\varGamma$-set $S$ to the topologically trivial covering $X\times S\to X$ with the diagonal $\varGamma$-action on $X\times S$ (note that as soon as the $\varGamma$-action on $S$ is nontrivial this is nontrivial as an object of $\mathbf{FCov}_{\varGamma }(X)$). These two exact functors induce homomorphisms of fundamental groups $\pioneet (X,x)\to\pioneet ([\varGamma\backslash X],x)\to\hat{\varGamma }$, where $\hat{\varGamma}$ is the profinite completion of $\varGamma$, which is canonically isomorphic to the fundamental group of $\GammaFSet$ at the forgetful fibre functor $\GammaFSet\to\mathbf{FSet}$.
\begin{Proposition}\label{Prop:SESStackyFibration}
	Let $X$ be a connected topological space endowed with a left action of an abstract group $\varGamma$ by homeomorphisms, and let $x\in X$. Then the sequence
	\begin{equation}\label{eqn:SesForEquivEtaleFG}
	\pioneet (X,x)\overset{F^{\ast }}{\to }\pioneet ([\varGamma \backslash X],x)\overset{I^{\ast }}{\to }\hat{\varGamma}\to 1
	\end{equation}
	is exact.
\end{Proposition}
\begin{proof}
	The functor $I\colon\GammaFSet\to\mathbf{FCov}_{\varGamma }(X)$ is fully faithful, hence it induces a surjection on fundamental groups by \cite[Expos\'e~V, Proposition~6.9]{MR2017446}. The functor $F\circ I$ sends every object of $\GammaFSet$ to a completely decomposed object in $\mathbf{FCov}(X)$; by \cite[Expos\'e~V, Corollaire~6.5]{MR2017446} this implies that $I^{\ast}\circ F^{\ast }=(F\circ I)^{\ast }=1$, or $\operatorname{im} F^{\ast }\subseteq\ker I^{\ast }$. For the reverse inclusion $\ker I^{\ast }\subseteq\operatorname{im}F^{\ast }$ we apply the criterion given in~\cite[Expos\'e~V, Proposition~6.11]{MR2017446}\footnote{Note, however, the misprint there: the two inclusions $\ker u\subset\operatorname{im}u'$ and $\ker u\supset\operatorname{im}u'$ must be exchanged.}. Using said criterion we can reduce this inclusion to the following claim: if $Y\to X$ is a connected object in $\mathbf{FCov}_{\varGamma }(X)$ whose image under $F$ admits a section (i.e.\ such that there is a continuous but not necessarily $\varGamma$-invariant section of $Y\to X$) then $Y$ is in the essential image of~$I$. Indeed, the tautological map $Y\to X\times\pi_0(Y)$ is then an isomorphism in $\mathbf{FCov}_{\varGamma }(X)$.
\end{proof}
\begin{Remark}\label{Rmk:RightActionSignInSes}
	\begin{altenumerate}
		\item The homomorphism $\pioneet (X,x)\to\pioneet ([\varGamma\backslash X],x)$ need not be injective. As a counterexample we may take $X=\bbs^1$ and $\varGamma =\mu_{\infty }$ acting by translations. Then for any $n>1$ the homeomorphism of $\bbs^1$ given by a primitive $n$-th root of unity does not lift to a homeomorphism of the same order along the degree $n$ covering $\bbs^1\to\bbs^1$, which shows that a finite covering of $\bbs^1$ admitting a lift of the $\mu_{\infty }$-action must already be trivial. Hence the map $\hat{\bbz }\cong\pioneet (\bbs^1)\to\pioneet ([\mu_{\infty }\backslash\bbs^1])$ is trivial (therefore not injective), and consequently $\pioneet ([\mu_{\infty }\backslash\bbs^1])\cong\hat{\mu}_{\infty }=1$.
		\item We will need to apply these constructions in the case where $\varGamma$ operates from the right on~$X$. This can be translated to an action from the left by setting $\gamma x=x\gamma^{-1}$ and $\pioneet ([X/\varGamma],x)=\pioneet ([\varGamma\backslash X],x)$. This way we still obtain an exact sequence
		\begin{equation*}
		\pioneet (X,x)\to\pioneet ([X/\varGamma ],x)\to\hat{\varGamma }\to 1,
		\end{equation*}
		but we have to bear in mind that the construction of the second map involves the inversion map $\gamma\mapsto\gamma^{-1}$.
	\end{altenumerate}
\end{Remark}

\subsection{Comparison between classical and \'etale fundamental groups} Let $X$ be a connected topological space and $x\in X$. Then there is a canonical homomorphism
\begin{equation}\label{eqn:DefTau}
\alpha\colon\pionepath (X,x)\to\pioneet (X,x)=\Aut (\Phi_x)
\end{equation}
constructed as follows. For $[\gamma]\in\pionepath (X,x)$ represented by a loop $\gamma\colon [0,1]\to X$ and a finite covering $p\colon Y\to X$ we let $\tau ([\gamma])$ operate on $\Phi_x(Y)=p^{-1}(x)$ by sending $y\in p^{-1}(x)$ to the end point $\tilde{\gamma}(1)$ of the unique continuous lift $\tilde{\gamma }\colon [0,1]\to Y$ of $\gamma$ with starting point $\tilde{\gamma }(0)=y$.
\begin{Proposition}
	Let $(X,x)$ be a connected pointed topological space. Then $\alpha$ as in (\ref{eqn:DefTau}) is continuous if $\pioneet (X,x)$ is endowed with its profinite topology, and $\pionepath (X,x)$ is endowed with either of the loop, $\tau$- and $\sigma$-topologies. It also extends uniquely to a continuous group homomorphism $\pi_1^{\Gal }(X,x)\to\pioneet (X,x)$.
\end{Proposition}
\begin{proof}
	We first show that $\alpha$ is continuous for the loop topology on $\pionepath (X,x)$. Since the open subgroups  of $\pioneet (X,x)$ form a basis of open neighbourhoods of the identity and since the loop topology turns $\pionepath (X,x)$ into a quasi-topological group, it suffices to show that the preimage of any open subgroup of $\pioneet (X,x)$ under $\alpha$ is open in $\pionepath (X,x)$.
	
	For such an open subgroup there is a pointed connected finite covering $p\colon (Y,y)\to (X,x)$ such that the subgroup is the image of $p_{\ast }\colon\pioneet (Y,y)\to\pioneet (X,x)$. There is a commutative diagram of continuous maps
	\begin{equation*}
	\xymatrix{
		\Omega (Y,y)\ar[r]\ar[d]_-{\Omega (p)} & \pionepath (Y,y)\ar[r]^-{\alpha }\ar[d]_-{p_{\ast }} & \pioneet (Y,y)\ar[d]^-{p_{\ast }}\\
		\Omega (X,x)\ar[r] &\pionepath (X,x)\ar[r]_-{\alpha }&\pioneet (X,x).
		}
	\end{equation*}
	By the unique homotopy lifting property for $p$ the map $\Omega (p)$ defines a homeomorphism from $\Omega (Y,y)$ to an open and closed subset of $\Omega (X,x)$; in particular it is open as a map $\Omega (Y,y)\to\Omega (X,x)$. Hence $p_{\ast }\colon\pionepath (Y,y)\to\pionepath (X,x)$ is also open, and its image is equal to the preimage in $\pionepath (X,x)$ of $p_{\ast }(\pioneet (Y,y))\subseteq\pioneet (X,x)$. Hence this preimage is open. Therefore $\alpha$ is continuous for the loop topology.
	
	One way to construct the $\tau$-topology from the loop topology is explained in \cite[Section~7]{Klevdal2015}: the forgetful functor from topological groups to quasi-topological groups has a left adjoint $\tau$ which preserves the underlying groups. Hence for any topological group $G$ and any quasi-topological group $\pi$ a group homomorphism $\pi\to G$ is continuous if and only if $\tau (\pi)\to G$ is continuous. Now $\tau$ applied to $\pionepath (X,x)$ with the loop topology yields $\pionepath (X,x)$ with the $\tau$-topology, and therefore $\alpha$ remains continuous when $\pionepath (X,x)$ is endowed with the $\tau$-topology.
	
	For the $\sigma$-topology we use again the fact that the open subgroups of $\pioneet (X,x)$ form a basis of open neighbourhoods of the identity. Since the preimage of each of these under $\alpha$ is an open subgroup for the $\tau$-topology, it is also an open subgroup for the $\sigma$-topology. Hence $\alpha$ is also continuous for the $\sigma$-topology.
	
	Finally, because $\pioneet (X,x)$ is complete, $\alpha$ extends to a continuous homomorphism $\pi_1^{\Gal }(X,x)\to\pioneet (X,x)$.
\end{proof}
\begin{Proposition}
	Let $(X,x)$ be a pointed connected topological space, and let $p\colon (\tilde{X},\tilde{x})\to (X,x)$ be its universal profinite covering space. Then the sequence of groups
	\begin{equation}\label{eqn:SesThreeFundamentalGroups}
	1\to\pionepath (\tilde{X},\tilde{x})\overset{p_{\ast}}{\to }\pionepath (X,x)\overset{\alpha }{\to}\pioneet (X,x)
	\end{equation}
	is exact.
\end{Proposition}
\begin{proof}
	We first show that $p_{\ast }$ is injective. Let $\tilde{\gamma }\colon [0,1]\to\tilde{X}$ be a loop based at $\tilde{x}$ such that $\gamma =p\circ\tilde{\gamma }$ is nullhomotopic, say by a homotopy $H\colon [0,1]\times [0,1]\to X$ with $H(0,t)=H(1,t)=x$ and $H(t,0)=\gamma (t)$ for all $t\in [0,1]$. Then by Proposition~\ref{Prop:UHLPforProfiniteCovering} $H$ lifts to a homotopy of paths $\tilde{H}\colon [0,1]\times [0,1]\to\tilde{X}$ with $\tilde{H}(0,t)=\tilde{x}$ and $\tilde{H}(t,0)=\tilde{\gamma }(t)$ for all $t\in [0,1]$. By construction, $\tilde{H}(1,t)\in p^{-1}(x)$ for all $t\in [0,1]$, and $\tilde{H}(1,0)=\tilde{x}$. Since $p^{-1}(x)$ is totally disconnected, $\tilde{H}(1,t)$ must be equal to $\tilde{x}$ for all $t\in [0,1]$. Hence $\tilde{H}$ really defines a homotopy of loops, and not just of paths, from $\tilde{\gamma }$ to the constant loop. Therefore the class of $\tilde{\gamma }$ in $\pionepath (\tilde{X},\tilde{x})$ is trivial. This shows the injectivity of~$p_{\ast }$.
	
	For exactness at $\pionepath (X,x)$, let $\gamma \colon [0,1]\to X$ be a loop based at~$x$. Then by Proposition~\ref{Prop:UHLPforProfiniteCovering} there exists a unique lift $\tilde{\gamma }\colon [0,1]\to\tilde{X}$ which is continuous and satisfies $\tilde{\gamma }(0)=\tilde{x}$. The end point $\tilde{x}'=\tilde{\gamma }(1)$ is another element of the fibre $p^{-1}(x)$, not necessarily equal to~$\tilde{x}$. Recall that $p^{-1}(x)$ is a principal homogeneous space for $\pioneet (X,x)$, hence there is a unique element of $\pioneet (X,x)$ that sends $\tilde{x}$ to~$\tilde{x}'$. Unravelling of definitions shows that this element is equal to~$\alpha ([\gamma ])$. Hence we find that the following conditions are equivalent:
	\begin{enumerate}
		\item $[\gamma ]\in p_{\ast }(\pionepath (\tilde{X},\tilde{x}))$;
		\item $\tilde{\gamma }(1)=\tilde{\gamma }(0)$;
		\item $\tilde{x}=\tilde{x}'$;
		\item $\alpha ([\gamma ])=1$.
	\end{enumerate}
	The equivalence of (i) and (iv) then shows exactness at $\pionepath (X,x)$.
\end{proof}
We can also characterise the image of~$\alpha$. Since $\pioneet (X,x)$ acts continuously on $\tilde{X}$ it permutes the path-components of that space.
\begin{Proposition}\label{Prop:ImageOfDiscreteAlpha}
	Let $(X,x)$ be a pointed connected space and let $p\colon (\tilde{X},\tilde{x})\to (X,x)$ be its universal profinite covering space. Let $\tilde{X}^{\circ}$ be the path-component of $\tilde{X}$ containing~$\tilde{x}$.
	
	Then the image of $\alpha\colon\pionepath (X,x)\to\pioneet (X,x)$ is the stabiliser of $\tilde{X}^{\circ}$ in $\pioneet (X,x)$.
\end{Proposition}
\begin{proof}
	Let $[\gamma ]\in\pionepath (X,x)$ be represented by a loop $\gamma\colon [0,1]\to X$ based at~$x$, and let $\tilde{x}'\in p^{-1}(x)\subseteq\tilde{X}$. Let $\tilde{\gamma }\colon [0,1]\to\tilde{X}$ be the unique continuous lift of $\gamma$ with $\tilde{\gamma }(0)=\tilde{x}'$. Then $\alpha ([\gamma ])(\tilde{x}')=\tilde{\gamma }(1)$, hence $\tilde{x}'$ and its image under $\alpha ([\gamma ])$ lie in the same path component of~$\tilde{X}$. Since $\tilde{x}'\in p^{-1}(x)$ was arbitrary this shows that $\alpha ([\gamma ])$ preserves all path components of $\tilde{X}$ meeting $p^{-1}(x)$, in particular~$\tilde{X}^{\circ}$.
	
	For the other inclusion let $\beta\in\pioneet (X,x)$ be an element preserving $\tilde{X}^{\circ}$. Then $\beta (\tilde{x})\in \tilde{X}^{\circ}$, hence there is a path $\tilde{\gamma }$ in $\tilde{X}$ from $\tilde{x}$ to $\beta (\tilde{x})$. Then $\gamma =p\circ\tilde{\gamma }$ is a closed loop in $X$ based at $x$ and therefore represents an element of $\pionepath (X,x)$. By construction, both $\alpha ([\gamma ])\in\pioneet (X,x)$ and $\beta$ send $\tilde{x}$ to $\beta (\tilde{x})$, but since $\pioneet (X,x)$ acts freely on $\tilde{X}$, they must be equal. Hence $\beta =\alpha ([\gamma ])$ is in the image of~$\alpha$.
\end{proof}
Hence we can rewrite (\ref{eqn:SesThreeFundamentalGroups}) in a more precise way: the sequence
\begin{equation}\label{eqn:SesThreeFGAlsoOnRight}
1\to\pionepath (\tilde{X},\tilde{x})\overset{p_{\ast}}{\to }\pionepath (X,x)\overset{\alpha}{\to }\operatorname{Stab}_{\pioneet (X,x)}(\tilde{X}^{\circ})\to 1
\end{equation}
is exact.

Let $\pionehat (X,x)$ be the profinite completion of the group $\pionepath (X,x)$. By the universal property of profinite completions $\alpha$ induces a continuous group homomorphism
\begin{equation}\label{eqn:DefAlphaHat}
\hat{\alpha }\colon\pionehat (X,x)\to\pioneet (X,x).
\end{equation}
\begin{Proposition}
	Let $X$ be path-connected, locally path-connected and semi-locally simply connected. Then $\hat{\alpha }$ as in (\ref{eqn:DefAlphaHat}) is an isomorphism of topological groups.
\end{Proposition}
\begin{proof}
	This follows from the classical theory of covering spaces: under the given circumstances, $(X,x)$ has a universal (possibly infinite) covering space, and discrete sets with an operation by $\pionepath (X,x)$ are equivalent to covering spaces of~$X$. Hence finite covering spaces of $X$ are equivalent to finite sets with $\pionepath (X,x)$-action, which are in turn equivalent to finite sets with continuous $\pioneet (X,x)$-action.
\end{proof}
For general path-connected spaces $\hat{\alpha }$ need not be an isomorphism; however, there is a weaker technical condition which still ensures that $\hat{\alpha }$ is surjective.
\begin{Definition}
	Let $X$ be a topological space. We say that $X$ is \emph{stably path-connected} if $X$ is path-connected, and for every finite covering $Y\to X$ with $Y$ connected, $Y$ is already path-connected.
\end{Definition}
\begin{Examples}\label{epl:WarsawCircle}
	\begin{altenumerate}
		\item If a topological space $X$ is connected and locally path-connected then it is also stably path-connected. To see this, note that any finite covering space of $X$ is then also locally path-connected, and a space which is connected and locally path-connected is also globally path-connected, cf.\ \cite[Theorem~25.5]{Munkres2000}.
		\item Let $\omega_1$ be the first uncountable ordinal, and let $L=\omega_1\times [0,1)$ be the \emph{long line}, equipped with the order topology (see \cite[Example~45]{MR507446}). Then $L$ is Hausdorff and locally homeomorphic to $\bbr$ but not paracompact. Every two points in $L$ are contained in an open subset homeomorphic to $\bbr$, hence $L$ is path-connected and locally path-connected. The one-point compactification $L^{\ast }=L\cup \{\ast \}$ is no longer path-connected.
		
		We then define the \emph{long circle} to be $C=L^{\ast }/(\ast\sim (0,0))$, see Figure~\ref{Fig:WarsawCircle}. The dotted part is so long on the left that no path can enter it from the left, but every point in it can be reached by a path entering from the right, which shows that $C$ is path-connected. $C$ is not locally path-connected at $\ast$, but everywhere else.
		
		For every $n\in\bbn$ the long circle admits a connected covering of degree $n$ with $n$ path components, hence it is not stably path-connected. See Figure~\ref{Fig:CoverWarsawCircle} for the case $n=3$; the different colours encode the path components.
		\item Similar remarks apply to the \emph{Warsaw circle}. Consider the  truncated topologist's sine curve
		\begin{equation*}
		S'=\{ (0,y)\mid -1\le y\le 1 \}\cup\left\{ \left( x,\sin\frac 1x\right)\,\middle\lvert\, 0<x\le \frac{1}{\pi } \right\} ;
		\end{equation*}
		the Warsaw circle is defined as the quotient $W=S'/{((0,0)\sim (0,\frac{1}{\pi }))}$, cf.\ Figure~\ref{Fig:WarsawCircle}. Like the long circle, $W$ is path-connected but not stably path-connected.
	\end{altenumerate}
\end{Examples}
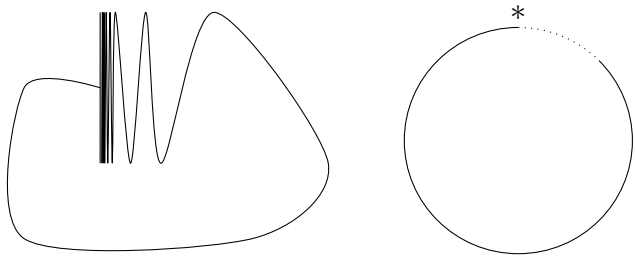
\begin{figure}
	\begin{tikzpicture}
	\draw plot [smooth] coordinates {(0.02,-1) (0.03,1) (0.04,-1) (0.05,1) (0.06,-1) (0.08,1) (0.1,-1) (0.13,1) (0.16,-1) (0.2,1) (0.4,-1) (0.6,1) (0.8,-1) (1.5,1) (3,-1) (2,-2) (-1,-2) (-1,0) (0,0)};
	\draw (0,1) -- (0,-1);
	\draw (5.5,0.8) arc (90:405:1.5cm);
	\draw[dotted] (5.5,0.8) arc (90:45:1.5cm);
	\draw (5.5,1) node {$\ast$};
	\end{tikzpicture}
	\caption{The Warsaw circle (left) and the long circle (right)}\label{Fig:WarsawCircle}
\end{figure}
\begin{figure}
	\begin{tikzpicture}
	\draw (3.5,1.2) arc (90:250:1.5cm) .. controls (3.45,-1.8) and (3.3,-1.4) .. (3.5,-1.4) arc (270:405:1.1cm);
	\draw[blue] (3.5,1) arc (90:240:1.3cm) .. controls (3.3,-1.6) and (3.3,-1.8) .. (3.5,-1.8) arc (270:405:1.5cm);
	\draw[red] (3.5,0.8) arc (90:240:1.1cm) .. controls (3.3,-1.4) and (3.3,-1.6) .. (3.5,-1.6) arc (270:405:1.3cm);
	\draw[dotted] (3.5,0.8) arc (90:45:1.1cm);
	\draw[blue,dotted] (3.5,1.2) arc (90:45:1.5cm);
	\draw[red,dotted] (3.5,1) arc (90:45:1.3cm);
	\end{tikzpicture}
	\caption{A connected but not path-connected covering of the long circle}\label{Fig:CoverWarsawCircle}
\end{figure}
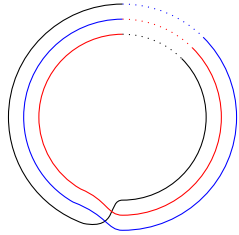
\begin{Proposition}\label{Prop:StablyPathConnectedThenAlphaSurjective}
	Let $X$ be a path-connected topological space and let $x\in X$. Then the following are equivalent.
	\begin{enumerate}
		\item $X$ is stably path-connected.
			\item For every finite Galois covering $p\colon Y\to X$ the map $\alpha_Y\colon\pionepath (X,x)\to\Aut (Y/X)$ is surjective.
		\item The map $\alpha\colon\pionepath (X,x)\to\pioneet (X,x)$ has dense image.
		\item The map $\hat{\alpha }\colon\hat{\pi}_1^{\mathrm{path}}(X,x)\to\pioneet (X,x)$ is surjective.
	\end{enumerate}
\end{Proposition}
\begin{proof}
	The equivalence of (ii), (iii) and (iv) is easily seen.
	
	We now show that (i) implies (ii). Assume that $X$ is stably path-connected, let $\beta\in\Aut (Y/X)$, and fix some $y\in p^{-1}(Y)$. By assumption $Y$ is path-connected, hence there exists a path $\tilde{\gamma }\colon [0,1]\to Y$ from $y$ to~$\beta (y)$. Then $\gamma =p\circ\tilde{\gamma }$ is a closed loop in $X$, and as in the proof of Proposition~\ref{Prop:ImageOfDiscreteAlpha} we find that $\alpha_Y([\gamma ])=\beta$. Hence (iv) holds.
	
	Finally we assume (ii) and show that it implies (i). Since every connected finite covering of $X$ is dominated by a finite Galois covering it suffices to show that all finite Galois covering spaces are path-connected. So, let $p\colon Y\to X$ be a finite Galois covering. First we note that every $y\in Y$ is in the same path-component as some element of the fibre $p^{-1}(x)$ (take a path from $p(y)$ to $x$ and lift it). It then suffices to show that any two points $y_1,y_2$ in $p^{-1}(x)$ can be linked by a path in~$Y$. Since $\Aut (Y/X)$ operates transitively on $p^{-1}$ and since $\alpha_Y$ is surjective, there is some loop $\gamma$ in $(X,x)$ such that $\alpha ([\gamma ])$ sends $y_1$ to $y_2$. There is then a lift of $\gamma$ to a path in $Y$ starting at $y_1$, and by our choice of $\gamma$ this path must end at~$y_2$. This shows~(iv).
\end{proof}
\begin{Remark}
	By Proposition~\ref{Prop:StablyPathConnectedThenAlphaSurjective} the image of $\alpha$ cannot be dense for our examples of path-connected spaces which are not stably path-connected. Indeed, the long circle $C$ from Example~\ref{epl:WarsawCircle}.(ii) is path-connected, and $C^{\lpc}$ is homeomorphic to the long line, hence $\pionepath (C)\cong\pionepath (C^{\lpc })$ is trivial. However, from the finite connected coverings of $C$ mentioned in Example~\ref{epl:WarsawCircle}.(ii) we see that $\pioneet (C)\cong\hat{\bbz }$. Similarly, the Warsaw circle $W$ from Example~\ref{epl:WarsawCircle}.(iii) has trivial classical fundamental group, and there exists a surjection $\pioneet (W)\to\hat{\bbz }$.
\end{Remark}

\subsection{Etale fundamental groups of schemes} We assume the classical theory of \'etale fundamental groups for schemes, as exposed in \cite{MR2017446}, to be known to the reader. Briefly, for a connected scheme $\mathcal{X}$ the category $\mathbf{FEt}(\mathcal{X})$ of \'etale coverings of $\mathcal{X}$ (i.e.\ schemes $\mathcal{Y}$ together with a finite \'etale morphism $\mathcal{Y}\to\mathcal{X}$) is a Galois category, and for every geometric point $\widebar{x}\colon\Spec\varOmega\to\mathcal{X}$ the functor
\begin{equation*}
\Phi_{\widebar{x}}\colon\mathbf{FEt}(\mathcal{X})\to\mathbf{FSet}
\end{equation*}
is a fibre functor, and the corresponding fundamental group $\Aut (\Phi_{\widebar{x}})$ is called the \emph{\'etale fundamental group} of $\mathcal{X}$ at $\widebar{x}$ and denoted by $\pioneet (\mathcal{X},\widebar{x})$.

If $\mathcal{X}$ is a connected scheme of finite type over $\bbc$, there is a canonical topology called the \emph{complex topology} on $\mathcal{X}(\bbc )$, turning it into a connected topological space. For an \'etale covering $\mathcal{Y}\to\mathcal{X}$ the map $\mathcal{Y}(\bbc )\to\mathcal{X}(\bbc )$ is then a finite covering. Hence we obtain a functor $\mathbf{FEt}(\mathcal{X})\to\mathbf{FCov}(\mathcal{X}(\bbc ))$ which is an equivalence of categories by \emph{Riemann's Existence Theorem}, cf.\ \cite[Expos\'e~XII, Th\'eor\`eme~5.1]{MR2017446}. In particular we obtain an isomorphism of profinite groups
\begin{equation*}
\pioneet (\mathcal{X},\widebar{x})\cong\pioneet (\mathcal{X}(\bbc ),\widebar{x}).
\end{equation*}
 See section~\ref{section:RelationXFCalXF} for a partial extension of these observations to schemes of infinite type over~$\bbc$.

For a field $F$, the \'etale covers of $\Spec F$ are of the form $\Spec E$ where $E$ is an \'etale $F$-algebra, i.e.\ a finite product of finite separable field extensions of~$F$. Consequently, a universal profinite covering is given by $\Spec\Fbar\to\Spec F$ where $\Fbar$ is a separable closure of~$F$. Note that the morphism $\widebar{x}\colon\Spec\Fbar\to\Spec F$ is also a geometric point of $\Spec F$, and we obtain an isomorphism of profinite groups
\begin{equation*}
\pioneet (\Spec F,\widebar{x})\cong\Gal (\Fbar /F),
\end{equation*}
where $\Gal (\Fbar /F)$ is endowed with the Krull topology.

We note a technical result on \'etale fundamental groups of schemes which is analogous to one for compact Hausdorff spaces mentioned before.

\begin{Lemma}\label{Lem:SimplyConnectedProEtaleIsUniversal}
	Let $\mathcal{X}$ be a connected quasi-compact quasi-separated scheme, and let $(p_{\alpha }\colon \mathcal{Y}_{\alpha }\to \mathcal{X})_{\alpha }$ be a cofiltered projective system of connected finite \'etale coverings such that $\tilde{\mathcal{Y}}=\varprojlim_{\alpha}\mathcal{Y}_{\alpha }$ is simply connected (in the sense that $\pioneet (\tilde{\mathcal{Y}})$ is trivial).
	
	Then every connected finite \'etale covering space of $\mathcal{X}$ is dominated by some~$\mathcal{Y}_{\alpha }$.
\end{Lemma}
\begin{proof}
The proof of Lemma~\ref{Lem:SimplyConnectedProEtaleIsUniversal} is strictly parallel to that of Lemma~\ref{Lem:SimplyConnectedProfiniteCovIsUniversal}.
\end{proof}

\section{Topological invariants of Pontryagin duals}

\noindent We begin by briefly summarising the basic results about Pontryagin duals; for a systematic introduction see~\cite{MR0152834}.

For a commutative, locally compact topological group $M$ let $M^{\vee }$ be its \emph{Pontryagin dual}, i.e.\ the set of continuous group homomorphisms $M\to\bbs^1$, endowed with the compact open topology. By Pontryagin duality, this is again a commutative, locally compact topological group, and the tautological map $M\to (M^{\vee })^{\vee }$ is an isomorphism.

The Pontryagin dual $M^{\vee }$ is compact if and only if $M$ is discrete. Moreover, $M^{\vee }$ is connected if and only if $M$ is torsion-free, and $M^{\vee }$ is totally disconnected if and only if $M$ is a torsion group.

\subsection{Pontryagin duals of discrete abelian groups}

For discrete abelian $M$ there is a short exact sequence
\begin{equation*}
0\to M_{\mathrm{tors}}\to M\to M_{\mathrm{tf}}\to 0,
\end{equation*}
where $M_{\mathrm{tors}}$ is the torsion subgroup of $M$, and $M_{\mathrm{tf}}$ is the maximal torsion-free quotient of~$M$. By duality we obtain a short exact sequence of compact topological groups
\begin{equation*}
0\to M_{\mathrm{tf}}^{\vee}\to M^{\vee }\to M_{\mathrm{tors}}^{\vee }\to 0
\end{equation*}
with $M_{\mathrm{tf}}^{\vee }$ connected and $M_{\mathrm{tors}}^{\vee}$ totally disconnected. In particular, the connected components of $M^{\vee }$ are precisely the translates of the subgroup $M_{\mathrm{tf}}^{\vee }$. Note that $M_{\mathrm{tf}}$ can be written as a filtered inductive limit over free abelian groups of finite rank, hence $M_{\mathrm{tf}}^{\vee }$ can be written as a cofiltered projective limit over finite-dimensional tori. We then also see that the canonical homomorphism $\pi_0(M^{\vee })\to M_{\mathrm{tors}}^{\vee }$ is an isomorphism of topological groups, in particular $\pi_0(M^{\vee })$ is compact and totally disconnected, therefore profinite.

The determination of the path components of $M^{\vee }$ is a bit more involved. Consider the short exact sequence
\begin{equation*}
0\to \bbz \to \bbr \to\bbs^1\to 0,
\end{equation*}
where $\bbr\to\bbs^1$ is the map $t\mapsto\mathrm{e}^{2\pi\mathrm{i}t}$, and apply the left-exact functor $\Hom (M,-)$ to it. This yields an exact sequence
\begin{equation*}
\begin{split}
0&\to\Hom (M,\bbz )\to\Hom (M,\bbr )\to \Hom (M,\bbs^1)\\
&\overset{\delta}{\to}\Ext (M,\bbz )\to\Ext (M,\bbr )\to\dotsc ;
\end{split}
\end{equation*}
here $\Hom (M,\bbs^1)=M^{\vee }$ (because $M$ is discrete), and since $\bbr$ is a $\bbq$-vector space, it is injective as an abelian group, hence $\Ext (M,\bbr )=0$. That is, the interesting part of our sequence can be rewritten as
\begin{equation}\label{eqn:LESForMAndRealExponentialSequence}
\Hom (M,\bbr )\to M^{\vee} \overset{\delta}{\to }\Ext (M,\bbz )\to 0.
\end{equation}
Recall that the connecting homomorphism $\delta$ can be described in a more explicit way: if $\chi\in M^{\vee}=\Hom (M,\bbs^1)$ then $\delta (\chi )\in\Ext (M,\bbz )$ is the class of the extension
\begin{equation*}
0\to\bbz \to E_{\chi }\to M\to 0,
\end{equation*}
where
\begin{equation*}
E_{\chi }= \bbr \times_{\bbs^1,\chi }M= \{ (r,m)\in\bbr \times M\mid\exp r=\chi (M) \} .
\end{equation*}
\begin{Proposition}\label{Prop:PiZeroPathPontryaginDual}
	 Let $M$ be a discrete abelian group. The path component of $M^{\vee }$ containing the trivial element is precisely the image of $\Hom (M,\bbr )$. In particular, the connecting homomorphism $\delta$ in (\ref{eqn:LESForMAndRealExponentialSequence}) induces a group isomorphism
	\begin{equation*}
	\pizeropath (M^{\vee })\overset{\cong }{\to }\Ext (M,\bbz ).
	\end{equation*}
\end{Proposition}
\begin{proof}
	Endow $\Hom (M,\bbr )$ with the compact-open topology. Then it is clearly path-connected (if $f$ is an element of this space, then so is $tf$ for every $t\in [0,1]$, and the assignment $t\mapsto tf$ is continuous). Furthermore, the map $\Hom (M,\bbr )\to M^{\vee }$ is continuous, therefore its image must be contained in the trivial path component of~$M^{\vee }$.
	
	For the other inclusion, consider a path beginning at the trivial element of~$M^{\vee }$. Such a path is given by a family $(\chi_t)_{t\in [0,1]}$ of group homomorphisms $\chi_t\colon M\to\bbs^1$ such that for every $m\in M$, $\chi_0(m)=1$ and the map $\gamma_m\colon [0,1]\to\bbs^1$, $t\mapsto\chi_t(m)$, is continuous. We need to show that then $\chi_t$ is in the image of $\Hom (M,\bbr )$, for every $t\in [0,1]$.
	
	Since $\bbr$ is the universal covering space of $\bbs^1$, there is a unique continuous lift $\tilde{\gamma}_m\colon [0,1]\to\bbr$ with $\tilde{\gamma }_m(0)=0$. Then for each $t\in [0,1]$ we define a map
	\begin{equation*}
	\tilde{\chi }_t\colon M\to\bbr ,\quad m\mapsto\tilde{\gamma }_m(t).
	\end{equation*}
	By construction, the composition
	\begin{equation*}
	M\overset{\tilde{\chi}_t}{\to}\bbr \overset{\exp}{\to }\bbs^1
	\end{equation*}
	is equal to $\chi_t$; it remains to be shown that $\tilde{\chi}_t$ is indeed a group homomorphism.
	
	For each $m,n\in M$ consider the map
	\begin{equation*}
	f_{m,n}\colon [0,1]\to\bbr ,\quad t\mapsto\tilde{\chi}_t(m)+\tilde{\chi}_t(n)-\tilde{\chi}_t(m+n) =\tilde{\gamma }_m(t)+\tilde{\gamma }_n(t)-\tilde{\gamma}_{m+n}(t).
	\end{equation*}
	This is a continuous map, as can be seen from the second expression. Also it has image in $\bbz $ because the $\chi_t=\exp\tilde{\chi}_t$ are group homomorphisms. Hence it is constant. Since $f_{m,n}(0)=0$, it has to be zero. That this holds for every $m,n\in M$ precisely means that each $\tilde{\chi }_t$ is a group homomorphism.
\end{proof}
\begin{Remark}
	A \emph{Whitehead group} is an abelian group $M$ with $\Ext (M,\bbz )=0$. By Proposition~\ref{Prop:PiZeroPathPontryaginDual} an abelian group is a Whitehead group if and only if its Pontryagin dual is path-connected. Clearly free abelian groups are Whitehead groups; their Pontryagin duals are products of circle groups. The \emph{Whitehead problem} is the question whether the reverse implication holds, i.e.\ whether every Whitehead group is free abelian or, equivalently, whether every abelian path-connected compact Hausdorff topological group is a product of circles.
	
	Stein \cite{MR0043219} showed that every countable Whitehead group is indeed free abelian; Shelah \cite{MR0357114, MR0389579} showed that the statement `every Whitehead group is free abelian' is independent of ZFC. More precisely, if G\"odel's constructibility axiom $V=L$ holds, then every Whitehead group is free abelian; if $2^{\aleph_0}>\aleph_1$ and Martin's Axiom holds, then there is a Whitehead group of cardinality $\aleph_1$ which is not free abelian.
\end{Remark}
The subspace topology on this path component may be odd (cf.\ Example~\ref{Epl:lpcFoliationsSolenoid}.(iii)), but its minimal locally path-connected refinement admits a simple description, at least in the case relevant to us:
\begin{Proposition}\label{Prop:LPCForPontryaginDuals}
	Let $M$ be a $\bbq$-vector space with discrete topology, and let $(M^{\vee})^{\circ}$ be the neutral path component of its Pontryagin dual. Then there is a natural isomorphism of topological groups
	\begin{equation*}
	\Hom (M,\bbr )\cong ((V^{\vee })^{\circ})^{\lpc},
	\end{equation*}
	where $\Hom (M,\bbr )$ is endowed with the compact-open topology.
\end{Proposition}
\begin{proof}
	From the short exact sequence (\ref{eqn:LESForMAndRealExponentialSequence}) we obtain a map $\Hom (M,\bbr)\to ((M^{\vee })^{\circ })^{\lpc}$. This map is a continuous bijection and an isomorphism of abstract groups. Arguing as in Example~\ref{Epl:lpcFoliationsSolenoid}.(iii) we find that it is in fact a homeomorphism.
\end{proof}
\begin{Corollary}\label{Cor:PontryaginDualHasTrivialPionepath}
	Let $M$ be a $\bbq$-vector space with discrete topology. Then $\pionepath (M^{\vee },1)$ is the trivial group.
\end{Corollary}
\begin{proof}
	This follows from Corollary~\ref{Cor:LPCPreservesManyInvariants} and Proposition~\ref{Prop:LPCForPontryaginDuals}.
\end{proof}
We will now determine the \'etale fundamental groups of Pontryagin duals.
\begin{Proposition}\label{Prop:FiniteCoversOfPontryaginDual}
	Let $M$ be a torsion-free discrete abelian group. Then the connected \'etale coverings of $M^{\vee }$ are precisely the $N^{\vee }$, where $N$ runs through the subgroups of $V=M\otimes\bbq$ that contain $M$ with finite index.
\end{Proposition}
\begin{proof}
	Let us first assume that $M$ is finitely generated; then $M^{\vee }$ is an $n$-dimensional torus, and the statement is well-known.
	
	Now consider the general case. We first show that each $N$ indeed defines a finite covering of $M^{\vee }$. Without loss of generality we may assume that $(N:M)$ is a prime~$p$. Then there exists a finitely generated subgroup $N'\subseteq N$ which is not completely contained in~$M$; the subgroup $M'=M\cap N'$ is then also finitely generated, and $(N':M')$ must be $p$ because it cannot be~$1$. Then $N=N'+M$ and thus
	\begin{equation*}
	N^{\vee }\cong (N')^{\vee}\times_{(M')^{\vee}}M^{\vee}.
	\end{equation*}
	Hence the map $N^{\vee }\to M^{\vee }$ is a base change of the finite covering $(N')^{\vee}\to (M')^{\vee }$, hence itself a finite covering. Since $N$ is also torsion-free, $N^{\vee }$ is connected.
	
	For the other implication, let $Y\to M^{\vee}$ be a finite covering. By Proposition \ref{Prop:FiniteCoveringDefinedOnFiniteLevel} this must be the base change of some finite covering space $Y_0\to (M')^{\vee }$ for a finitely generated subgroup $M'\subseteq M$, and by what we have already shown, $Y_0=(N')^{\vee }$ with $(N':M')$ finite. Then $Y=(N'\oplus_{M'}M)^{\vee }$; since this is assumed to be connected, $N'\oplus_{M'}M$ must be torsion-free, hence embed via the obvious map into $V=M\otimes\bbq$.
\end{proof}
The next corollary follows essentially formally from Proposition~\ref{Prop:FiniteCoversOfPontryaginDual}.
\begin{Corollary}\label{Cor:PioneetPontryaginDual}
	Let $M$ be a torsion-free discrete abelian group. Then $\pioneet (M^{\vee },0)\cong (M\otimes\bbq /M)^{\vee}$ as topological groups.
	
	In particular, if $M$ is a $\bbq$-vector space, then $\pioneet (M^{\vee })$ is trivial.
	\hfill $\square$
\end{Corollary}
Applying this to $M=\bbq$ we find that the solenoid $\mathcal{S}\cong\bbq^{\vee}$ has trivial \'etale fundamental group.

\subsubsection*{Cohomology}

For a topological space $X$ and an abelian group $A$ there are several different ways to define cohomology groups $\Hup^p (X,A)$, and they do not always agree.
\begin{enumerate}
	\item First there is the singular cohomology group $\Hup_{\mathrm{sing}}^p(X,A)$, which is the cohomology of the singular cochain complex $\mathrm{C}^{\bullet }(X,A)$ with
	\begin{equation*}
	\mathrm{C}^p(X,A)=\Hom (\mathrm{C}_p(X),A).
	\end{equation*}
	\item There is another construction using sheaf cohomology: the category $\Sh (X)$ of sheaves of abelian groups on $X$ is an abelian category with enough injectives, and the global sections functor $\Gamma_X\colon\Sh (X)\to\mathbf{Ab}$ sending a sheaf $\mathcal{A}$ to $\Gamma (X,\mathcal{A})$ is left exact. Hence it admits right derived functors, and we set $\Hup^p(X,\mathcal{A})=\Rup^p\Gamma_X(\mathcal{A})$. For each abelian group $A$ there is a constant sheaf $A_X$ modelled on $A$, and we set $\Hup^p(X,A)=\Hup^p(X,A_X)$.
	\item Finally we can also consider \v{C}ech cohomology groups for sheaves:
	\begin{equation*}
	\check{\Hup }^p(X,\mathcal{A})=\varinjlim_{\mathfrak{U}}\check{\Hup }^p(\mathfrak{U},\mathcal{A})
	\end{equation*}
	where the limits goes over ever finer open covers of~$X$, and again we set $\check{\Hup }^p(X,A)=\check{\Hup }^p(X,A_X)$.
\end{enumerate}
If $X$ is a paracompact Hausdorff space, sheaf cohomology in terms of derived functors is always isomorphic to \v{C}ech cohomology by \cite[Th\'eor\`eme~5.10.1]{MR0345092}. However, sheaf cohomology and singular cohomology do not even agree on all compact Hausdorff spaces, see Remark~\ref{Rmk:ExampleForSheafCohomologyNotSingularCohomology} below. We prefer to work with sheaf cohomology because it behaves better with respect to projective limits of spaces: \begin{Proposition}\label{Prop:SheafCohomologyProjectiveLimits}
	Let $A$ be an abelian group and let $(X_j)_{j\in J}$ be a cofiltered projective system of compact Hausdorff spaces and let $X=\varprojlim_jX_j$. Then the canonical map of sheaf cohomology groups
	\begin{equation*}
	\varinjlim_{j\in J}\Hup^p(X_j,A)\to\Hup^p(X,A)
	\end{equation*}
	is an isomorphism.
\end{Proposition}
\begin{proof}
	As already remarked we may identify these groups with \v{C}ech cohomology groups; the corresponding statement for \v{C}ech cohomology groups is \cite[Chapter~X, Theorem~3.1]{MR0050886}. The proof is based on the observation that for a compact space \v{C}ech cohomology may be computed using only finite covers.
\end{proof}
\begin{Proposition}\label{Prop:CohomologyOfPontryaginDuals}
	Let $V$ be a $\bbq$-vector space.
	\begin{enumerate}
		\item If $A$ is an abelian torsion group then $\Hup^0(V^{\vee },A)=A$ and $\Hup^p(V^{\vee },A)=0$ for all $p>0$.
		\item For each $p\ge 0$ there is a canonical isomorphism
		\begin{equation*}
		\Hup^p(V^{\vee },\bbq  )\cong\textstyle{\bigwedge}^p_{\bbq }V.
		\end{equation*}
		Under these isomorphisms wedge and cup products correspond to each other.
		\item We have $\Hup^0(V^{\vee },\bbz )=\bbz$, and for each $p>0$ the inclusion $\bbz\hookrightarrow\bbq$ induces an isomorphism
		\begin{equation*}
		\Hup^p(V^{\vee },\bbz )\cong\Hup^p(V^{\vee },\bbq ).
		\end{equation*}
	\end{enumerate}
\end{Proposition}
\begin{proof}
	\begin{enumerate}
		\item First note that it suffices to prove this statement for $A=\bbz /n\bbz$: On compact Hausdorff spaces, sheaf cohomology commutes with direct sums and filtered limits of sheaves, and every abelian torsion group can be built from the groups $\bbz /n\bbz$ for $n\in\bbn$ using these constructions.
		
		Let $\mathcal{M}$ be the set of all finitely generated free abelian subgroups of $V$, so that $V=\varinjlim_{M\in \mathcal{M}}M$ is a filtered limit. Then by Proposition~\ref{Prop:SheafCohomologyProjectiveLimits} we can write $$\Hup^p(V^{\vee },\bbz /n\bbz )\cong\varinjlim_{M\in \mathcal{M}}\Hup^p(M^{\vee },\bbz /n\bbz).$$
		We need to show that the image of each $c\in\Hup^p(M^{\vee },\bbz /\bbz )$ for $p>0$ is trivial. Consider the subgroup $\frac{1}{n}M\in\mathcal{M}$; a short calculation using that $M^{\vee }$ and $(\frac 1nM)^{\vee }$ are tori of the same dimension shows that the pullback of $c$ to $\Hup^p((\frac 1nM)^{\vee },\bbz /n\bbz)$ is zero (but only for $p>0$).
		\item Again we write $V=\varinjlim_{M\in\mathcal{M}}M$ and use that $\Hup^p(M^{\vee },\bbq )\cong\bigwedge_{\bbq }^p(M\otimes\bbq )$.
		\item This follows from (i) and (ii) using the long exact cohomology sequence induced by the short exact sequence $0\to\bbz\to\bbq\to\bbq /\bbz\to 0$.\qedhere
	\end{enumerate}
\end{proof}
\begin{Remark}\label{Rmk:ExampleForSheafCohomologyNotSingularCohomology}
Note that, by contrast, for any abelian group $A$ the singular cohomology of $V^{\vee }$ with $A$-coefficients can be calculated as
\begin{equation*}
\Hup^0_{\mathrm{sing}}(V^{\vee },A )\cong A^{\Ext (V,\bbz )} ,\quad\Hup^p_{\mathrm{sing }}(V^{\vee },A)=0\text{ for all }p>0
\end{equation*}
by Corollary~\ref{Cor:LPCPreservesManyInvariants}.(iii) and Proposition~\ref{Prop:PiZeroPathPontryaginDual}. So, for every $p\ge 0$ the groups $\Hup^p(V^{\vee },\bbz )$ and $\Hup^p_{\mathrm{sing}}(V^{\vee },\bbz )$ are not isomorphic. The same holds for rational coefficients.
\end{Remark}

\subsection{Spectra of group algebras}

We assemble some simple results on the spectra of group algebras; they will play the role in our scheme-theoretic considerations that is played by Pontryagin duals in the topological case.
\begin{Proposition}\label{Prop:PiZeroOfGroupAlgebra}
	Let $M$ be an abelian group. Then $\Spec \bbc [M]$ is connected if and only if $M$ is torsion-free.
\end{Proposition}
\begin{proof}
	Note that a $\bbc$-scheme $\mathcal{X}$ is connected if and only if every $\bbc$-morphism $\mathcal{X}\to \mathcal{S}=\Spec\bbc\coprod\Spec\bbc$ is constant.
	
	First assume that $M$ is torsion-free. Then $M$ is a filtered limit $\varinjlim_{i\in I}M_i$ where the $M_i$ are free abelian groups of finite rank. Then
	\begin{equation*}
	\Spec\bbc [M]=\varprojlim_{i\in I}\Spec \bbc [M_i],
	\end{equation*}
	so each $\bbc$-morphism $\Spec \bbc [M]\to \mathcal{S}$ factors through some $\Spec \bbc [M_i]$. But $M_i\simeq\bbz^n$ for some $n\in\bbn$, hence $\Spec \bbc [M_i]\simeq\bbg_m^n$ is connected.
	
	Now assume that $M$ contains a nontrivial finite subgroup $M_0$. Then $\Spec \bbc [M]\to\Spec \bbc [M_0]$ is surjective, but the target is a disjoint union of $\lvert M_0\rvert$ copies of $\Spec\bbc$, so $\Spec \bbc [M]$ cannot be connected.
\end{proof}
For any abelian group $M$ the group algebra $\bbc [M]$ is a Hopf algebra: the comultiplication $\bbc [M]\to \bbc [M]\otimes_{\bbc } \bbc [M]$ is determined by $m\mapsto m\otimes m$ for all $m\in M$, and the other structure maps are even more obvious. This turns $\bbc [M]$ into a commutative group scheme over~$\bbc $. From the short exact sequence
\begin{equation*}
0\to M_{\mathrm{tors}}\to M\to M_{\mathrm{tf}}\to 0,
\end{equation*}
where $M_{\mathrm{tors}}$ is the torsion subgroup and $M_{\mathrm{tf}}$ is the maximal torsion-free quotient, we obtain a short exact sequence of group schemes
\begin{equation*}
0\to\Spec\bbc [M_{\mathrm{tf}}]\to\Spec \bbc [M]\to\Spec \bbc [M_{\mathrm{tors}}]\to 0.
\end{equation*}
Here $\Spec \bbc [M_{\mathrm{tf}}]$ is connected by Proposition~\ref{Prop:PiZeroOfGroupAlgebra}, and it is easy to see that the topological group underlying $\Spec\bbc [M_{\mathrm{tors}}]$ is isomorphic to $(M_{\mathrm{tors}})^{\vee }$, in particular totally disconnected. From this we see:
\begin{Corollary}\label{Cor:PiZeroOfGroupAlgebra}
	Let $M$ be an abelian group. Then $\Spec \bbc [M]\to\Spec\bbc [M_{\mathrm{tors}}]$ induces a homeomorphism on~$\pi_0(\cdot )$, and the identity component is isomorphic to $\Spec \bbc [M_{\mathrm{tf}}]$. In particular $\pi_0(\Spec\bbc [M]))$ is canonically isomorphic to $(M_{\mathrm{tors}})^{\vee }$ as a topological group, and the \'etale fundamental group of $\Spec \bbc [M]$ at any base point is isomorphic to that of $\Spec \bbc [M_{\mathrm{tf}}]$.
\end{Corollary}
\begin{proof}
	Everything is clear from the preceding, except the statement about fundamental groups; but because $\Spec\bbc [M]$ is a group scheme and any connected component contains a $\bbc$-rational point, every two connected components are isomorphic as schemes.
\end{proof}

\begin{Proposition}\label{Prop:FiniteCoversOfGroupAlgebraSpectrum}
	Let $M$ be a torsion-free abelian group. Then the connected \'etale coverings of $\Spec \bbc [M]$ are precisely given by the $\Spec \bbc [N]$, where $N$ runs through the subgroups of $V=M\otimes\bbq$ that contain $M$ with finite index.
\end{Proposition}
\begin{proof}
	The proof Proposition~\ref{Prop:FiniteCoversOfGroupAlgebraSpectrum} is closely analogous to that of Proposition~\ref{Prop:FiniteCoversOfPontryaginDual}. We begin again by observing that the desired result is well-known in the finitely generated case.
	
	For the general case, the argument that each $N$ defines an \'etale covering $\Spec\bbc [N]$ $\to\Spec\bbc [M]$ is directly parallel to the corresponding argument in the proof of Proposition~\ref{Prop:FiniteCoversOfPontryaginDual}, and we shall not repeat it.
	
	For the other implication, let $\mathcal{Y}\to\Spec \bbc [M]$ be an \'etale covering. Then $\mathcal{Y}$ must be affine, say $\mathcal{Y}=\Spec B$ for some finite \'etale ring homomorphism $\Spec \bbc [M]\to B$. By \cite[Tag 00U2, item (9)]{StacksProject} this must be the base change of some \'etale ring homomorphism $\bbc [M']\to B'$ for a finitely generated subgroup $M'\subseteq M$. Since $B$ is finite over $\bbc [M]$ so must $B'$ be over $\bbc [M']$, i.e.\ $\Spec B'\to\Spec \bbc [M']$ must be an \'etale covering, and by what we have already shown, $B'=\bbc [N']$ with $(N':M')$ finite. Then $B=\bbc [N'\oplus_{M'}M]$; since the spectrum of this algebra must be connected, $N'\oplus_{M'}M$ must be torsion-free, hence embed via the obvious map into $V=M\otimes\bbq$.
\end{proof}

\begin{Corollary}\label{Cor:FGofDiagonalisableGroupScheme}
	Let $M$ be a torsion-free abelian group. Then $\Spec\bbc [M]$ is connected, and for any geometric base point $x$ we obtain a natural isomorphism of profinite groups
	\begin{equation*}
	\pioneet (\Spec\bbc [M],x)\cong (M\otimes\bbq /M)^{\vee }.
	\end{equation*}
	In particular, $\Spec\bbc [M]$ is simply connected if and only if $M$ is a $\bbq$-vector space.\hfill $\square$
\end{Corollary}
This corollary is again directly analogous to Corollary~\ref{Cor:PioneetPontryaginDual}. Finally we note the following analogue of Proposition~\ref{Prop:CohomologyOfPontryaginDuals}.(i):
\begin{Proposition}\label{Prop:EtaleCohomologyOfSpecGroupAlgebra}
	Let $V$ be a $\bbq$-vector space and let $A$ be an abelian torsion group. Then the \'etale cohomology groups $\Hup^p_{\et }(\Spec\bbc [V],A)$ vanish for all $p>0$.
\end{Proposition} 
\begin{proof}
	The proof is analogous to that of Proposition~\ref{Prop:CohomologyOfPontryaginDuals}.(i).
\end{proof}

\section{Galois groups as \'etale fundamental groups of $\bbc$-schemes}\label{sec:GaloisAsEtaleFGOfSchemes}

\subsection{Rational Witt vectors}

We begin by recalling rings of `big' Witt vectors. Note that there are many different constructions of these rings; for more information see~\cite{MR2553661}.

In the following, rings are always supposed commutative and unital. The ring $\mathrm{W}(A)$ of Witt vectors in $A$ is defined for any ring $A$. Its underlying set is the set $1+tA[\![t]\!]$ of formal power series in $A[\![t]\!]$ with constant coefficient one. Addition of Witt vectors is multiplication of power series: $f\oplus g=fg$. Multiplication of Witt vectors is more involved.
\begin{Proposition}
	There is a unique system of binary operations $\odot$, consisting of one binary operation $\odot\colon\mathrm{W} (A)\times\mathrm{W} (A)\to\mathrm{W} (A)$ for each ring $A$, such that the following statements hold:
	\begin{enumerate}
		\item With $\oplus$ as addition and $\odot$ as multiplication, $\mathrm{W} (A)$ becomes a ring.
		\item For any ring $A$ and elements $a,b\in A$ the equation
		\begin{equation*}
		(1-at)\odot (1-bt)=1-abt
		\end{equation*}
		holds in $\mathrm{W} (A)$.
		\item The operation $\odot$ is functorial in $A$: for a ring homomorphism $\varphi\colon A\to B$ and elements $f,g\in\mathrm{W} (A)$ the equation $\mathrm{W} (\varphi )(f\odot g)=\mathrm{W} (\varphi)(f)\odot\mathrm{W} (\varphi )(g)$ holds.
		
		Here $\mathrm{W}(\varphi )\colon\mathrm{W}(A)\to\mathrm{W}(B)$ is the obvious map that sends $t$ to $t$ and acts as $\varphi$ on the coefficients.
		\item The operation $\odot$ is continuous for the $t$-adic topology on $\mathrm{W} (A)$.
	\end{enumerate}
	 Hence $(\mathrm{W} (A),\oplus,\odot )$ becomes a complete topological ring, and $\mathrm{W}$ becomes a functor from rings to complete topological rings.
\end{Proposition}
The proof of this result can be found in many sources, e.g.\ \cite[Section~9]{MR2553661}.
\begin{Proposition}
	Let $A$ be a ring. The set
	\begin{equation*}
	\Wrat (A)\defined \left\{ f\in\mathrm{W(A)}\,\middle\lvert\, f=\frac{1+a_1t+a_2t^2+\dotsb +a_nt^n}{1+b_1t+b_2t^2+\dotsb +b_mt^m},\, a_i,b_j\in A \right\}
	\end{equation*}
	is a subring of $\mathrm{W}(A)$.
\end{Proposition}
It seems that this result first appeared explicitly in the literature as \cite[Proposition~3.4]{MR0424786}. The elements of $\Wrat (A)$ are called \emph{rational Witt vectors}. The rings $\Wrat (A)$ occur naturally in some problems in K-theory, see \cite{MR0424786, MR523461, MR0260842}.

\begin{Remark}\label{Rmk:RationalWittVectorsForFields}
	In case $A=F$ is a field of characteristic zero, there is a more elementary description of $\Wrat (F)$. First assume that $F$ is algebraically closed. Then the set of all polynomials $1-\alpha t$, where $\alpha$ runs through $F^{\times }$, is a basis of the abelian group underlying $\Wrat (F)$, and the product of two basis elements corresponding to $\alpha$ and $\beta$, respectively, is the basis element corresponding to $\alpha\beta$. This means that $\Wrat (F)$ is canonically isomorphic to the group ring $\bbz [F^{\times }]$.

In the general case, choose an algebraic closure $\Fbar /F$. Then there is a natural action of $\Gal (\Fbar /F)$ on $\Wrat (\Fbar )$, the ring of invariants being canonically isomorphic to $\Wrat (F)$. The isomorphism $\Wrat (\Fbar)\cong\bbz [\Fbar^{\times}]$ is equivariant for this Galois action, hence $\Wrat (F)$ is canonically isomorphic to the ring of $\Gal (\Fbar /F)$-invariants in $\bbz [\Fbar^{\times}]$.
\end{Remark}

Now we can give the the construction of the schemes $\mathcal{X}_F$. So, let $F$ be a field containing $\bbq (\zeta_{\infty})\defined\bigcup_n \bbq(\zeta_n)$, where the latter is assumed embedded into~$\bbc$. The group homomorphism
\begin{equation*}
\mu_{\infty }\to\Wrat (F),\quad \zeta\mapsto [\zeta]\defined 1-\zeta t
\end{equation*}
and the canonical inclusion
\begin{equation*}
\mu_{\infty }\hookrightarrow \bbc^{\times}
\end{equation*}
define ring homomorphisms $\bbz [\mu_{\infty}]\to\Wrat (F)$ and $\bbz [\mu_{\infty }]\to\bbc$, respectively. We then set
\begin{equation*}
A_F=\Wrat (F)\otimes_{\bbz [\mu_{\infty}]}\bbc\qquad\text{and}\qquad \mathcal{X}_F=\Spec A_F.
\end{equation*}
Note that if $\Fbar /F$ is an algebraic closure, $A_{\Fbar}$ comes with an action by $\Gal (\Fbar /F)$, and the ring of invariants is canonically isomorphic to~$A_F$.

\begin{Remark}\label{Rmk:XEForArbitraryAlgebras}
	In fact we may define a $\bbc$-algebra $A_E$ and a $\bbc$-scheme $\mathcal{X}_E=\Spec A_E$ for any $\bbq (\zeta_{\infty })$-algebra $E$ by the same formula. We will only use this construction in the case where $E$ is a finite product of fields, say $E=E_1\times\dotsb\times E_n$ with each $E_{\nu }$ a field extension of~$\bbq (\zeta_{\infty })$; then there are natural isomorphisms
	\begin{equation*}
	A_E\cong\prod_{\nu =1}^n A_{E_{\nu }}\quad\text{and}\quad\mathcal{X}_E\cong\coprod_{\nu =1}^n\mathcal{X}_{E_{\nu }}.
	\end{equation*}
\end{Remark}

The following is the version of Theorem~\ref{thm:Main} for schemes.

\begin{Theorem}\label{thm:MainScheme} For any field $F\supset \bbq (\zeta_\infty)$, $\mathcal{X}_F$ is connected, and the \'etale fundamental group of $\mathcal{X}_F$ is isomorphic to the absolute Galois group of $F$. More precisely, the functor $E\mapsto \mathcal{X}_E$ induces a (degree-preserving) equivalence of categories between the category of finite \'etale $F$-algebras and the category of finite \'etale schemes over $\mathcal{X}_F$.
\end{Theorem}

A special case of this can be handled directly.

\begin{Proposition}\label{Prop:XFbarIsSC}
	If $F$ is algebraically closed, $\mathcal{X}_F$ is connected and simply connected.
\end{Proposition}
\begin{proof}
  We may identify $\mathcal{X}_F$ with
  \begin{equation*}
  \Spec \left((\Wrat (F)\otimes_{\bbz}\bbc )\otimes_{\bbc [\mu_{\infty}]}\bbc\right) \cong \Spec\bbc [F^{\times}]\times_{\Spec \bbc [\mu_{\infty}]}\Spec\bbc .
  \end{equation*}
  The embedding $\mu_{\infty}\hookrightarrow F^{\times}$ induces an isomorphism on torsion subgroups, hence by Corollary~\ref{Cor:PiZeroOfGroupAlgebra} the morphism $\Spec\bbc [F^{\times}]\to\Spec \bbc [\mu_{\infty}]$ induces a homeomorphism on $\pi_0(\cdot )$. The morphism $\Spec \bbc\to\Spec\bbc [\mu_{\infty}]$ picks one connected component, hence the fibre product $\mathcal{X}_F$ must be connected, more precisely identified with one particular connected component of $\Spec\bbc [F^{\times}]$.
  
  By Corollary~\ref{Cor:PiZeroOfGroupAlgebra}, $\pioneet (Y_F)$ is therefore isomorphic to $\pioneet (\Spec \bbc [(F^{\times})_{\mathrm{tf}}])$. Since $F^{\times}$ is divisible, $(F^{\times})_{\mathrm{tf}}$ is a $\bbq$-vector space, and by Corollary~\ref{Cor:FGofDiagonalisableGroupScheme} the spectrum of its group algebra is simply connected. Hence so is~$\mathcal{X}_F$.
\end{proof}

\subsection{Recognising properties of scheme morphisms on geometric points}

We assemble a few well-known observations on morphisms of schemes which will be useful later.

\begin{Proposition}\label{Prop:QuotientByFreeFiniteActionFiniteEtale}
	Let $\mathcal{X}=\Spec A$ be an affine scheme and let $G$ be a finite group operating on $\mathcal{X}$, hence on $A$. Assume that for any algebraically closed field $k$ the group $G$ operates freely on $\mathcal{X}(k)$.
	
	Then the natural morphism $\pi\colon \mathcal{X}\to \mathcal{Y}=\Spec A^G$ is a finite \'etale Galois covering with deck group~$G$.
\end{Proposition}

\begin{proof}
	We claim that the natural morphism $\mathcal{X}\times G\to \mathcal{X}\times \mathcal{X}$, $(x,g)\mapsto (x,xg)$ is a closed immersion. It is clearly finite, as the composite with each projection $\mathcal{X}\times\mathcal{X}\to\mathcal{X}$ is finite. Thus, by \cite[Corollaire~18.12.6]{MR0238860}, it remains to show that it is a monomorphism. Thus, let $\mathcal{S}$ be some scheme with two maps $a=(a_\mathcal{X},a_G),b=(b_\mathcal{X},b_G)\colon \mathcal{S}\to \mathcal{X}\times G$ whose composites with $\mathcal{X}\times G\to \mathcal{X}\times \mathcal{X}$ agree. In particular, it follows that the two maps $a_\mathcal{X},b_\mathcal{X}\colon \mathcal{S}\to \mathcal{X}$ agree. It remains to see that the two maps $a_G,b_G\colon \mathcal{S}\to G$ agree. As both maps are constant on connected components, it suffices to check this on geometric points, where it follows from the assumption.

Since the component maps $\mathcal{X}\times G\to \mathcal{X}$ are also finite \'etale, the morphism $\mathcal{X}\times G\to \mathcal{X}\times \mathcal{X}$ is a finite \'etale equivalence relation on~$\mathcal{X}$. Hence its quotient $[\mathcal{X}/G]$ is an algebraic space; by \cite[Tag 03BM]{StacksProject} it must be representable by an affine scheme, and the morphism $\mathcal{X}\to [\mathcal{X}/G]$ is finite. This affine scheme represents the same functor as $\mathcal{Y}$, therefore $\mathcal{Y}=[\mathcal{X}/G]$ and $\mathcal{X}\to \mathcal{Y}$ is finite. It is also \'etale by \cite[Tag 02WV]{StacksProject}.
\end{proof}

\begin{Lemma}\label{Lem:IntegralSurjectiveOnGeomPoints}
	Let $B\subseteq A$ be an integral ring extension, and let $f\colon \mathcal{X}=\Spec A\to\Spec B=\mathcal{Y}$ be the corresponding morphism of schemes. Then for every algebraically closed field $k$ the map $\mathcal{X}(k)\to \mathcal{Y}(k)$ is surjective.
\end{Lemma}

\begin{proof}
	Let $\widebar{y}\colon\Spec k\to \mathcal{Y}$ be a geometric point with image $y\in \mathcal{Y}$. Then $\widebar{y}$ factors as $\Spec k\to\Spec\kappa (y)\hookrightarrow \mathcal{Y}$, defined by a field extension $\kappa (y)\hookrightarrow k$.
	
	By \cite[Tag 00GQ]{StacksProject} $f$ is surjective on topological points, hence there exists some $x\in \mathcal{X}$ with $f(x)=y$. By integrality, $\kappa (x)$ is an algebraic extension of $\kappa (y)$. Since $k$ was assumed to be algebraically closed, we can find an embedding $\kappa (y)\hookrightarrow k$ making the diagram
	\begin{equation*}
	\xymatrix{
		&\Spec\kappa (x)\ar[r]\ar[d] &\mathcal{X}\ar[d]\\
		\Spec k\ar[r]\ar@{-->}[ru] & \Spec\kappa (y)\ar[r] & \mathcal{Y}
		}
	\end{equation*}
	commute. Then the composition $\widebar{x}\colon\Spec k\to\Spec\kappa (x)\to \mathcal{X}$ is a preimage of $\widebar{y}$ in~$\mathcal{X}(k)$.
\end{proof}

\begin{Lemma}\label{Lem:GeoPointsOfQuotient}
	Let $A$ be a ring, and let $G$ be a profinite group acting continuously on~$A$. Let $A^G\subseteq A$ be the ring of invariants, and let $\varphi\colon \mathcal{X}=\Spec A\to\Spec A^G=\mathcal{Y}$ be the associated morphism of schemes. Then for every algebraically closed field $k$ the induced map $\mathcal{X}(k)/G\to \mathcal{Y}(k)$ is bijective.
\end{Lemma}

\begin{proof}
	Note that the map $A^G\to A$ is integral: For any $a\in A$, the $G$-orbit $Ga = \{a_1,\ldots,a_n\}$ of $a$ is finite, and then $a$ is a root of the monic polynomial $P(X)=\prod_{i=1}^n (X-a_i)\in A^G[X]$. Surjectivity therefore follows from Lemma~\ref{Lem:IntegralSurjectiveOnGeomPoints}, and we only need to show injectivity.
	
	First, we handle the case of finite $G$. Let $p$ be the characteristic of $k$. First, we reduce to the case $A$ is an algebra over the corresponding prime field. If $p=0$, then $(A\otimes \bbq )^G = A^G\otimes \bbq$ as invariants commute with filtered colimits, so we can replace $A$ by $A\otimes \bbq$. If $p>0$, we first assume that $A$ is flat over $\bbz$, by replacing $A$ by the free algebra $\bbz [x_a|a\in A]$ on the elements of $A$, which admits a natural $G$-equivariant surjective map to $A$. Assuming now that $A$ is flat over $\bbz$, the map $A^G/p\to (A/p)^G$ is injective, but need not be an isomorphism. However, we claim that the map induces an isomorphism on perfections, i.e.~on the filtered colimit over $a\mapsto a^p$; in particular, the $k$-valued points are the same. We need to show that whenever $a\in A/p$ is $G$-invariant, there is some $n$ such that $a^{p^n}$ lifts to an element of $A^G$. Note that there is a commutative diagram
\[\xymatrix{
0\ar[r] & A^G\ar[r]\ar[d] & A^G\ar[r]\ar[d]& (A/p)^G\ar[r]\ar[d] & \Hup^1(G,A)\ar[d] \\
0\ar[r] & (A/p^n)^G\ar[r] & (A/p^{n+1})^G \ar[r] & (A/p)^G\ar[r] & \Hup^1(G,A/p^n)\ .\\
}\]
Choose $n$ large enough that the $p$-part of the order of $G$ divides $p^n$. Then $\Hup^1(G,A)$ is killed by $p^n$, and thus the map $\Hup^1(G,A)\to \Hup^1(G,A/p^n)$ is injective. Thus, if an element of $(A/p)^G$ can be lifted to $(A/p^{n+1})^G$, then it can be lifted all the way to $A^G$. But for any $a\in A/p$, the element $a^{p^n}$ lifts canonically to $A/p^{n+1}$: Indeed, for any two lifts $\tilde{a}_1,\tilde{a}_2\in A$ of $a$, one has $\tilde{a}_1^{p^n}=\tilde{a}_2^{p^n}\in A/p^{n+1}$. It follows that for any $a\in (A/p)^G$, $a^{p^n}$ lifts to $(A/p^{n+1})^G$.

In particular, we can assume that $A$ is defined over a field. Let $x_0,x_1\in \mathcal{X}(k)$ be in different $G$-orbits. Then for every $g\in G$ the induced homomorphisms $x_{\nu}\circ g\colon A_k=A\otimes k\to k$ are surjective, hence their kernels are maximal ideals in~$A_k$. By assumption, these ideals are all distinct, hence (by maximality) coprime. By the Chinese Remainder Theorem we then find some $f\in A_k$ which is sent to $0$ by all $x_0\circ g$ and to $1$ by all $x_1\circ g$. After possibly replacing $f$ by
	\begin{equation*}
	\prod_{g\in G}g(f)
	\end{equation*}
	we may assume that $f\in A_k^G=A^G\otimes k$ (as $k$ is a free module over its prime field). Then $f(\varphi (x_0))=0$ and $f(\varphi (x_1))=1$, whence $\varphi (x_0)\neq\varphi (x_1)$. 

	This finishes the case that $G$ is finite. In general, $A$ is the filtered colimit of its subrings $A^H$, where $H\subset G$ runs through open subgroups. Let $\mathcal{Y}_H=\Spec(A^H)$, so that $\mathcal{X}(k)=\varprojlim_{H\subset G} \mathcal{Y}_H(k)$. By the case of finite $G$, we know that $\mathcal{Y}_H(k)/(G/H)=\mathcal{Y}(k)$. Therefore, if $x,y\in \mathcal{X}(k)$ map to the same element of $\mathcal{Y}(k)$, then their images in $\mathcal{Y}_H(k)$ lie in the same $G/H$-orbit, in particular in the same $G$-orbit. For each $H$, we get a nonempty closed subset of $G$ of elements which carry the image of $x$ in $\mathcal{Y}_H(k)$ to the image of $y$ in $\mathcal{Y}_H(k)$. By the variant of Cantor's Intersection Theorem given as Lemma~\ref{Lem:CantorIntersection} below, their intersection is nonempty, which gives an element of $G$ carrying $x$ to $y$.
\end{proof}
\begin{Lemma}\label{Lem:CantorIntersection}
	Let $X$ be a compact topological space and let $(A_i)_{i\in I}$ be a family of non-empty closed subspaces of~$X$. Assume the family is cofiltered in the sense that for every $i,j\in I$ there is some $k\in I$ such that $A_k\subseteq A_i\cap A_j$.
	
	Then the intersection $\bigcap_{i\in I}A_i$ is non-empty.
\end{Lemma}
\begin{proof}
	Assume that the intersection is empty. Then the union of the open subsets $U_i=X\smallsetminus A_i$ is all of $X$, that is, the $U_i$ form an open cover of~$X$. By compactness there exists some finite subcover, say $X=U_{i_1}\cup\dotsb\cup U_{i_n}$. This means that $A_{i_1}\cap\dotsb\cap A_{i_n}=\varnothing$. But by our assumption there exists some $k\in I$ such that $A_k\subseteq A_{i_1}\cap\dotsb\cap A_{i_n}=\varnothing$, which contradicts our assumption that $A\neq\varnothing$.
\end{proof}

\subsection{Classification of \'etale covering spaces of~$\mathcal{X}_F$}

Now let $F$ be any field containing $\bbq (\zeta_{\infty })\subset\bbc$. Consider the $\bbc$-scheme $\mathcal{X}_F$ as defined before. By Remark~\ref{Rmk:XEForArbitraryAlgebras} we obtain a contravariant functor $E\mapsto\mathcal{X}_E$ from $F$-algebras to $\mathcal{X}_F$-schemes, in other words, a covariant functor from affine $(\Spec F)$-schemes to $\mathcal{X}_F$-schemes.
\begin{Theorem}\label{Theorem:CoversOfXFandExtensionsOfF}
	For any field $F\supseteq\bbq (\zeta_{\infty })$ the $\bbc$-scheme $\mathcal{X}_F$ is connected. If $E/F$ is a finite \'etale $F$-algebra then $\mathcal{X}_E\to \mathcal{X}_F$ is a finite \'etale covering space. The resulting functor
	\begin{equation*}
	\mathbf{FEt}(\Spec F)\to\mathbf{FEt}(\mathcal{X}_F),\quad \Spec E\mapsto\mathcal{X}_E,
	\end{equation*}
	is an equivalence of categories.
\end{Theorem}
The proof of this theorem rests on the generalities proved before, as well as the following observations.
\begin{Lemma}\label{Lem:MitHilbert90}
	Let $F\supseteq\bbq (\zeta_{\infty })$ be a field, let $\Fbar /F$ be an algebraic closure and let $G=\Gal (\Fbar /F)$ be the corresponding Galois group. Let $k$ be a field. Then $G$ operates freely on the set
	\begin{equation*}
	\mathcal{I}(\Fbar,k)=\{\chi\colon (\Fbar)^{\times }\to k^{\times }\mid \text{$\chi$ is a group homomorphism, injective on $\mu_{\infty}$} \}.
	\end{equation*}
\end{Lemma}
\begin{proof}
	Let $\sigma\in G\smallsetminus\{ 1 \}$ and $\chi\in\mathcal{I}(\Fbar ,k)$. We need to show that $\sigma (\chi )\neq\chi$.
	
	Without loss of generality we may assume that $\sigma$ topologically generates $G$ (otherwise we replace $F$ by the fixed field of $\sigma$). Then $G$ is a nontrivial procyclic group. We claim that there exists some prime $p$ such that the closed subgroup generated by $\sigma^p$ is a proper subgroup of~$G$: otherwise, for every finite quotient $H$ of $G$ and every prime $p$ the $p$-th power map $H\to H$ is surjective; since $H$ is a finite abelian group this is only possible if $H$ is trivial. But if every finite quotient of $G$ is trivial, then so is $G$ itself, contradiction.
	
	Now pick some $p$ such that $\overline{\langle\sigma^p\rangle}$ is properly contained in $G$, and let $E$ be the fixed field of~$\sigma^p$. Then $E/F$ is a cyclic Galois extension of degree~$p$.
	
	Let $\zeta\in F$ be a primitive $p$-th root of unity. Note that since $\chi$ is injective on $\mu_{\infty}$ we must have $\chi (\zeta )\neq 1$. Furthermore $\mathrm{N}_{E/F}(\zeta )=\zeta^p=1$, hence by the original form of Hilbert's `Satz~90' (\cite[Satz~90]{Zahlbericht}, see also \cite[Chapter~IV, Theorem~3.5]{MR1697859}) there is some $\alpha\in E^{\times}$ with $\zeta =\sigma (\alpha )/\alpha$. But then $\chi (\sigma (\alpha )/\alpha )=\chi (\zeta )\neq 1$, hence $\chi (\sigma (\alpha ))\neq \chi (\alpha )$. Therefore $\sigma (\chi )\neq\chi$.
\end{proof}
\begin{Lemma}\label{Lem:AnwendungHilbert90AufGeoPunkte}
		Let $F$, $\Fbar$, $G$ and $k$ be as in Lemma~\ref{Lem:MitHilbert90}. Then $G$ operates freely on $\mathcal{X}_{\Fbar }(k)$.
\end{Lemma}
\begin{proof}
	Just note that
	\begin{equation*}
	\mathcal{X}_{\Fbar}(k)=\Hom (\bbz [\Fbar^{\times}]\otimes_{\bbz [\mu_{\infty }]}\bbc ,k)
	\end{equation*}
	can be identified with the set of pairs $(\chi ,g)$, where $\chi \colon\Fbar^{\times}\to k^{\times}$ is a group homomorphism and $g\colon \bbc\to k$ is a field embedding such that $\chi $ and $g$ agree on all roots of unity. The Galois action is given by $\sigma (\chi ,g)=(\sigma (\chi ),g)$. The $\chi$ occurring are all injective on $\mu_{\infty}$, and therefore the desired result follows from Lemma~\ref{Lem:MitHilbert90}.
\end{proof}
\begin{Lemma}\label{Lem:GeometricPointsOfXF}
	The map $\mathcal{X}_{\Fbar}(k)\to\mathcal{X}_F(k)$ is constant on $G$-orbits, and the induced map $\mathcal{X}_{\Fbar }(k)/G\to\mathcal{X}_F(k)$ is a bijection.
\end{Lemma}
\begin{proof}
	Recall that $\mathcal{X}_{\Fbar }=\Spec A_{\Fbar }$ and $\mathcal{X}_F=\Spec A_{\Fbar }^G$ (the ring of $G$-invariants), so the result follows from Lemma~\ref{Lem:GeoPointsOfQuotient}.
\end{proof}
\begin{proof}[Proof of Theorem~\ref{Theorem:CoversOfXFandExtensionsOfF}]
Let $E$ be a finite \'etale $F$-algebra. We will show that $\mathcal{X}_E\to\mathcal{X}_F$ is a finite \'etale covering.

First assume that $E$ is a field, Galois over~$F$. Combining Lemmas \ref{Lem:AnwendungHilbert90AufGeoPunkte} and \ref{Lem:GeometricPointsOfXF} we find that $\Gal (E/F)$ operates freely on $\mathcal{X}_E(k)$, for any algebraically closed field~$k$. By Proposition~\ref{Prop:QuotientByFreeFiniteActionFiniteEtale} then $\mathcal{X}_E\to\mathcal{X}_F$ must be finite \'etale.

For a general finite field extension $E/F$ let $E'/F$ be the Galois closure of $E$ in $\Fbar$; then the composition $\mathcal{X}_{E'}\to\mathcal{X}_E\to\mathcal{X}_F$ is finite \'etale, as is the first component, hence the second has to be finite \'etale, too.

Finally, if $E$ is an arbitrary finite \'etale $F$-algebra, there is a canonical isomorphism
\begin{equation*}
E\overset{\cong}{\to}\prod_{\mathfrak{p}\in\Spec E}E/\mathfrak{p}
\end{equation*}
with the $E/\mathfrak{p}$ being finite field extensions of~$F$. Therefore, by Remark~\ref{Rmk:XEForArbitraryAlgebras}, we obtain an isomorphism of $\mathcal{X}_F$-schemes
\begin{equation*}
\mathcal{X}_E\overset{\cong}{\to}\coprod_{\mathfrak{p}\in\Spec E}\mathcal{X}_{E/\mathfrak{p}},
\end{equation*}
hence $\mathcal{X}_E\to\mathcal{X}_F$ is also finite \'etale.

We now have shown that the functor $\Spec E\mapsto\mathcal{X}_E$ really sends $\mathbf{FEt}(\Spec F)$ to $\mathbf{FEt}(\mathcal{X}_F)$. Note that if $E/F$ is Galois, then $\mathcal{X}_E\to\mathcal{X}_F$ is a Galois covering with group $\Gal (E/F)$. From this we deduce that the functor is fully faithful. It remains to be shown that it is essentially surjective; this follows by applying Lemma~\ref{Lem:SimplyConnectedProEtaleIsUniversal} to the system of \'etale coverings $(\mathcal{X}_E\to\mathcal{X}_F)_{E/F\text{ finite}}$, using that the limit $\mathcal{X}_{\Fbar }=\varprojlim_{E}\mathcal{X}_E$ is simply connected by Proposition~\ref{Prop:XFbarIsSC}.
\end{proof}

\begin{Corollary}\label{Cor:PioneetOfXFIsAbsoluteGalois}
	`The' \'etale fundamental group of $\mathcal{X}_F$ is naturally isomorphic to `the' absolute Galois group of~$F$.\hfill $\square$
\end{Corollary}
The formulation of this corollary requires some explanation. Note that classically, to speak sensibly of the absolute Galois group of a field $F$ one needs to fix an algebraic closure~$\Fbar$. However, we can get by with a slightly more general type of object. In Grothendieck's interpretation of the absolute Galois group, $\Gal (\Fbar /F)$ is the \'etale fundamental group of $\Spec F$ at the geometric point $\Spec\Fbar\to\Spec F$. This is the automorphism group of the fibre functor $\Psi_{\Fbar}$ on the Galois category $\mathbf{FEt}(\Spec F)$, given by $\Psi_{\Fbar}(X)=X(\Fbar)$. If $\Psi$ is \emph{any} fibre functor on $\mathbf{FEt}(\Spec F)$, then $\Aut \Psi$ is still isomorphic to $\Gal (\Fbar /F)$, the isomorphism being canonical up to inner automorphisms. Hence we might call $\Aut (\Psi)$ `the absolute Galois group of $F$ at~$\Psi$'.

Now, let $\overline{x}\colon\Spec\varOmega\to \mathcal{X}_F$ be a geometric point. Then $\pioneet (\mathcal{X}_F,\overline{x})$ is the automorphism group of the fibre functor $\Phi_{\overline{x}}$ on $\mathbf{FEt}(\mathcal{X}_F)$ with $\Phi_{\overline{x}}(Y)=Y_{\overline{x}}$. The composition
\begin{equation*}
\Psi_{\overline{x}}\colon\mathbf{FEt}(\Spec F)\overset{\mathcal{X}_{(\cdot )}}{\longrightarrow}\mathbf{FEt}(\mathcal{X}_F)\overset{\Phi_{\overline{x}}}{\longrightarrow}\mathbf{Sets}
\end{equation*}
is still a fibre functor on $\mathbf{FEt}(\Spec F)$, and Corollary~\ref{Cor:PioneetOfXFIsAbsoluteGalois} says that the absolute Galois group of $F$ at $\Psi_{\overline{x}}$ is canonically isomorphic to $\pioneet (\mathcal{X}_F,\overline{x})$, for any geometric point $\overline{x}$ of~$\mathcal{X}_F$.

\begin{Remark} Let $A$ be any ring and $B$ a finite \'etale $A$-algebra of constant degree $d$. Above, we have shown that if $A$ is a field over $\bbq (\zeta_\infty)$, then the map $\Wrat(A)\to \Wrat(B)$ becomes finite \'etale after base change from $\bbz [\mu_{\infty} ]$ to $\bbc$; by faithfully flat base change, this is already true after base change along $\bbz [\mu_{\infty } ]\to \bbq (\zeta_\infty)$. It is not evident from our proof how general this result is, so we want to mention the following generalisation. We do not know whether the assumption $\Pic (A)=0$ is necessary, and it would be nice to remove it.

\begin{Theorem}\label{thm:Wratfiniteetale} Let $A$ be a ring with $\Pic (A)=0$, let $S$ be a set of primes which is invertible in $A$, and let $B$ be a finite \'etale $A$-algebra. For a prime $p\in S$, consider the element $\Phi_p\in \Wrat(A)$ given by the cyclotomic polynomial $\prod_{i=1}^{p-1} (1-\zeta_p^i t)\in \Wrat(\bbz )\to \Wrat(A)$. Then the map
\[
\Wrat(A)[ (\Phi_p-(p-1))^{-1}\mid p\in S]\to \Wrat(B)[ (\Phi_p-(p-1))^{-1}\mid p\in S]
\]
is finite \'etale in the following cases:
\begin{enumerate}
\item[{\rm (i)}] The algebra $B$ is everywhere of degree $\leq d$ over $A$, and $S$ contains the set of primes $p\leq d$.
\item[{\rm (ii)}] The algebra $B$ is Galois over $A$ with Galois group $G$, and $S$ contains the set of primes dividing the order of $G$.
\end{enumerate}
\end{Theorem}

Note that if $A$ contains a $p$-th root of unity $\zeta_p\in A$ and we denote $[\zeta_p] = 1-\zeta_p t\in \Wrat(A)$, then $\Phi_p = [\zeta_p] +[\zeta_p^2] + \ldots + [\zeta_p^{p-1}]$, and inverting $\Phi_p-(p-1)$ is equivalent to inverting $[\zeta_p]-1$. We do not know to what extent it is necessary to invert $\Phi_p-(p-1)$ for all $p\in S$.

\begin{proof} One may reduce part (i) to part (ii) by passing to the Galois closure. Now $\Wrat(A) = \Wrat(B)^G$ and this passes to the filtered colimit, giving
\[
\Wrat(A)[ (\Phi_p-(p-1))^{-1}\mid p\in S] = \Wrat(B)[ (\Phi_p-(p-1))^{-1}\mid p\in S]^G\ .
\]
By Proposition~\ref{Prop:QuotientByFreeFiniteActionFiniteEtale}, it is enough to check that $G$ acts freely on geometric points of $\Spec(\Wrat(B)[ ([\zeta_p]-1)^{-1}\mid p\in S])$. Thus, for every $1\neq g\in G$, we need to check that $g$ acts freely; replacing $g$ by a power, we can assume that the order of $g$ is a prime $p$. We can also replace $G$ by the cyclic subgroup generated by $g$ (and $A$ by the invariants of $B$ under this subgroup), and assume that $G\cong \bbz /p\bbz$ is cyclic of prime order.

We make a further reduction to assume that $A$ contains a $p$-th root of unity. Indeed, let $A_1 = A\otimes_{\bbz } \bbz [\zeta_p]$, which is a finite \'etale Galois cover with Galois group $G_1=(\bbz /p\bbz )^\times$. We claim that
\[
\Wrat(A)[ (\Phi_p-(p-1))^{-1}]\to \Wrat(A_1)[ (\Phi_p-(p-1))^{-1}] = \Wrat(A_1)[([\zeta_p]-1)^{-1}]
\]
is a finite \'etale Galois cover with Galois group $G_1$. To check this, as before it is enough to check that $G_1$ acts trivially on geometric points of $\Spec(\Wrat(A_1)[([\zeta_p]-1)^{-1}])$. But for any map $\Wrat(A_1)[([\zeta_p]-1)^{-1}]\to k$, the image of $[\zeta_p]$ in $k$ will be a nontrivial $p$-th root of $1$, so that no nontrivial element of $G_1$ fixes the image of $[\zeta_p]$ in $k$, and in particular $G_1$ acts freely on geometric points. A similar statement holds for $B\to B_1\defined B\otimes_A A_1$, and by faithfully flat base change it is enough to check the result for $A_1\to B_1$.

Thus, we are reduced to the case that $A$ is a $\bbz [\frac 1p,\zeta_p]$-algebra, and $B$ is a finite \'etale $G=\bbz /p\bbz$-cover of $A$; we want to prove that the map
\[
\Wrat(A)[ ([\zeta_p]-1)^{-1}]\to \Wrat(B)[ ([\zeta_p]-1)^{-1}]
\]
is finite \'etale, for which it is enough to check that $G$ acts freely on geometric points of $\Spec(\Wrat(B)[ ([\zeta_p]-1)^{-1}])$. In that case, by Kummer theory and our assumption $\Pic (A)=0$, $B$ is given by adjoining the $p$-th root $a^{1/p}\in B$ of some element $a\in A^\times$. But then $a^{1/p}$ gives an element $[a^{1/p}]\in \Wrat(B)$ on which $1\in G=\bbz /p\bbz$ acts by $[a^{1/p}]\mapsto [\zeta_p a^{1/p}] = [\zeta_p] [a^{1/p}]$. For any geometric point $\Wrat(B)[([\zeta_p]-1)^{-1}]\to k$, the image of $[\zeta_p]-1$ is invertible in $k$, so it follows that $[a^{1/p}]$ maps to $0$ in $k$; but this is impossible, as $a^{1/p}$ and thus $[a^{1/p}]$ is a unit.
\end{proof}
\end{Remark}

\section{Galois groups as \'etale fundamental groups of topological spaces}\label{sec:GaloisTop}

\noindent We now present the construction of a compact Hausdorff space $X_F$ for every field $F$ containing $\bbq (\mu_{\infty })$, with properties analogous to those of~$\mathcal{X}_F$. In fact, the two constructions are closely related, see Theorem~\ref{Thm:CompatibilityBetweenXFSpaceAndXFScheme} below.

\subsection{The spaces $X_F$}

Let $F$ be a field containing~$\bbq (\zeta_{\infty })$ and let $\Fbar /F$ be an algebraic closure. We endow $\Fbar^{\times }$ with the discrete topology and consider the Pontryagin dual $(\Fbar^{\times })^{\vee }=\Hom (\Fbar^{\times },\bbs^1)$. Letting $\iota\colon\mu_{\infty }\hookrightarrow\bbs^1$ be the obvious embedding, we then set
\begin{equation*}
X_{\Fbar }=\{ \chi\in (\Fbar^{\times })^{\vee  }\mid \chi\rvert_{\mu_{\infty }}=\iota \} ,
\end{equation*}
endowed with the subspace topology from $(\Fbar^{\times })^{\vee}$. The absolute Galois group $G=\Gal (\Fbar /F)$ operates from the left on $X_{\Fbar }$ by homeomorphisms via $\sigma (f)=f\circ\sigma^{-1}$. Then we set
\begin{equation*}
X_F=G\backslash X_{\Fbar },
\end{equation*}
endowed with the quotient topology. Note that if $\Fbar '/F$ is another algebraic closure of~$F$, the version of $X_F$ constructed from $\Fbar$ and that constructed from $\Fbar '$ are canonically homeomorphic.
\begin{Proposition}\label{Prop:SimplePropertiesOfCHXF}
	Let $F\supseteq\bbq (\zeta_{\infty })$ be a field with algebraic closure $\Fbar$ and absolute Galois group $G=\Gal (\Fbar /F)$.
	\begin{enumerate}
		\item The $G$-action on $X_{\Fbar }$ is proper and free.
		\item The space $X_F$ is nonempty, connected, compact and Hausdorff.
		\item The \'etale fundamental group of $X_{\Fbar}$ is trivial.
	\end{enumerate}
\end{Proposition}
\begin{proof}
	We begin by showing that $X_{\Fbar }$ is nonempty; it will follow that $X_F$ is also nonempty. The group $\Fbar^{\times }$ is divisible, and its torsion subgroup is equal to~$\mu_{\infty }$. By general facts about divisible groups it can then be written as a direct sum $\Fbar^{\times }=V\oplus\mu_{\infty}$ where $V$ is a $\bbq$-vector space. Then we can construct an element $\chi\in X_{\Fbar }\subset\Hom (\Fbar^{\times },\bbs^1)=\Hom (V,\bbs^1)\times\Hom (\mu_{\infty },\bbs^1)$ by declaring it to be $\iota$ on $\mu_{\infty }$ and any group homomorphism on~$V$. Note, however, that $V$ can in general not be chosen Galois invariant.

	 The space $X_{\Fbar}$ is a translate of, and therefore homeomorphic to, the closed subgroup $\Hom (\Fbar^{\times }/\mu_{\infty},\bbs^1)\subset\Hom (\Fbar^{\times },\bbs^1)$. This is clearly a connected compact Hausdorff space, and it has trivial \'etale fundamental group by Corollary~\ref{Cor:PioneetPontryaginDual}, which proves (iii).
	
	For (i), note that the action being proper means that the map $G\times X_{\Fbar }\to X_{\Fbar}\times X_{\Fbar }$ sending $(g,x)$ to $(gx,x)$ is proper; but since both $G$ and $X_{\Fbar }$ are compact, this follows automatically from continuity. That the action is free is a direct consequence of Lemma~\ref{Lem:MitHilbert90}.
	
	For (ii), note that we already know the corresponding statement for~$X_{\Fbar}$. From this and (i) we easily deduce~(ii).
\end{proof}
Similarly to the scheme-theoretic case, let $E$ be finite \'etale $F$-algebra. Recall that there is a canonical isomorphism
\begin{equation*}
E\overset{\cong}{\to}\prod_{\mathfrak{p}\in\Spec E}E/\mathfrak{p};
\end{equation*}
we therefore set
\begin{equation*}
X_E=\coprod_{\mathfrak{p}\in\Spec E}X_{E/\mathfrak{p}}.
\end{equation*}
This assignment extends to a functor from $\mathbf{FEt}(\Spec F)$ to the category $\mathbf{CH}/X_F$ of compact Hausdorff spaces over~$X_F$. To give its action on morphisms, note that
\begin{equation*}
\begin{split}
\Hom_F(\Spec E_1,\Spec E_2)&\cong\Hom_F\left( \coprod_{\mathfrak{p}_1\in\Spec E_1}\Spec E_1/\mathfrak{p}_1,\coprod_{\mathfrak{p}_2\in\Spec E_2}\Spec E_2/\mathfrak{p}_2\right)\\ &\cong\prod_{\mathfrak{p}_1}\coprod_{\mathfrak{p}_2}\Hom_F(\Spec E_1/\mathfrak{p}_1,\Spec E_2/\mathfrak{p}_2)\\
&\cong\prod_{\mathfrak{p}_1}\coprod_{\mathfrak{p}_2}\Hom_F(E_2/\mathfrak{p}_2,E_1/\mathfrak{p}_1)
\end{split}
\end{equation*}
and similarly
\begin{equation*}
\begin{split}
\Hom_{X_F}(X_{E_1},X_{E_2})&\cong\Hom_{X_F}\left( \coprod_{\mathfrak{p}_1\in\Spec E_1}X_{E_1/\mathfrak{p}_1},\coprod_{\mathfrak{p}_2\in\Spec E_2}X_{E_2/\mathfrak{p}_2}\right)\\ &\cong\prod_{\mathfrak{p}_1}\coprod_{\mathfrak{p}_2}\Hom_{X_F}(X_{E_1/\mathfrak{p}_1},X_{E_2/\mathfrak{p}_2}).
\end{split}
\end{equation*}
By piecing together morphisms the obvious way, it therefore suffices to give a continuous map $X_E\to X_F$ for a finite field extension $E/F$. For this, choose an algebraic closure $\widebar{E} /E$, which is then also an algebraic closure of $F$, and let $X_E\to X_F$ be the forgetful map
\begin{equation*}
X_E=\Gal (\widebar{E} /E)\backslash X_{\widebar{E} }\to \Gal (\widebar{E} /F)\backslash X_{\widebar{E}}=X_F.
\end{equation*}
\begin{Theorem}\label{Theorem:TopologicalEtaleFundamentalGroupOfCHSIsGalois}
	The functor $\mathbf{FEt}(\Spec F)\to\mathbf{CH}/X_F$ sending an \'etale covering space $\Spec E\to\Spec F$ to the map $X_E\to X_F$ in fact has image in $\mathbf{FCov}(X_F)$ and defines an equivalence of categories $\mathbf{FEt}(\Spec F)\to \mathbf{FCov}(X_F)$.
\end{Theorem}
\begin{proof}
	Let first $E/F$ be a finite Galois extension. Then $X_E\to X_F$ is the quotient map for the action of the finite group $\Gal (E/F)$ on $X_E$, which is free by Proposition~\ref{Prop:SimplePropertiesOfCHXF}. Since the spaces involved are compact Hausdorff spaces, $X_E\to X_F$ is then a finite covering, and $\Gal (E/F)=\Aut_{\Spec F}(\Spec E)\to\Aut_{X_F}(X_E)$ is an isomorphism of finite groups.
	
	By passing to a Galois closures, we find that for finite but not necessarily Galois field extensions $E/F$ the map $X_E\to X_F$ is still a finite covering. From this, the same statement for finite \'etale $F$-algebras follows formally.
	
	Hence the functor $\mathbf{FEt}(\Spec F)\to\mathbf{FCov}(X_F)$ is well-defined. That it is fully faithful can again be reduced to the case of automorphisms of a Galois object, which we have already seen. Finally, its essential surjectivity follows from the combination of Propositions~\ref{Lem:SimplyConnectedProfiniteCovIsUniversal} and~\ref{Prop:SimplePropertiesOfCHXF}.(iii).
\end{proof}

\subsection{The relation between $X_F$ and $\mathcal{X}_F$}\label{section:RelationXFCalXF}

Recall that for a scheme $\mathcal{X}$ of finite type over $\bbc$ there is a canonical topology on $\mathcal{X}(\bbc )$, called the \emph{complex topology}, cf.\ \cite{MR0082175} and \cite[Expos\'e~XII]{MR2017446}. Here $\mathcal{X}(\bbc )$ designates the set of sections of the structural morphism $\mathcal{X}\to\Spec\bbc$ (rather than all scheme morphisms $\Spec\bbc\to\mathcal{X}$); we hope that no confusion with the usage in section~\ref{sec:GaloisAsEtaleFGOfSchemes} will arise. The complex topology is characterised by the following properties:
\begin{enumerate}
	\item The complex topology on $\bba^1(\bbc )\cong\bbc$ is the Euclidean topology, i.e.\ the metric topology on~$\bbc$ induced by the metric $d(z,w)=\lvert z-w\rvert$.
	\item The complex topology on $(\mathcal{X}\times_{\Spec\bbc } \mathcal{Y})(\bbc )\cong \mathcal{X}(\bbc )\times \mathcal{Y}(\bbc )$ is the product topology defined by the complex topologies on $\mathcal{X}(\bbc )$ and $\mathcal{Y}(\bbc )$.
	\item If $\mathcal{Y}\hookrightarrow \mathcal{X}$ is an open, resp.\ closed, embedding, then so is $\mathcal{Y}(\bbc )\hookrightarrow \mathcal{X}(\bbc )$.
\end{enumerate}
In particular, for a quasiprojective variety $\mathcal{X}\subseteq\bbp^n_{\bbc }$ the complex topology on $\mathcal{X}(\bbc )$ is the subspace topology induced by the Euclidean topology on $\bbp^n(\bbc)$.

The complex topology can easily be generalised to arbitrary $\bbc$-schemes; its description is facilitated when restricted to affine $\bbc$-schemes, which suffices for our purposes.

Let $A$ be a $\bbc$-algebra, so that $\mathcal{X}=\Spec A$ becomes an affine $\bbc$-scheme. Note that $\mathcal{X}(\bbc )$ can be identified with the set of $\bbc$-algebra homomorphisms $A\to\bbc$. We can interpret an element $f\in A$ as a function $f\colon\mathcal{X}(\bbc )\to \bbc$ by sending $\mathcal{X}(\bbc )\ni x\colon A\to\bbc$ to $x(f)$, i.e.\ by writing $f(x)$ for what formally is $x(f)$.

This way we define for every $f\in A$ a subset $U_f\subseteq \mathcal{X}(\bbc )$ by
	\begin{equation*}
	U_f=\{ x\in \mathcal{X}(\bbc )\mid \lvert f(x)\rvert <1 \} .
	\end{equation*}
\begin{Definition}\label{def:complextop}
	Let $\mathcal{X}=\Spec A$ be an affine $\bbc$-scheme. The \emph{complex topology} on $\mathcal{X}(\bbc )$ is the unique topology for which $\{ U_f\mid f\in A \}$ is a subbasis of open sets.
\end{Definition}
That is, a subset of $\mathcal{X}(\bbc )$ is open if it can be written as a union of subsets of the form $U_{f_1,\dotsc ,f_n}=U_{f_1}\cap\dotsb\cap U_{f_n}$ (for possibly varying $n$).
\begin{Proposition}\label{Prop:ComplexTopologyComparedWithSchemes}
	\begin{enumerate}
		\item If $\mathcal{X}$ is an affine $\bbc$-scheme of finite type, the above definition of the complex topology on $\mathcal{X}(\bbc )$ agrees with the classical one.
		\item If $\mathcal{X}=\varprojlim_{\alpha }\mathcal{X}_{\alpha }$ is a cofiltered limit of affine $\bbc$-schemes $\mathcal{X}_{\alpha }$ of finite type, then $\mathcal{X}(\bbc )=\varprojlim_{\alpha }\mathcal{X}_{\alpha }(\bbc )$ as topological spaces.
		\item The complex topology is compatible with fibre products: $(\mathcal{X}\times_{\mathcal{S}}\mathcal{Y})(\bbc )=\mathcal{X}(\bbc )\times_{\mathcal{S}(\bbc )}\mathcal{Y}(\bbc )$ as topological spaces.
		\item If $\mathcal{Y}\to \mathcal{X}$ is a finite \'etale covering, then $\mathcal{Y}(\bbc )\to \mathcal{X}(\bbc )$ is a finite covering.
	\end{enumerate}
\end{Proposition}
\begin{proof}
	\begin{enumerate}
		\item Easy.
		\item Note that we can write $\mathcal{X}=\Spec A$ and $\mathcal{X}_{\alpha }=\Spec A_{\alpha }$ where $A=\varinjlim_{\alpha }A_{\alpha }$. Then because the limit is filtered every finite subset of $A$ is the image of a finite subset of some $A_{\alpha }$. Hence every basic open set $U_{f_1,\dotsc ,f_n}\subseteq \mathcal{X}(\bbc )$ is the preimage of an open subset of some $\mathcal{X}_{\alpha }(\bbc )$.
		\item By (i) this holds if $\mathcal{X}$, $\mathcal{Y}$ and $\mathcal{S}$ are of finite type over~$\bbc$. Apply (ii) to deduce the general case from this.
		\item Again this is well-known if $\mathcal{X}$ (and then automatically also $\mathcal{Y}$) is of finite type over $\bbc$. It follows e.g.\ from \cite[Expos\'e XII, Propositions 3.1.(iii) and 3.2.(vi)]{MR2017446}. Every \'etale covering of an arbitrary $\mathcal{X}$ is the pullback from an \'etale covering of a $\bbc$-scheme of finite type, cf.\ \cite[Tag 00U2, item (9)]{StacksProject}. Apply this, (ii) and (iii) to deduce the general case.
		\qedhere
	\end{enumerate}
\end{proof}
\begin{Remark}
	Note that for a nonempty affine $\bbc$-scheme $\mathcal{X}$ the space $\mathcal{X}(\bbc )$ may well be empty, e.g.\ for $\mathcal{X}=\Spec \bbc (T)$.
\end{Remark}
For each field $F$ containing $\bbq (\zeta_{\infty })$ we may consider $\mathcal{X}_F(\bbc )$, and from now on we will tacitly assume it to be endowed with its complex topology.
\begin{Corollary}
	Let $F$ be a field extension of $\bbq (\zeta_{\infty })$. Then $\mathcal{X}_F(\bbc )$ is connected. For a finite extension $E/F$ the map $\mathcal{X}_E(\bbc )\to \mathcal{X}_F(\bbc )$ is a covering map of degree $[E:F]$.
\end{Corollary}

\begin{proof} We only need to prove that $\mathcal{X}_F(\bbc)$ is connected. This reduces to the case that $F$ is algebraically closed, where it follows from Proposition~\ref{Prop:ComplexTopologyComparedWithSchemes} and Theorem~\ref{Thm:CompatibilityBetweenXFSpaceAndXFScheme}~(i).
\end{proof}

This suggests that also $\pioneet (\mathcal{X}_F(\bbc ))\cong\Gal (\Fbar /F)$; the difficulty here is to show that every topological finite covering of $\mathcal{X}_F(\bbc )$ comes from an \'etale covering of the scheme~$\mathcal{X}_F$. An analogous statement in the case of $\bbc$-schemes of finite type is known under the name `Riemann's existence theorem', cf.\ \cite[Expos\'e~XII, Th\'eor\`eme~5.1]{MR2017446}. It is, however, not known to the authors in which generality this can hold for affine schemes of infinite type over $\bbc$.

However, we can now show in a roundabout way that indeed $\pioneet (\mathcal{X}_F(\bbc ))\cong\Gal (\Fbar /F)$.
\begin{Theorem}\label{Thm:CompatibilityBetweenXFSpaceAndXFScheme}
	Let $F$ be a field containing $\bbq (\zeta_{\infty })$, and let $\Fbar /F$ be an algebraic closure.
	\begin{enumerate}
	\item There is a canonical $\Gal (\Fbar /F)$-equivariant homeomorphism
	\begin{equation*}
	\mathcal{X}_{\Fbar }(\bbc )\overset{\cong }{\to }X_{\Fbar }\times\Hom (\Fbar^{\times },\bbr ),
	\end{equation*}
	where the second factor denotes the set of group homomorphisms $\Fbar^{\times }\to\bbr$, endowed with the compact-open topology (the topology on $\Fbar^{\times }$ being discrete).
	\item The homeomorphism from (i) induces a continuous map $\mathcal{X}_F(\bbc )\to X_F$, which is a deformation retraction. Each fibre of this map is homeomorphic to $\Hom (\Fbar^{\times },\bbr )$.
	\item The diagram
	\begin{equation*}
	\xymatrix{
		\mathbf{FEt}(\Spec F) \ar[r]^-{\varPhi_{\mathrm{sch}} } \ar[d]_-{\varPhi_{\mathrm{top}}} & \mathbf{FEt}(\mathcal{X}_F) \ar[d]^-{\mathcal{Y}\mapsto\mathcal{Y}(\bbc )}\\
		\mathbf{FCov}(X_F) \ar[r]_-{\varPsi} & \mathbf{FCov} (\mathcal{X}_F(\bbc ))
		}
	\end{equation*}
	commutes up to isomorphism of functors, and all functors in it are equivalences of categories. Here the two functors $\varPhi_{\mathrm{sch}}$ and $\varPhi_{\mathrm{top}}$ are those from Theorems \ref{Theorem:CoversOfXFandExtensionsOfF} and \ref{Theorem:TopologicalEtaleFundamentalGroupOfCHSIsGalois}, respectively, and $\varPsi$ is induced by the map $\mathcal{X}_F(\bbc )\to X_F$ in (ii).
	\end{enumerate}
\end{Theorem}
\begin{proof}\begin{enumerate}
		\item Note that there are canonical bijections
		\begin{equation*}\begin{split}
		\mathcal{X}_{\Fbar }(\bbc )&\cong\Hom_{\text{$\bbc$-algebras}}(\Wrat (\Fbar)\otimes_{\bbz [\mu_{\infty }]}\bbc ,\bbc )\\
		&\cong\Hom_{\text{$\bbc$-algebras}}(\bbz [\Fbar^{\times }]\otimes_{\bbz [\mu_{\infty }]}\bbc ,\bbc)\\
		&\cong \Hom_{\text{$\bbc$-algebras}}(\bbc [\Fbar^{\times }]\otimes_{\bbc [\mu_{\infty }]}\bbc ,\bbc )\\
		&\cong\{ \chi\in\Hom_{\text{groups}}(\Fbar^{\times },\bbc^{\times })\mid \chi |_{\mu_{\infty }}=\id_{\mu_{\infty }} \} .
		\end{split}
		\end{equation*}
		Now $\bbc^{\times }\cong\bbs^1\times\bbr$ as Lie groups, and from this we obtain a product decomposition
		\begin{equation*}
		\mathcal{X}_{\Fbar }(\bbc )\cong\{ \chi\in\Hom (\Fbar^{\times },\bbs^1)\mid\chi |_{\mu_{\infty}}=\id_{\mu_{\infty}} \}\times \Hom (\Fbar^{\times },\bbr ).
		\end{equation*}
		The first factor is equal to~$X_{\Fbar}$. This bijection is clearly equivariant for $\Gal (\Fbar /F)$, and it is straightforward to show that it is a homeomorphism.
		\item There is a $\Gal (\Fbar /F)$-equivariant deformation retraction $\mathcal{X}_{\Fbar }(\bbc )\to X_{\Fbar}$: using the description of $\mathcal{X}_{\Fbar }(\bbc )$ from (i) we can define it as
		\begin{equation*}
		\begin{split}
		H\colon X_{\Fbar}\times\Hom (\Fbar^{\times },\bbr )\times [0,1]&\to X_{\Fbar}\times\Hom (\Fbar^{\times },\bbr )\\
		(\chi ,\lambda ,t)&\mapsto (\chi ,t\lambda).
		\end{split}
		\end{equation*}
		By equivariance, $H$ descends to a deformation retraction $\mathcal{X}_F(\bbc )\to X_F$ (note that $X_F$ is the quotient of $X_{\Fbar }$ by the Galois group by construction, and $\mathcal{X}_F(\bbc )$ is the quotient by the Galois group by Lemma~\ref{Lem:GeometricPointsOfXF}). The statement about the fibres follows from the fact that $\Gal (\Fbar /F)$ operates freely on $X_{\Fbar }$.
		\item That the diagram commutes up to isomorphism of functors is a direct calculation. We already know that three of the functors are equivalences: $\varPhi_{\mathrm{sch}}$ is an equivalence by Theorem~\ref{Theorem:CoversOfXFandExtensionsOfF}, $\varPhi_{\mathrm{top}}$ is an equivalence by Theorem~\ref{Theorem:TopologicalEtaleFundamentalGroupOfCHSIsGalois}, and $\varPsi$ is an equivalence by (ii) and Corollary~\ref{Cor:HomotopyEquivalenceInducesIsoOnPioneet}. Hence the fourth functor is also an equivalence.\qedhere
	\end{enumerate}
\end{proof}

\newpage

\section{Classical fundamental groups inside Galois groups}

\subsection{Path components of the spaces~$X_F$}
First let $\Fbar$ be an algebraically closed field. To determine the set of path components of $X_{\Fbar }$ we need to contemplate a large commutative diagram.
\begin{Lemma}
	Let $\Fbar$ be an algebraically closed field containing $\bbq (\zeta_{\infty })$. Then there is a commutative diagram with exact rows and columns:
	\begin{equation}\label{eqn:HugeDiagramReal}
	\xymatrix{
		& 0\ar[d] & 0\ar[d] & 0\ar[d] \\
		0\ar[r] & \Hom (\Fbar^{\times }_{\mathrm{tf}},\bbr )\ar[r]\ar[d] & \Hom (\Fbar^{\times }_{\mathrm{tf}},\bbs^1)\ar[r]\ar[d] &\Ext (\Fbar^{\times }_{\mathrm{tf}},\bbz )\ar[r]\ar[d] & 0\\
		0\ar[r] & \Hom (\Fbar^{\times },\bbr )\ar[r]\ar[d] & \Hom (\Fbar^{\times },\bbs^1)\ar[r]\ar[d] &\Ext (\Fbar^{\times },\bbz )\ar[r]\ar[d] & 0\\
		 & 0\ar[r] & \Hom (\mu_{\infty },\bbs^1)\ar[r]\ar[d] &\Ext (\mu_{\infty },\bbz )\ar[r]\ar[d] & 0\\
		&  & 0 & 0
	}
	\end{equation}
	Here, all $\Hom$ and $\Ext$ groups are understood to be in the category of abelian groups.
\end{Lemma}
\begin{proof}
	This diagram is essentially obtained by applying the bifunctor $\Hom (-,-)$ and its derivative $\Ext (-,-)$ to the short exact sequences
	\begin{equation*}
	0\to\mu_{\infty }\to F^{\times }\to F\to 0
	\end{equation*}
	in the first variable and
	\begin{equation*}
	0\to \bbz \to\bbr \to\bbs^1\to 0
	\end{equation*}
	in the second variable. Hence it commutes by functoriality. We now show exactness in the rows and columns.
	
	For the rows, let $A$ be one of the groups $\mu_{\infty }$, $\Fbar^{\times }$ and $\Fbar^{\times}_{\mathrm{tf}}$. Then we obtain an exact sequence
	\begin{equation}\label{eqn:FirstSesInProofOfLemmaBigCDExact}
	\begin{split}
	\Hom (A,\bbz )&\to\Hom (A,\bbr )\to\Hom (A,\bbs^1)\to\\
	\Ext (A,\bbz )&\to\Ext (A,\bbr ).
	\end{split}
	\end{equation}
	In all three cases $A$ is divisible, therefore the first term $\Hom (A,\bbz )$ vanishes; the last term $\Ext (A,\bbr )$ vanishes because $\bbr $ is divisible. In the case $A=\mu_{\infty }$ the term $\Hom (A,\bbr )$ is also trivial. Hence the exact sequences (\ref{eqn:FirstSesInProofOfLemmaBigCDExact}) for these choices of $A$ can be identified with the rows of~(\ref{eqn:HugeDiagramReal}).
	
	As to the columns, the exactness of the first column is trivial. The second column is the exact sequence
	\begin{equation*}
	0\to\Hom (\Fbar^{\times}_{\mathrm{tf}},\bbs^1)\to\Hom (\Fbar^{\times },\bbs^1)\to\Hom (\mu_{\infty },\bbs^1)\to\Ext (\Fbar^{\times}_{\mathrm{tf}},\bbs^1)
	\end{equation*}
	where the last term is zero because $\bbs^1$ is divisible. Finally the third column is the exact sequence
	\begin{equation*}
	\Hom (\mu_{\infty },\bbz )\to\Ext (\Fbar^{\times}_{\mathrm{tf}},\bbs^1)\to\Ext (\Fbar^{\times },\bbs^1)\to\Ext (\mu_{\infty },\bbs^1)\to 0,
	\end{equation*}
	which is exact because $\Ext^2$ is zero on the category of abelian groups, and whose first member $\Hom (\mu_{\infty },\bbz )$ is clearly trivial.
\end{proof}
Recall that by Proposition~\ref{Prop:PiZeroPathPontryaginDual} there is a canonical bijection between $\pizeropath ((\Fbar^{\times })^{\vee })$  and $\Ext (\Fbar^{\times },\bbz )$. Denote by
\begin{equation*}
\Ext_{\exp }(\Fbar^{\times },\bbz )\subset\Ext (\Fbar^{\times },\bbz )
\end{equation*}
the subset of those extensions whose restriction to $\mu_{\infty }\subset\Fbar^{\times }$ is isomorphic to the exponential sequence
\begin{equation*}
0\to\bbz \to\bbq \to\mu_{\infty }\to 0,
\end{equation*}
i.e.\ the preimage of the element in $\Ext (\mu_{\infty },\bbz )$ encoding this extension in $\Ext (\Fbar^{\times },\bbz )$.
\begin{Proposition}
	Let $\Fbar\supset\bbq (\zeta_{\infty })$ be an algebraically closed field. The subset $X_{\Fbar }\subset (\Fbar^{\times})^{\vee }$ is the union of those path components corresponding to the subset $\Ext_{\exp }(\Fbar^{\times },\bbz )\subset\Ext (\Fbar^{\times },\bbz )\cong\pizeropath ((\Fbar^{\times })^{\vee })$.
\end{Proposition}
\begin{proof}
	The middle column in (\ref{eqn:HugeDiagramReal}) can be rewritten as
	\begin{equation*}
	0\to (\Fbar^{\times }_{\mathrm{tf}})^{\vee }\to (\Fbar^{\times })^{\vee }\to\mu_{\infty }^{\vee}\to 0,
	\end{equation*}
	and $X_F$ is equal to the preimage of the inclusion $\iota\in\Hom (\mu_{\infty },\bbs^1)=\mu_{\infty }^{\vee }$. The horizontal map $(\Fbar^{\times })^{\vee }=\Hom (\Fbar^{\times },\bbs^1)\to\Ext (\Fbar^{\times },\bbz )$ occurring in (\ref{eqn:HugeDiagramReal}) is precisely the one taking a point in $(\Fbar^{\times })^{\vee }$ to its path component. The exactness and commutativity of~(\ref{eqn:HugeDiagramReal}) imply that $X_F$ can equally be described as the preimage of the exponential sequence
	\begin{equation*}
	[0\to\bbz \to\bbq \to\mu_{\infty }\to 0]\in\Ext (\mu_{\infty },\bbz )
	\end{equation*}
	in $(\Fbar^{\times })^{\vee }$.
\end{proof}
\begin{Corollary}\label{Cor:PiZeroPathOfXFbar}
	There is a canonical $\Gal (\Fbar /F)$-equivariant bijection between the sets $\pizeropath (X_{\Fbar})$ and $\Ext_{\exp }(\Fbar^{\times },\bbz )$.\hfill $\square$
\end{Corollary}
By Theorem~\ref{Thm:CompatibilityBetweenXFSpaceAndXFScheme}.(i) there is then also a $\Gal (\Fbar /F)$-equivariant bijection between $\pizeropath (\mathcal{X}_{\Fbar }(\bbc ))$ and $\Ext_{\exp }(\Fbar^{\times },\bbz )$.

\subsection{Multiplicatively free fields}\label{ssn:MultiplicativelyFreeFields} We will now investigate conditions for $X_F$, and therefore also $\mathcal{X}_F(\bbc )$, to be path-connected. First we note that for a large class of fields it cannot be path-connected.
\begin{Proposition}\label{Prop:ManyRootsUncountablyManyComp}
	Let $F\supseteq\bbq (\zeta_{\infty })$ be a field, and assume that there exists an element $\alpha\in F^{\times }$ which is not a root of unity, such that for infinitely many $n\in\bbn$ there exists an $n$-th root of $\alpha$ in~$F$.
	
	Then $X_F$ has uncountably many path components.
\end{Proposition}
For example, $X_F$ has uncountably many path components for $F=\bbq (\zeta_{\infty },p^{1/\infty })$ with $p$ a rational prime.
\begin{proof}
	Let $M\subset F^{\times }$ be the smallest saturated subgroup of $F^{\times }$ containing $\alpha$; then $\mu_{\infty }\subset M$, and $M_{\mathrm{tf}}=M/\mu_{\infty }$ is a torsion-free abelian group of rank one, hence it can be embedded into the additive group of~$\bbq$. By construction it has unbounded denominators, i.e.\ it is not contained in $\frac 1n\bbz$ for any $n\in\bbn$.
	
	The inclusion $M\hookrightarrow \Fbar^{\times }$ defines a continuous surjection $X_{\Fbar }\to\Hom_{\exp }(M,\bbs^1)$ which is clearly $\Gal (\Fbar /F)$-equivariant, hence we obtain a continuous surjection $X_{\Fbar }\to\Hom_{\exp }(M,\bbs^1)$, where the target is (non-canonically) homeomorphic to $M^{\vee }$. We therefore obtain a surjection $\pizeropath (X_F)\to\pizeropath (M^{\vee })$.
	
	By Proposition~\ref{Prop:PiZeroPathPontryaginDual} there is a bijection $\pizeropath (M^{\vee })\cong\Ext (M,\bbz )$. Hence we need to show that $\Ext (M,\bbz )$ is uncountable.
	
	We may assume that $\bbz\subset M$. Then there is a short exact sequence
	\begin{equation*}
	0\to\bbz \to M\to M/\bbz \to 0
	\end{equation*}
	and hence a long exact sequence
	\begin{equation}\label{eqn:SESUncountableExt}
	\underbrace{\Hom (M,\bbz)}_{=0}\to\underbrace{\Hom (\bbz ,\bbz )}_{=\bbz }\to\Ext (M/\bbz ,\bbz )\to\Ext (M,\bbz )\to\underbrace{\Ext (\bbz ,\bbz)}_{=0}.
	\end{equation}
	Letting $\mathfrak{N}$ be the set of cyclic subgroups $N\subset M$ containing $\bbz$ we can then write $M=\varinjlim_{N\in \mathfrak{N}}N$, hence we obtain a spectral sequence with
	\begin{equation*}
	E_2^{p,q}=\mathrm{R}^p\varprojlim_{N\in\mathfrak{N}}\Ext^q(N/\bbz ,\bbz )\Rightarrow\Ext^{p+q}(M/\bbz ,\bbz) .
	\end{equation*}
	Since the $\Ext^q(N,\bbz )$ satisfy a Mittag-Leffler condition the higher limits vanish and the spectral sequence degenerates at $E_2$. We therefore obtain an isomorphism
	\begin{equation*}
	\Ext (M/\bbz ,\bbz )\cong\varprojlim_{N\in\mathfrak{N}}\Ext (N/\bbz ,\bbz ).
	\end{equation*}
	This is an inverse limit over infinitely many finite groups with surjective transition maps, hence an uncountable profinite group. By (\ref{eqn:SESUncountableExt}) $\Ext (M,\bbz )\cong\pizeropath (M^{\vee})$ must then also be uncountable.
\end{proof}
Here is a class of fields where an $\alpha$ as in Proposition~\ref{Prop:ManyRootsUncountablyManyComp} cannot occur. 
\begin{Definition}
	Let $F$ be a field containing $\bbq (\zeta_{\infty })$. Then $F$ is called \emph{multiplicatively free} if $F^{\times}_{\mathrm{tf}}=F^{\times }/\mu_{\infty}$ is a free abelian group. $F$ is called \emph{stably multiplicatively free} if every finite extension of $F$ is multiplicatively free.
\end{Definition}
\begin{Remark}
	\begin{altenumerate}
		\item It is unknown to the authors whether there exists a field which is multiplicatively free but not stably multiplicatively free.
		\item Recall that a subgroup $B$ of an abelian group $A$ is \emph{saturated} if whenever $a\in A$ and there is some $n\in\bbn$ with $na\in B$, then $a\in B$. By a result of Pontryagin \cite[Lemma~16]{MR1503168}, a torsion-free abelian group $A$ is free if and only if every finitely generated subgroup of $A$ is contained in a saturated finitely generated subgroup of~$A$. Hence a field $F\supseteq\bbq (\zeta_{\infty })$ is multiplicatively free if and only if every finite subset of $F^{\times }$ is contained in a saturated subgroup of $F^{\times }$ generated (as a group) by all roots of unity and possibly finitely many additional elements.
		\item The condition in Proposition~\ref{Prop:ManyRootsUncountablyManyComp} and the property of being multiplicatively free are mutually exclusive. It is again unknown to the authors whether always one or the other holds. For general abelian torsion-free groups, not necessarily isomorphic to $F^{\times}_{\mathrm{tf}}$ for a field $F$, both can be false.
		
		More precisely, there exists an abelian group $A$ with the following properties: it is  torsion-free; it has rank two (i.e., $A\hookrightarrow A\otimes_{\bbz }\bbq\cong\bbq^2$); for any $a\in A\smallsetminus\{ 0\}$ there are only finitely many $n\in\bbn$ for which there exists $b\in A$ with $a=nb$; it is not free, in fact, every rank one quotient of $A$ is divisible. This group $A$ is constructed in \cite[Lemma~2]{MR0289476}.\footnote{Here is the construction. Let $\varphi\in\End (\bbq /\bbz )\cong\End (\hat{\bbz })\cong\hat{\bbz }\cong\prod_{p\text{ prime}}\bbz_p$ be such that the component $\varphi_p\in\bbz_p$ at each $p$ is transcendental over~$\bbq$. Then we set $A=\{ (a,b)\in\bbq^2\mid \varphi (a\bmod\bbz )=b\bmod z \}$.} 
	\end{altenumerate}
\end{Remark}
\begin{Proposition}
	Let $F$ be an algebraic extension of $\bbq (\zeta_{\infty })$ which can be written as an abelian extension of a finite extension of~$\bbq$. Then $F$ is stably multiplicatively free.
	
	In particular, $\bbq (\zeta_{\infty })$ is stably multiplicatively free.
\end{Proposition}
\begin{proof}
	We first show that $F$ is multiplicatively free.
	
	For a finite set $S$ of rational primes let $\mathfrak{o}_{F,S}$ be the ring of $S$-integers in~$F$. By a result of May \cite[Theorem]{MR0258786} the group $\mathfrak{o}_{F,S}^{\times }/\mu_{\infty }$ is then free abelian. Note that the free abelian subgroups $\mathfrak{o}_{F,S}^{\times }/\mu_{\infty }$ are saturated in $F^{\times }/\mu_{\infty }$, and every finitely generated subgroup of $F^{\times }/\mu_{\infty }$ is contained in one of them. By \cite[Lemma]{MR0258786} $F^{\times }/\mu_{\infty }$ must then be free abelian itself. Hence $F$ is multiplicatively free.
	
	Now let $E/F$ be a finite extension; we can write $E=F(\alpha )$ for some $\alpha\in E$. By assumption there exists a subfield $K\subset F$ which is finite over $\bbq$ such that the extension $F/K$ is abelian. Then $E/K(\alpha )$ is abelian, too. By what we have just shown $E$ is therefore multiplicatively free.
\end{proof}

\subsubsection*{The Kummer pairing} For a field $F$ containing $\bbq (\zeta_{\infty })$ and an algebraic closure $\Fbar /F$, let $F^{\times}_{\mathrm{sat}}$ be the saturation of $F^{\times }$ in $\Fbar^{\times}$, i.e.\ the group of all $\alpha\in\Fbar^{\times }$ such that there exists some $n\in\bbn$ with $\alpha^n\in F^{\times }$. Note that $F_{\mathrm{sat}}^{\times }$ is divisible, hence $F_{\mathrm{sat}}^{\times }/\mu_{\infty }=(F_{\mathrm{sat}}^{\times })_{\mathrm{tf}}$ is a $\bbq$-vector space. Then there is a canonical biadditive pairing
\begin{equation*}
\langle\cdot ,\cdot\rangle\colon \Gal (\Fbar /F)\times F_{\mathrm{sat}}^{\times}\to\mu_{\infty },\quad (\sigma ,\alpha)\mapsto \langle\sigma ,\alpha\rangle =\frac{\sigma (\alpha )}{\alpha }.
\end{equation*}
It clearly factors through $\Gal (F^{\mathrm{ab}}/F)\times F_{\mathrm{sat}}^{\times }/F^{\times }$. By Kummer theory, cf.\ \cite[section~VI.8]{MR1878556}, the maximal abelian extension $F^{\mathrm{ab}}$ is obtained by adjoining all elements of $F_{\mathrm{sat}}^{\times }$ to~$F$, and the resulting homomorphism
\begin{equation}\label{eqn:KummerIsomorphismGlobal}
\kappa\colon\Gal (F(F_{\mathrm{sat}}^{\times})/F)\to\Hom (F_{\mathrm{sat}}^{\times}/F^{\times},\mu_{\infty }),\quad\sigma\mapsto\langle\sigma ,\cdot\rangle
\end{equation}
is an isomorphism.

Now assume in addition that $F$ is stably multiplicatively free. We wish to understand the action of $\Gal (\Fbar /F)$ on $\Ext_{\exp}(\Fbar^{\times},\bbz )$.

We begin by considering a saturated subgroup $V\subset\Fbar^{\times }$ such that $\mu_{\infty }\subseteq V$, $V_{\mathrm{tf}}=V/\mu_{\infty }$ is a $\bbq$-vector space of finite rank, $V$ is stable under the Galois action, and $V$ is the saturation of $\varLambda =V\cap F^{\times }$. Note that then $\varLambda_{\mathrm{tf}}=\varLambda /\mu_{\infty }$ is a $\bbz$-lattice in~$V$.

Then the Galois action on $V$ admits the following description: for $\alpha\in V$ and $\sigma\in\Gal (\Fbar /F)$ we obtain
\begin{equation*}
\sigma (\alpha )=\langle\sigma ,\alpha\rangle\cdot\alpha .
\end{equation*}
Then there is also a natural action of $\Gal (\Fbar /F)$ on $\Hom (V,\mu_{\infty })$, namely $\sigma (\chi )=\chi\circ\sigma^{-1}$, that is,
\begin{equation*}
\sigma (\chi)(\alpha)=\chi (\sigma^{-1}\alpha )=\chi (\langle\sigma^{-1},\alpha\rangle\cdot\alpha)=\chi (\langle \sigma ,\alpha\rangle )^{-1}\cdot\chi (\alpha).
\end{equation*}
Hence for the subset
\begin{equation*}
\Hom_{\exp }(V,\mu_{\infty })=\{ \chi\in\Hom (V,\mu_{\infty })\mid\chi |_{\mu_{\infty}}=\mathrm{id}_{\mu_{\infty}} \},
\end{equation*}
which is a translate of the subgroup $\Hom (V_{\mathrm{tf}},\mu_{\infty })\subset\Hom (V,\mu_{\infty })$, we obtain a particularly simple description of the Galois action: for $\chi\in\Hom_{\exp }(V,\mu_{\infty })$ and $\sigma\in\Gal (\Fbar /F)$ we have
\begin{equation}\label{eqn:GaloisActionOnHomExpVMuInfty}
\sigma (\chi)=\langle\sigma ,\cdot\rangle^{-1}\cdot\chi =\kappa_V (\sigma)^{-1}\cdot\chi ,
\end{equation}
where $\kappa_V$ is the isomorphism $\Gal (F(V)/F)\to\Hom (V/\varLambda ,\mu_{\infty })$ induced by (\ref{eqn:KummerIsomorphismGlobal}).

\begin{Lemma}\label{Lem:OpenSGWhereGaloisActsUnipotentlyThroughKummer}
	Let $F\supseteq\bbq (\zeta_{\infty })$ be a stably multiplicatively free field, let $\Fbar /F$ be an algebraic closure, and let $V\subset\Fbar^{\times }$ be a Galois-stable saturated subgroup of finite rank.
	
	Then there exist
	\begin{itemize}
		\item an open subgroup $H\subseteq\Gal (\Fbar /F)$,
		\item an open compact subgroup $L\subseteq\Hom (V_{\mathrm{tf}},\mu_{\infty })\cong\Hom(V,\adf )$ and
		\item a surjective continuous group homomorphism $\kappa\colon H\to L$
	\end{itemize}
	such that $H\subseteq\Gal (\Fbar /F)$ operates on $\Hom_{\exp }(V,\mu_{\infty })$ by $\tau (\chi )=\chi +\kappa (\tau )$ (where the group structure on $\mu_{\infty }$ is written additively).
\end{Lemma}
\begin{proof}
	Let $B\subset V$ be a finite set which maps to a $\bbq$-basis of $V_{\mathrm{tf}}$, and let $E$ be the subfield of $\Fbar$ generated by~$B$. Then $E/F$ is a finite extension, hence $E$ is multiplicatively free. Let $\varLambda =V\cap E^{\times }$; this is a subgroup of $V$ containing $\mu_{\infty }$, and the quotient $\varLambda_{\mathrm{tf}}=\varLambda /\mu_{\infty }$ is a full $\bbz$-lattice in~$V_{\mathrm{tf}}$.
	
	We set $L=\Hom (V/\varLambda ,\mu_{\infty })\subset\Hom (V_{\mathrm{tf}},\mu_{\infty })$; under the isomorphism
	\begin{equation*}
	\Hom (V_{\mathrm{tf}},\mu_{\infty })\cong\Hom (V,\adf )\cong\Hom (\varLambda,\adf )
	\end{equation*}
	it corresponds to the subgroup $\Hom (\varLambda,\hat{\bbz })$, which is clearly open and compact.
	
	We then set $H=\Gal (\Fbar /E)$ and let
	$\kappa\colon\Gal (\Fbar /E)\to\Hom (V/\varLambda ,\mu_{\infty })$ be the map induced by the Kummer pairing, as in (\ref{eqn:KummerIsomorphismGlobal}). From (\ref{eqn:GaloisActionOnHomExpVMuInfty}) we see that the $H$-action on $\Hom_{\exp }(V,\mu_{\infty })$ is indeed as described.
\end{proof}
\begin{Lemma}\label{Lem:TransitiveGaloisActionOnHomExpVMuInfty}
	Let $F\supseteq\bbq (\zeta_{\infty })$ be a stably multiplicatively free field, let $\Fbar /F$ be an algebraic closure and let $V\subset\Fbar^{\times }$ be a Galois-stable saturated subgroup of finite rank.
	
	Then $\Gal (\Fbar /F)$ operates transitively on $\Ext_{\exp }(V,\bbz )$.
\end{Lemma}
\begin{proof}
	Using the short exact sequences $0\to\bbz \to\bbq \to\mu_{\infty }\to 0$ and $0\to\mu_{\infty }\to V\to V_{\mathrm{tf}}\to 0$ we obtain a commutative diagram with exact rows and columns analogous to (\ref{eqn:HugeDiagramReal}):
	\begin{equation*}
	\xymatrix{
		& 0\ar[d] & 0\ar[d] & 0\ar[d] \\
		0\ar[r] & \Hom (V_{\mathrm{tf}},\bbq )\ar[r]\ar[d] & \Hom (V_{\mathrm{tf}},\mu_{\infty })\ar[r]\ar[d] &\Ext (V_{\mathrm{tf}},\bbz )\ar[r]\ar[d] & 0\\
		0\ar[r] & \Hom (V,\bbq )\ar[r]\ar[d] & \Hom (V,\mu_{\infty})\ar[r]\ar[d] &\Ext (V,\bbz )\ar[r]\ar[d] & 0\\
		& 0\ar[r] & \Hom (\mu_{\infty },\mu_{\infty})\ar[r]\ar[d] &\Ext (\mu_{\infty },\bbz )\ar[r]\ar[d] & 0.\\
		&  & 0 & 0
	}
	\end{equation*}
	Here $\Ext_{\exp }(V,\bbz )$ is the preimage of the exponential extension $\varepsilon_{\exp }\in\Ext (\mu_{\infty },\bbz )$ in $\Ext (V,\bbz )$, hence a translate of the subgroup $\Ext (V_{\mathrm{tf}},\bbz )$. It can therefore also be described as the quotient of $\Hom_{\exp }(V,\mu_{\infty })$, a translate of $\Hom (V_{\mathrm{tf}},\mu_{\infty})$ in $\Hom (V,\mu_{\infty } )$, by the subgroup $\Hom (V,\bbq )$.
	
	From Lemma~\ref{Lem:OpenSGWhereGaloisActsUnipotentlyThroughKummer} we see that there is an open subgroup of $\Hom (V_{\mathrm{tf}},\mu_{\infty })$ on whose translates in $\Hom_{\exp }(V,\mu_{\infty })$ a suitable open subgroup $H$ of the Galois group acts transitively. Since the subgroup $\Hom (V_{\mathrm{tf}},\bbq )$ is dense in $\Hom (V_{\mathrm{tf}},\mu_{\infty })\cong\Hom (V_{\mathrm{tf}},\adf )$ this implies that $H$, and therefore also $\Gal (\Fbar /F)$, operates transitively on $\Ext_{\exp }(V,\bbz )$.
\end{proof}
\begin{Lemma}\label{Lem:InductionStepInGaloisTransitivity}
	Let $V\subseteq W\subset\Fbar^{\times }$ be Galois-stable saturated subgroups of finite rank. Let also $\varepsilon_1,\varepsilon_2\in\Ext_{\exp }(\Fbar^{\times },\bbz )$, and let $\chi_1^{(V)},\chi_2^{(V)}\in\Hom (V,\mu_{\infty})$ satisfying the following conditions:
	\begin{enumerate}
		\item the connecting homomorphism $\delta\colon\Hom (V,\mu_{\infty})\to\Ext (V,\bbz )$ induced by the exponential sequence sends $\chi_i^{(V)}$ to $\varepsilon_i\rvert_V$, for $i=1,2$;
		\item $\chi_1^{(V)}$ and $\chi_2^{(V)}$ lie in the same $\Gal (\Fbar /F)$-orbit.
	\end{enumerate}
	Then there exist elements $\chi_1^{(W)},\chi_2^{(W)}\in\Hom (W,\mu_{\infty})$ such that
	\begin{enumerate}
	\setcounter{enumi}{2}
		\item the connecting homomorphism $\delta\colon\Hom (W,\mu_{\infty})\to\Ext (W,\bbz )$ sends $\chi_i^{(W)}$ to $\varepsilon_i\rvert_W$, for $i=1,2$;
		\item $\chi_i^{(W)}\rvert_V=\chi_i^{(V)}$, for $i=1,2$;
		\item $\chi_1^{(W)}$ and $\chi_2^{(W)}$ lie in the same $\Gal (\Fbar /F)$-orbit.
	\end{enumerate}
\end{Lemma}

\begin{proof}
	Again we contemplate a large commutative diagram, obtained from  $0\to V\to W\to W/V\to 0$ and the exponential sequence:
	\begin{equation*}
	\xymatrix{
		& 0\ar[d] & 0\ar[d] & 0\ar[d] \\
		0\ar[r] & \Hom (W/V,\bbq )\ar[r]\ar[d] & \Hom (W/V,\mu_{\infty})\ar[r]\ar[d] & \Ext (W/V,\bbz )\ar[r]\ar[d] & 0\\
		0\ar[r] & \Hom (W,\bbq )\ar[r]\ar[d] & \Hom (W,\mu_{\infty})\ar[r]\ar[d] & \Ext (W,\bbz )\ar[r]\ar[d] & 0\\
		0\ar[r] & \Hom (V,\bbq )\ar[r]\ar[d] & \Hom (V,\mu_{\infty})\ar[r]\ar[d] & \Ext (V,\bbz )\ar[r]\ar[d] & 0\\
		& 0 & 0 & 0
		}
	\end{equation*}
	First we consider the problem of lifting an individual character. So, let $\varepsilon\in\Ext (\Fbar^{\times },\bbz )$ and $\chi^{(V)}\in\Hom (V,\mu_{\infty })$ with $\delta (\chi^{(V)})=\varepsilon\rvert_V$. A short diagram chase then shows that there exists some $\chi^{(W)}\in\Hom (W,\mu_{\infty })$ with $\delta (\chi^{(W)})=\varepsilon\rvert_W$ and $\chi^{(W)}\rvert_V=\chi^{(V)}$, and that moreover the set of all such $\chi^{(W)}$ is a translate of $\Hom (W/V,\bbq )$ (which can be considered a subgroup of $\Hom (W,\mu_{\infty})$).
	
	With this preparation we find an element $\chi_1^{(W)}$ satisfying (iii) and (iv). We choose some $\sigma\in\Gal (\Fbar /F)$ with $\sigma (\chi_1^{(V)})=\chi_2^{(V)}$, and let $\psi =\sigma (\chi_1^{(V)})$. Then $\psi$ is our first approximation to $\chi_2^{(W)}$, and it clearly satisfies (iv) and (v), but not necessarily (iii).
	
	From Lemma~\ref{Lem:OpenSGWhereGaloisActsUnipotentlyThroughKummer} we deduce the existence of the following objects:
	\begin{itemize}
		\item a closed subgroup $H\subseteq\Gal (\Fbar /F)$,
		\item an open compact subgroup $L\subseteq\Hom (W/V,\mu_{\infty })\cong\Hom (W/V,\adf )$ and
		\item a surjective continuous group homomorphism $\kappa\colon H\to L$
	\end{itemize}
	such that on $\Hom_{\exp }(V,\mu_{\infty })$ the group $H\subseteq\Gal (\Fbar /F)$ operates by $\tau (\chi )=\chi +\kappa (\tau )$ (here we write the group structure in $\mu_{\infty}$ additively).
	
	Since $\varepsilon_2\rvert_W$ and $\delta (\psi)$ both restrict to $\varepsilon_2\rvert_V\in\Ext (V,\bbz )$ their difference lies in $\Ext (W/V,\bbz )$. Hence there exists an element $\alpha\in\Hom (W/V,\mu_{\infty })\cong\Hom (W/V,\adf )$ such that $\delta (\alpha )=\varepsilon_2\rvert_W-\delta (\psi)$. We let further $\tau\in H\subseteq\Gal (\Fbar /F)$ with $\kappa (\tau )=\alpha$.
	
	Finally we set $\chi_2^{(W)}=\psi +\alpha$. We check that all desired conditions are met:
	\begin{enumerate}
		\setcounter{enumi}{2}
		\item $\chi_1^{(W)}$ was chosen such that $\delta (\chi_1^{(W)})=\varepsilon_2\rvert_W$, and
		\begin{equation*}
		\delta (\chi_2^{(W)})=\delta (\psi )+\delta (\alpha ) =\delta (\psi )+\varepsilon_2\rvert_W-\delta (\psi )=\varepsilon_2\rvert_W.
		\end{equation*}
		\item $\chi_1^{(W)}$ was chosen such that $\chi_1^{(W)}\rvert_V=\chi_1^{(V)}$, and
		\begin{equation*}
		\chi_2^{(W)}\rvert_V=\psi\rvert_V+\alpha\rvert_V=\chi_2^{(V)}+0=\chi_2^{(V)}.
		\end{equation*}
		\item $\chi_1^{(W)}$ and $\chi_2^{(W)}$ are in the same Galois orbit because
		\begin{equation*}
		\tau\sigma (\chi_1^{(W)})=\tau (\psi )=\psi +\kappa (\tau )=\psi +\alpha =\chi_2^{(W)}.\qedhere
		\end{equation*}
	\end{enumerate}
\end{proof}
\begin{Proposition}\label{Prop:StablyMultFreeImpliesTransitiveGaloisOnExt}
	Let $F\supseteq\bbq (\zeta_{\infty })$ be a \emph{countable} stably multiplicatively free field, and let $\Fbar /F$ be an algebraic closure.
	
	Then $\Gal (\Fbar /F)$ operates transitively on~$\Ext_{\exp }(\Fbar^{\times },\bbz )$.
\end{Proposition}
\begin{proof}
	By countability we can find an ascending chain (indexed by $\bbn$) of Galois-stable subgroups $\mu_{\infty }\subset V_1\subset V_2\subset\cdots$ of finite rank, whose union is~$\Fbar^{\times }$. Hence there is a spectral sequence with
	\begin{equation}\label{eqn:SpecSequExtLimits}
	E_2^{p,q}=\mathrm{R}^p\varprojlim_n\Ext^q(V_n,\bbz )\Rightarrow\Ext^{p+q}(\Fbar^{\times },\bbz ).
	\end{equation}
	Since these $\Ext$ groups are taken in the category of abelian groups, the entries with $q>1$ vanish. Likewise, it is easy to see that the structure maps $\Hom (V_n,\bbz )\to\Hom (V_m,\bbz )$ and $\Ext (V_n,\bbz )\to\Ext (V_m,\bbz )$ for $n\ge m$ are surjective, hence the inverse systems in (\ref{eqn:SpecSequExtLimits}) satisfy a Mittag-Leffler condition, and the higher direct images also vanish. Therefore the spectral sequence degenerates at $E_2$, and the natural map
	\begin{equation*}
	\Ext (\Fbar^{\times },\bbz )\to\varprojlim_n\Ext (V_n,\bbz )
	\end{equation*}
	is an isomorphism. We deduce that the restriction
	\begin{equation}\label{eqn:ContinuityOfExtExp}
	\Ext_{\exp } (\Fbar^{\times },\bbz )\to\varprojlim_n\Ext_{\exp } (V_n,\bbz )
	\end{equation}
	is a bijection.
	
	Let $\varepsilon_1,\varepsilon_2\in\Ext_{\exp }(\Fbar^{\times },\bbz )$. We need to show that there is a $\sigma\in\Gal (\Fbar /F)$ with $\sigma (\varepsilon_1)=\varepsilon_2$; by what we have just seen this is equivalent to $\sigma (\varepsilon_1\vert_{V_n})=\varepsilon_2\rvert_{V_n}$ for all $n\in\bbn$. Using Lemmas \ref{Lem:TransitiveGaloisActionOnHomExpVMuInfty} and \ref{Lem:InductionStepInGaloisTransitivity} we inductively produce elements $\chi_1^{(n)},\chi_2^{(n)}\in\Hom (V_n,\mu_{\infty })$ such that the following conditions hold:
	\begin{enumerate}
	\setcounter{enumi}{2}
		\item the connecting homomorphism $\delta\colon\Hom (V_n,\mu_{\infty })\to\Ext (V_n,\bbz )$ sends $\chi_i^{(n)}$ to $\varepsilon_i\rvert_{V_n}$, for $i=1,2$ and $n\in\bbn$;
		\item $\chi_i^{(n+1)}\rvert_{V_n}=\chi_i^{(n)}$ for $i=1,2$ and $n\in\bbn$;
		\item $\chi_1^{(n)}$ and $\chi_2^{(n)}$ lie in the same $\Gal (\Fbar /F)$-orbit.
	\end{enumerate}
	We let
	\begin{equation*}
	T_n=\{ \sigma\in\Gal (\Fbar /F)\mid \sigma (\chi_1^{(n)})=\chi_2^{(n)} \} .
	\end{equation*}
	By (v) each $T_n$ is a nonempty subset of $\Gal (\Fbar /F)$, and by (iv) the sequence of subsets $(T_n)$ is descending, i.e.\ $T_n\supseteq T_{n+1}$ for all $n\in\bbn$. Moreover, $\Gal (\Fbar /F)$ operates continuously on $\Hom (V_n,\mu_{\infty })$ when the latter is endowed with the compact-open topology; since $\Hom (V_n,\mu_{\infty })$ is Hausdorff, points in this space are closed, hence the $T_n$ are closed subsets of $\Gal (\Fbar /F)$.
	
	Now we can apply Cantor's Intersection Theorem (Lemma~\ref{Lem:CantorIntersection} above) to conclude that
	\begin{equation*}
	\bigcap_{n\in\bbn }T_n\neq\varnothing .
	\end{equation*}
	By construction, any element in this intersection sends $\varepsilon_1$ to $\varepsilon_2$.
\end{proof}
\begin{Remark}
	The reader may wonder why we do not simply proceed as follows to prove Proposition~\ref{Prop:StablyMultFreeImpliesTransitiveGaloisOnExt}. For any Galois-stable saturated subgroup $V\subset\Fbar^{\times }$ of finite rank we set $U_V=\{ \sigma\in\Gal (\Fbar /F)\mid\sigma (\varepsilon_1\rvert_V)=\varepsilon_2\rvert_V \}$; this is a coset of the stabiliser of $\varepsilon_1\rvert_V$, and it is nonempty by Lemma~\ref{Lem:TransitiveGaloisActionOnHomExpVMuInfty}. Then an application of Lemma~\ref{Lem:CantorIntersection} should show that the intersection of all $U_V$ is nonempty.
	
	The problem with this argument is that $U_V$ is not closed, only an $F_{\sigma}$-subset (a countable union of closed subsets). Rigidifying the situation by adding the auxiliary conditions that the lifts $\chi_i^{(n)}\in\Hom (V,\mu_{\infty })$ also be fixed replaces the $U_V$ by the closed subsets~$T_V$, which allows us to apply Lemma~\ref{Lem:CantorIntersection}.
\end{Remark}
\begin{Corollary}\label{Cor:SMFImpliesPathConnected}
	Let $F$ be a countable stably multiplicatively free field. Then $X_F$ and $\mathcal{X}_F(\bbc )$ are path-connected.
\end{Corollary}
\begin{proof}
	By Corollary~\ref{Cor:PiZeroPathOfXFbar} and Proposition~\ref{Prop:StablyMultFreeImpliesTransitiveGaloisOnExt}, $\Gal (\Fbar /F)$ operates transitively on $\pizeropath (X_{\Fbar })$, hence $X_F=\Gal (\Fbar /F)\backslash X_{\Fbar }$ is path-connected. By Theorem~\ref{Thm:CompatibilityBetweenXFSpaceAndXFScheme}.(ii) then also $\mathcal{X}_F(\bbc )$ is path-connected.
\end{proof}

\subsection{Classical fundamental groups of the spaces~$X_F$}

 It will be convenient to fix a basepoint $\tilde{\chi}\in X_{\Fbar }$ and denote its image in $X_F$ by~$\chi$. Then $p\colon (X_{\Fbar },\tilde{\chi})\to (X_F,\chi)$ is a universal profinite covering space. Recall that there is then a short exact sequence of abstract groups (\ref{eqn:SesThreeFGAlsoOnRight}) which in our case becomes
\begin{equation*}
1\to \pionepath (X_{\Fbar },\tilde{\chi})\overset{p_{\ast }}{\to }\pionepath (X_F,\chi )\overset{\alpha}{\to }\operatorname{Stab}_{\pioneet (X_F,\chi )}X_{\Fbar }^{\circ}\to 1.
\end{equation*}
Since $X_{\Fbar }$ is homeomorphic to the Pontryagin dual of a $\bbq$-vector space its classical fundamental group $\pionepath (X_{\Fbar },\tilde{\chi})$ is trivial by Corollary~\ref{Cor:PontryaginDualHasTrivialPionepath}, hence $\alpha$ maps $\pionepath (X_F,\chi )$ isomorphically to the stabiliser. The latter can be rewritten: the set $\pizeropath (X_{\Fbar })$ is in canonical bijection with $\Ext_{\exp }(\Fbar^{\times },\bbz )$, and this bijection is equivariant for the isomorphism $\pioneet (X_F,\chi )\cong\Gal (\Fbar /F)$. Hence we have shown the following:
\begin{Proposition}\label{prop:descrpionepath}
	Let $F\supseteq\bbq (\zeta_{\infty })$ be a field, and let $\tilde{\chi }\in X_{\Fbar }$. Denote the image of $\tilde{\chi }$ in $X_F$ by $\chi$, and let $\varepsilon\in\Ext_{\exp }(\Fbar^{\times },\bbz )$ be the pullback of the extension
	\begin{equation*}
	[0\to\bbz \to\bbr \to\bbs^1\to 0]\in\Ext (\bbs^1,\bbz )
	\end{equation*}
	along~$\tilde{\chi }$.
	
	Then $\pionepath (X_F,\chi )$ is canonically isomorphic to the stabiliser of $\varepsilon$ in $\Gal (\Fbar /F)$.\hfill $\square$
\end{Proposition}
This stabiliser seems to be hard to determine in general. However, in the countable stably multiplicatively free case we can at least say that it is large.
\begin{Proposition}
	Let $F\supseteq\bbq (\zeta_{\infty })$ be a countable stably multiplicatively free field, and let $\chi\in X_F$ be any basepoint. Then the image of $\pionepath (X_F,\chi )$ is a dense subgroup of $\Gal (\Fbar /F)$.
\end{Proposition}
\begin{proof}
	If $F$ is countable and stably multiplicatively free then so is any finite extension $E$ of~$F$. Hence all coverings $X_E$ defined by finite extensions of $F$ are path-connected by Corollary~\ref{Cor:SMFImpliesPathConnected}. By Theorem~\ref{Theorem:TopologicalEtaleFundamentalGroupOfCHSIsGalois} every connected covering space of $X_F$ is of this form, therefore $X_F$ is stably path-connected. Hence by Proposition~\ref{Prop:StablyPathConnectedThenAlphaSurjective} the image of $\pionepath (X_F,\chi )$ in $\pioneet (X_F,\chi )\cong\Gal (\Fbar /F)$ is dense.
\end{proof}
This applies in particular to $F=\bbq (\zeta_{\infty })$.

\subsubsection*{The fundamental group as an inverse limit}\label{subsec:FundGroup} Let $F\supseteq\bbq (\zeta_{\infty })$ be a countable stably multiplicatively free field. We shall write $\pionepath (X_F)$ as an inverse limit of discrete groups which are extensions of finite groups by free abelian groups of finite rank. In particular, $\pionepath (X_F)$ will be endowed with a non-discrete topology.

Fix an algebraic closure $\Fbar /F$, and let $\mathfrak{L}(\Fbar /F)$ denote the set of all subgroups $\varLambda <\Fbar^{\times }$ satisfying the following conditions:
\begin{enumerate}
	\item $\varLambda$ contains $\mu_{\infty }$, and $\varLambda_{\mathrm{tf}}=\varLambda /\mu_{\infty }$ is a free abelian group of finite rank.
	\item $\varLambda$ is stable under $\Gal (\Fbar /F)$.
	\item Set $E=F(\varLambda )$; this is a finite Galois extension of~$F$ by (i) and (ii). Also, let $V$ be the saturation of $\varLambda$ in~$\Fbar^{\times }$. Then $E^{\times }\cap V=\varLambda$.
	\item For every $\sigma\in\Gal (E/F)$ there exists a $\lambda\in\varLambda$ such that $\sigma (\lambda )/\lambda =\zeta_n$, where $n$ is the order of $\sigma$ in $\Gal (E/F)$.
\end{enumerate}
\begin{Lemma}
	Let $F\supseteq\bbq (\zeta_{\infty })$ be a stably multiplicatively free field with algebraic closure~$\Fbar$. Then $\Fbar^\times$ can be written as the filtered union
	\begin{equation*}
	\Fbar^{\times }=\bigcup_{\varLambda\in\mathfrak{L} (\Fbar /F)}\varLambda .
	\end{equation*}
\end{Lemma}
\begin{proof}
	Let $\alpha_1,\ldots,\alpha_n\in\Fbar^{\times }$; we need to find a $\varLambda\in\mathfrak{L}(\Fbar /F)$ with $\alpha_1,\ldots,\alpha_n\in\varLambda$.
	
	First let $E$ be the Galois closure of $F(\alpha_1,\ldots,\alpha_n)$ in~$\Fbar$. By Hilbert's Theorem 90 in the form already used in the proof of Lemma~\ref{Lem:MitHilbert90}, for every $\sigma\in\Gal (E/F)$ there exists some $\lambda_{\sigma }\in E^{\times }$ such that $\sigma (\lambda_{\sigma })/\lambda_{\sigma }=\zeta_n$. Let $V$ be the smallest saturated subgroup of $\Fbar^{\times }$ containing the $\lambda_{\sigma }$ and all Galois conjugates of~$\alpha_1,\ldots,\alpha_n$. Then $V_{\mathrm{tf}}$ is a $\bbq$-vector space of finite rank, and $\varLambda =E^{\times }\cap V$ will be an element of $\mathfrak{L}(\Fbar /F)$ containing~$\alpha_1,\ldots,\alpha_n$.
\end{proof}
For $\varLambda\in\mathfrak{L}(\Fbar /F)$ set $X_{\Fbar }(\varLambda )=\Hom_{\exp }(\varLambda,\bbs^1)$. There is a natural continuous action of $\Gal (\Fbar /F)$ on $X_{\Fbar }(\varLambda)$, and we set $X_F(\varLambda )=\Gal (\Fbar /F)\backslash X_{\Fbar }(\varLambda )$.
\begin{Lemma}
	Let $F\supseteq\bbq (\zeta_{\infty })$ be a stably multiplicatively free field, and let $\Fbar /F$ be an algebraic closure.
	\begin{enumerate}
		\item For every $\varLambda\in\mathfrak{L}(\Fbar /F)$ the space $X_{\Fbar }(\varLambda )$ is homeomorphic to a torus of dimension $\operatorname{rank}\varLambda$. For any basepoint $\chi\in X_{\Fbar }(\varLambda )$ there is a canonical isomorphism
		\begin{equation*}
		\pionepath (X_{\Fbar },\chi )\cong\Hom (\varLambda,\bbz ).
		\end{equation*}
		\item Let $\varLambda\in\mathfrak{L}(\Fbar /F)$ and let $E$ be the subfield of $\Fbar$ generated by~$\varLambda$, a finite Galois extension of~$F$. Then the action of $\Gal (\Fbar /F)$ on $X_{\Fbar}(\varLambda )$ factors through $\Gal (E/F)$, and the induced action of $\Gal (E/F)$ is free. Hence $X_{\Fbar }(\varLambda )\to X_F(\varLambda )$ is a finite covering.
		\item For any basepoint $\chi\in X_F(\varLambda )$ there is a natural exact sequence
		\begin{equation*}
		1\to\Hom (\varLambda ,\bbz )\to\pionepath (X_F(\varLambda ),\chi )\to\Gal (E/F)\to 1.
		\end{equation*}
	\end{enumerate}
\end{Lemma}
\begin{proof}
	For (i) note that $X_{\Fbar }(\varLambda )$ is a translate of the subgroup $\varLambda_{\mathrm{tf}}^{\vee }$ in $\varLambda^{\vee }$; this subgroup is a torus whose classical fundamental group is canonically isomorphic to $\Hom (\varLambda ,\bbz )$.
	
	As to (ii), it is clear that the Galois action factors through $\Gal (E/F)$. We will now show that the induced action of this finite group is free. Let $1\neq\sigma\in\Gal (E/F)$ be an element of order $n>1$, and assume that there is some $\chi\in X_{\Fbar }(\varLambda )$ with $\sigma (\chi )=\chi$. By condition (iv) in the definition of $\mathfrak{L}(\Fbar /F)$ there exists some $\lambda\in\varLambda$ with $\sigma (\lambda )/\lambda =\zeta_n$. But then
	\begin{equation*}
	\mathrm{e}^{2\pi\mathrm{i}/n}=\chi (\zeta_n)=\frac{\chi (\sigma (\lambda ))}{\chi (\lambda )}=\frac{\sigma (\chi )(\lambda )}{\chi (\lambda )}=\frac{\chi (\lambda )}{\chi (\lambda )}=1,
	\end{equation*}
	a contradiction. Therefore $\sigma$ cannot have a fixed point in $X_{\Fbar }(\varLambda )$, and the action is free.
	
	Part (iii) then follows easily (note that these spaces are path-connected and locally path-connected, in fact manifolds, so the classical theory of fundamental groups and covering spaces applies).
\end{proof}
\begin{Proposition}
	Let $F\supseteq\bbq (\zeta_{\infty })$ be a stably multiplicatively free field with algebraic closure~$\Fbar /F$. Then the canonical map
	\begin{equation}\label{eqn:XFAsALimitOfXFLambda}
	X_F\to\varprojlim_{\varLambda\in\mathfrak{L}(\Fbar /F)}X_F(\varLambda )
	\end{equation}
	is a homeomorphism.
\end{Proposition}
\begin{proof}
	Consider first the map $X_{\Fbar }\to\varprojlim X_F(\varLambda )=\varprojlim \Gal (\Fbar /F)\backslash X_{\Fbar }(\varLambda )$. This is clearly surjective, and if two elements of $X_{\Fbar }$ have the same image they must be in the same Galois orbit, by an argument using Lemma~\ref{Lem:CantorIntersection} (Cantor's Intersection Theorem) similar to that used in the proof of Proposition~\ref{Prop:StablyMultFreeImpliesTransitiveGaloisOnExt}. Hence the map (\ref{eqn:XFAsALimitOfXFLambda}) is bijective. It is also continuous, and domain and target are compact Hausdorff spaces. Therefore it is a homeomorphism.
\end{proof}
\begin{Proposition}
	Let $F\supseteq\bbq (\zeta_{\infty })$ be a countable stably multiplicatively free field with algebraic closure~$\Fbar$. Choose a basepoint $\chi\in X_F$, and for each $\varLambda\in\mathfrak{L}(\Fbar /F)$ denote its image in $X_F(\varLambda )$ by~$\chi_{\varLambda}$. Then the natural map
	\begin{equation}\label{eqn:PionepathXFAsInverseLimit}
	\pionepath (X_F,\chi )\to\varprojlim_{\varLambda\in\mathfrak{L}(\Fbar /F)}\pionepath (X_F(\varLambda) ,\chi_{\varLambda})
	\end{equation}
	is an isomorphism.
\end{Proposition}
\begin{proof}
	Since $F$ (hence also $\Fbar$) is countable, there exists a cofinal sequence $(\varLambda_n)_{n\in\bbn }$ in $\mathfrak{L}(\Fbar /F)$. To see this, choose an enumeration $\Fbar =\{ a_1,a_2,a_3,\ldots  \}$ and choose the $\varLambda_n$ inductively in such a way that $\varLambda_n\subseteq\varLambda_{n+1}$ and $a_1,\ldots ,a_n\in\varLambda_n$. Hence $X_F\to\varprojlim_nX_F(\varLambda_n)$ is
	is also a homeomorphism, and it suffices to show that
	\begin{equation*}
	\pionepath (X_F,\chi )\to\varprojlim_{n\in\bbn }\pionepath (X_F(\varLambda_n ),\chi_{\varLambda_n})
	\end{equation*}
	is an isomorphism.
	
	For each $n\in\bbn$ there is a commutative diagram
	\begin{equation*}
	\xymatrix{
		X_{\Fbar }(\varLambda_{n+1})\ar[d]\ar[r]^-{\tilde{p}} & X_{\Fbar }(\varLambda_n)\ar[d]\\
		X_F(\varLambda_{n+1})\ar[r]_-p & X_F(\varLambda_n)
		}
	\end{equation*}
	where $\tilde{p}$ is a fibration and the vertical maps are finite coverings. Therefore $p$ is also a fibration.
	
	In general if $(X_n)_{n\in\bbn }$ is a projective system of pointed topological spaces where the transition maps are fibrations, then there is a short exact sequence
	\begin{equation*}
	1\to\mathrm{R}^1\varprojlim_{n\in\bbn }\pi_2^{\mathrm{path}}(X_n)\to\pionepath (X)\to\varprojlim_{n\in\bbn }\pionepath (X_n)\to 1
	\end{equation*}
	of abstract groups (compatible choice of basepoints understood), see \cite{MR0346781} (see also \cite[Theorem~2.1]{Hirschhorn2015} for a more elementary exposition). Hence in our case there is a short exact sequence
	\begin{equation*}
	1\to\mathrm{R}^1\varprojlim_{n\in\bbn }\pi_2^{\mathrm{path}}(X_F(\varLambda_n))\to\pionepath (X_F)\to\varprojlim_{n\in\bbn }\pionepath (X_F(\varLambda_n))\to 1.
	\end{equation*}
	Since the $X_F(\varLambda_n)$ admit finite covering spaces which are tori, their second homotopy groups vanish.
\end{proof}
\begin{Proposition}
	The loop topology turns $\pionepath (X_F,\chi )$ into a topological group with a basis of open neighborhoods of the identity given by open subgroups, hence it is equal to the $\tau$- and $\sigma$-topologies. If we endow each $\pionepath (X(\varLambda ),\chi_{\varLambda })$ with the discrete topology, then (\ref{eqn:PionepathXFAsInverseLimit}) becomes an isomorphism of topological groups.
	
	Moreover, $\pionepath (X_F,\chi )$ is complete for this topology. Therefore $\pionepath (X_F,\chi )\cong\pi_1^{\Gal }(X_F,\chi )$ is a Noohi group.
\end{Proposition}
\begin{proof}
	Since the $X_F(\varLambda )$ are manifolds, their classical fundamental groups are discrete for the loop topology. Consider the commutative diagram
	\begin{equation*}
	\xymatrix{
		\Omega (X_F,\chi )\ar[r]\ar[d] & \varprojlim\limits_{\varLambda\in\mathfrak{L}(\Fbar /F)}\Omega (X_F(\varLambda ),\chi_{\varLambda })\ar[d]\\
		\pionepath (X_F,\chi )\ar[r]_-{(\ref{eqn:PionepathXFAsInverseLimit})} &\varprojlim\limits_{\varLambda\in\mathfrak{L}(\Fbar /F)}\pionepath (X_F(\varLambda ),\chi_{\varLambda}).
		}
	\end{equation*}
	Here the upper horizontal map is a homeomorphism and the vertical maps are open. Hence (\ref{eqn:PionepathXFAsInverseLimit}) is a bijection which is continuous and open, hence also a homeomorphism.
	
	A quasi-topological group which is a projective limit of topological groups is itself a topological group, and hence the loop and $\tau$-topologies on $\pionepath (X_F,\chi )$ agree. Since $\pionepath (X,\chi )$ is a projective limit of discrete groups its $\tau$- and $\sigma$-topologies agree, and it is complete.
\end{proof}
In particular we find that the loop topology turns $\pionepath (X_F)$, which we may identify with a subgroup of $\Gal (\Fbar /F)$, into a complete topological group whose topology is strictly finer than the subspace topology induced from the Krull topology on $\Gal (\Fbar /F)$, because it has infinite discrete quotients.

\section{Cohomology}

\noindent We will next show how to realise Galois cohomology groups with constant coefficients as suitable cohomology groups of the spaces $X_F$ and the schemes~$\mathcal{X}_F$.

\subsection{The Cartan--Leray spectral sequence} Consider the following situation: $X$ is a compact Hausdorff space and $G$ is a profinite group operating freely and continuously on~$X$, with quotient space~$Y=G\backslash X$. We will construct a spectral sequence relating the (sheaf) cohomologies of $X$ and $Y$ with the continuous group cohomology of~$G$.

\subsubsection*{Continuous group cohomology} Let $G$ be a profinite group. A \emph{continuous $G$-module} is an abelian group $A$ with a $G$-operation which becomes continuous when $A$ is endowed with the discrete topology. The continuous $G$-modules form an abelian category $\GMod$ in an obvious way; this category has enough injectives. The functor $(-)^G\colon\GMod\to\mathbf{Ab}$ sending a $G$-module $A$ to its invariant submodule $A^G$ is left exact. Hence we obtain a total derived functor between derived categories $\mathrm{R}(-)^G\colon\mathcal{D}(\GMod )\to\mathcal{D}(\mathbf{Ab})$, and derived functors in the classical sense which we call \emph{continuous group cohomology}:
\begin{equation*}
\Hup^p(G,A)=\Rup^p(-)^G (A).
\end{equation*}
Note that this may well differ from the classical group cohomology $\Hup^p(G^{\delta },A)$ where $G^{\delta }$ is $G$ as an abstract group. However, there is a canonical isomorphism
\begin{equation*}
\varinjlim_H\Hup^p(G/H,A^H)\cong\Hup^p(G,A),
\end{equation*}
where the limit is over all normal open subgroups $H$ of $G$, cf.\ the discussion in \cite[section~2.2]{MR0180551}.

For a field $F$ with separable closure $\Fbar /F$ and a continuous $\Gal (\Fbar /F)$-module $A$ we write shortly
\begin{equation*}
\Hup^m(F,A)=\Hup^m(\Gal (\Fbar /F),A);
\end{equation*}
these groups are called \emph{Galois cohomology groups}. Note that if $A$ is an abelian group interpreted as a constant module for the Galois group, then $\Hup^m(F,A)$ does not depend on the choice of a separable closure of $F$, i.e.\ for another separable closure $\Fbar '/F$ there is a canonical isomorphism $\Hup^m(\Gal (\Fbar '/F),A)\cong\Hup^m(\Gal (\Fbar /F),A)$; this justifies the notation $\Hup^m(F,A)$.

\subsubsection*{Equivariant sheaves and their cohomology} Let $X$ be a compact Hausdorff space and $G$ a profinite group operating continuously and freely on~$X$. There are several ways to define the category $\Sh_G(X)$ of $G$-equivariant abelian sheaves on~$X$.

For instance, a sheaf of abelian groups $\mathcal{A}$ on $X$ corresponds to an \emph{espace \'etal\'e} $\pi\colon A\to X$, which is a topological space $A$ with a local homeomorphism $\pi\colon A\to X$ and an abelian group structure in the category of $X$-spaces. Here $\pi^{-1}(x)\cong\mathcal{A}_x$ (the stalk), and for an open set $U\subseteq X$ sections of $\mathcal{A}$ on $U$ are the same as continuous sections of the map $\pi^{-1}(U)\to U$. Then we define a \emph{$G$-equivariant sheaf} on $X$ to be a sheaf $\mathcal{A}$ of abelian groups on $X$ together with a lift of the $G$-action on $X$ to a continuous $G$-action on~$A$. For a more `modern' definition that is more amenable to generalisations see e.g.\ \cite[section~1]{MR1620705}.

Again the $G$-equivariant sheaves on $X$ form an abelian category $\Sh_G(X)$ with enough injectives; there is a canonical equivalence $\Sh_G(\ast )\simeq\GMod$, where $\ast$ denotes the one-point space.

For a $G$-equivariant sheaf $\mathcal{A}$ the group $\Gamma (X,\mathcal{A})$ of global sections comes naturally with a continuous $G$-action. Hence we obtain a left exact functor
\begin{equation*}
\Gamma_{X,G}\colon\Sh_G(X)\to\GMod ,\quad\mathcal{A}\mapsto\Gamma (X,\mathcal{A}).
\end{equation*}
For a $G$-equivariant sheaf $\mathcal{A}$ on $X$ we obtain a complex $\Rup\Gamma_{X,G}(\mathcal{A})\in\mathcal{D}(\GMod )$; but we may also forget its $G$-structure and apply the derived functor of the usual global sections functor $\Gamma_X\colon\Sh (X)\to\mathbf{Ab}$ to it.
\begin{Lemma}
	For $\mathcal{A}\in\Sh_G(X)$ the complex of abelian groups underlying $\Rup\Gamma_{X,G}(\mathcal{A})$ (i.e.\ its image in $\mathcal{D}(\mathbf{Ab})$) is canonically isomorphic to the complex $\Rup\Gamma_X(\mathcal{A})$.
\end{Lemma}
In particular, the cohomology groups of either of these complexes become continuous $G$-modules whose underlying abelian groups are the ordinary sheaf cohomology groups $\Hup^p(X,\mathcal{A})$.
\begin{proof}
	This follows directly from the fact, proved in \cite[Corollary~3]{MR1620705}, that the forgetful functor $\Sh_G(X)\to\Sh (X)$ sends injective objects to soft sheaves, hence sends an injective resolution of $\mathcal{A}$ in $\Sh_G(X)$ to an acyclic resolution of $\mathcal{A}$ in~$\Sh (X)$.
\end{proof}
\subsubsection*{Sheaves on the quotient} Let $X$ and $G$ as before, and consider the quotient map $p\colon X\to G\backslash X=Y$. There is a canonical equivalence of abelian categories between $\Sh_G(X)$ and $\Sh(Y)$, which can again be described rather simply in terms of espaces \'etal\'es:

If $\mathcal{B}$ is a sheaf of abelian groups on $Y$ with espace \'etal\'e $B$, then $\pi^{-1}\mathcal{B}$ has a natural $G$-structure since its espace \'etal\'e is the fibre product $B\times_YX$, where $G$ operates on the second factor. Vice versa, if $\mathcal{A}$ is a $G$-equivariant sheaf on $X$ with espace \'etal\'e $A$, we may form the quotient $G\backslash A\to Y$ which is the espace \'etal\'e of a sheaf on~$Y$. It is not hard to see that these two constructions are mutually inverse.
\begin{Proposition}\label{Prop:IdentityOfDerivedFunctors}
	Let $X$ be a compact Hausdorff space and let $G$ be a profinite group acting continuously and freely on~$X$, with quotient $Y=G\backslash X$. Let $A$ be an abelian group, and denote the constant sheaves on $X$ and $Y$ modelled on $A$ by $A_X$ and $A_Y$, respectively; endow $A_X$ with the tautological $G$-operation. Then there is a natural isomorphism in $\mathcal{D}(\mathbf{Ab})$:
	\begin{equation*}
	\Rup (-)^G\left( \Rup\Gamma_{X,G} (A_X) \right)\cong\Rup\Gamma_Y(A_Y).
	\end{equation*}
\end{Proposition}
\begin{proof}
	Consider the following diagram of left exact functors between abelian categories:
	\begin{equation*}
	\xymatrix{
		\Sh (Y) \ar[d]_-{\Gamma_Y}\ar[r]^-{\simeq } & \Sh_G(X) \ar[d]^-{\Gamma_{X,G}} \\
		\mathbf{Ab} & \GMod \ar[l]^-{(-)^G}	
	}
	\end{equation*}
	It is easy to see that it is commutative up to isomorphism of functors.

	The equivalence on the upper horizontal line sends $A_Y$ to $A_X$, as can be seen on their espaces \'etal\'es, which are simply $A\times Y$ and $A\times X$ with $G$ operating trivially on~$A$. The claim then follows by the chain rule for derived functors $\Rup (F\circ G)\cong\Rup F\circ\Rup G$.
\end{proof}
\begin{Corollary}\label{Cor:HochschildSerreTopological}
	With the same assumptions as in Proposition~\ref{Prop:IdentityOfDerivedFunctors} there is a spectral sequence with
	\begin{equation*}
	E_2^{p,q}=\Hup^p(G,\Hup^q(X,A))\Rightarrow\Hup^{p+q}(Y,A),
	\end{equation*}
	where $\Hup^p(G,-)$ denotes continuous group cohomology, and $\Hup^q(X,A)$ and $\Hup^{p+q}(Y,A)$ denote sheaf cohomology.\hfill $\square$
\end{Corollary}
Proposition~\ref{Prop:IdentityOfDerivedFunctors} and Corollary~\ref{Cor:HochschildSerreTopological} have analogues in \'etale cohomology. We content ourselves with stating the analogue of the latter.
\begin{Proposition}\label{Prop:CartanLerayEtale}
	Let $\mathcal{X}\to\mathcal{Y}$ be a pro-\'etale Galois covering of schemes, with a profinite deck transformation group~$G$, and let $A$ be an abelian group. Then there is a natural spectral sequence with
	\begin{equation*}
	E_2^{p,q}=\Hup^p(G,\Hup^q_{\et }(\mathcal{X},A))\Rightarrow\Hup^{p+q}_{\et }(\mathcal{Y},A).
	\end{equation*}
\end{Proposition}
\begin{proof}
	This is shown in \cite[Chapter~III, Remark~2.21.(b)]{MR559531}. Here is a brief summary of the proof.
	
	First we assume that $G$ is finite. Then $\mathcal{X}$ is an object of the small \'etale site of $\mathcal{Y}$ on which $G$ acts by automorphism, hence the functor
	\begin{equation*}
	\Sh (\mathcal{Y}_{\et})\to\GMod ,\quad\mathcal{F}\mapsto\Gamma (\mathcal{X},\mathcal{F}),
	\end{equation*}
	is well-defined. Its composition with the forgetful functor $\GMod\to\mathbf{Ab}$ is the usual global sections functor. Hence we obtain a spectral sequence relating the derived functors of these functors.
	
	We deduce the general case by passing to the limit over all coverings $H\backslash\mathcal{X}\to\mathcal{Y}$ with $H\subseteq G$ an open normal subgroup.
\end{proof}

\subsection{The cohomology of $X_F$ and $\mathcal{X}_F$} We can now compute some cohomology groups of these spaces using the Cartan--Leray spectral sequence. We discuss the topological case in detail, the \'etale case for torsion coefficients is analogous.

We begin by computing the cohomology of $X_F$ when $F$ is algebraically closed.
\begin{Proposition}\label{Prop:CohomologyOfXFBar}
	Let $\Fbar$ be an algebraically closed field containing $\bbq (\zeta_{\infty })$.
	\begin{enumerate}
		\item Let $A$ be an abelian torsion group. Then $\Hup^0(X_{\Fbar },A)=A$ and $\Hup^p(X_{\Fbar },A)=0$ for all $p>0$.
		\item There is a canonical isomorphism of graded algebras
		\begin{equation*}
		\bigoplus_{p\ge 0}\Hup^p(X_{\Fbar },\bbq )\cong\bigoplus_{p\ge 0}\bigwedge\nolimits_{\bbq }^{\!\! p}\Fbar^{\times }_{\mathrm{tf}}.
		\end{equation*}
		\item $\Hup^0(X_{\Fbar },\bbz )=\bbz$, and for each $p>0$ the inclusion $\bbz\hookrightarrow\bbq$ induces an isomorphism $\Hup^p(X_{\Fbar },\bbz )\cong\Hup^p(X_{\Fbar },\bbq )$.
	\end{enumerate}
\end{Proposition}
\begin{proof}
	Recall that $X_{\Fbar }$ is a translate of the subgroup $(\Fbar^{\times }_{\mathrm{tf}})^{\vee }\subset (\Fbar^{\times })^{\vee }$. Hence for any $\chi\in X_{\Fbar }$ we obtain a homeomorphism $t_{\chi }\colon (\Fbar^{\times }_{\mathrm{tf}})^{\vee }\to X_{\Fbar }$ by $t_{\chi }(\omega )=\chi\omega$, and therefore an isomorphism of cohomology groups $t_{\chi}^{\ast }\colon\Hup^{\bullet }(X_{\Fbar },A)\to\Hup^{\bullet }((\Fbar^{\times }_{\mathrm{tf}})^{\vee },A)$. These depend continuously on $\chi\in X_{\Fbar }$, but since $X_{\Fbar }$ is connected, they must be independent of~$\chi$. Hence we obtain a canonical isomorphism $\Hup^{\bullet }(X_{\Fbar },A)\cong\Hup^{\bullet }((\Fbar^{\times }_{\mathrm{tf}})^{\vee },A)$, and the statements follow from Proposition~\ref{Prop:CohomologyOfPontryaginDuals}.
\end{proof}
\begin{Theorem}\label{Thm:CechCohomologyOfXF}
	Let $F$ be a field containing $\bbq (\zeta_{\infty})$.
	\begin{enumerate}
		\item For every abelian torsion group $A$ and every $m\ge 0$ there is a natural isomorphism
		\begin{equation*}
		\Hup^m(X_F,A)\cong\Hup^m(F,A).
		\end{equation*}
		\item For each $m\ge 0$ there are natural isomorphisms
		\begin{equation*}
		\Hup^m(X_F,\bbq )\cong\left( \bigwedge\nolimits_{\bbq }^{\!\! m}\Fbar^{\times}_{\mathrm{tf}}\right)^{\Gal (\Fbar /F)}.
		\end{equation*}
		In low degrees this simplifies to
		\begin{equation*}\Hup^0(X_F,\bbq )=\bbq\qquad\text{and}\qquad\Hup^1(X_F,\bbq )\cong F^{\times }\otimes_{\bbz }\bbq .
		\end{equation*}
		\item The cohomology groups with integral coefficients begin with 
		\begin{equation*}
		\Hup^0(X_F,\bbz )=\bbz\qquad\text{and}\qquad\Hup^1(X_F,\bbz )\cong F^{\times }_{\mathrm{tf}}.
		\end{equation*}
	\end{enumerate}
\end{Theorem}
\begin{proof}
	We consider the Cartan--Leray spectral sequence as in Corollary~\ref{Cor:HochschildSerreTopological} for $X=X_{\Fbar }$, $G=\Gal (\Fbar /F)$ and $Y=X_F$:
	\begin{equation}\label{eqn:CLSSForOurTopologicalSpace}
	E_2^{p,q}=\Hup^p(\Gal (\Fbar /F),\Hup^q(X_{\Fbar },A)\Rightarrow\Hup^{p+q}(X_F,A).
	\end{equation}
	\begin{altenumerate}
		\item If $A$ is torsion then $\Hup^q(X_{\Fbar },A)=0$ for all $q>0$ by Proposition~\ref{Prop:CohomologyOfXFBar}, hence the spectral sequence (\ref{eqn:CLSSForOurTopologicalSpace}) degenerates at $E_2$ and we obtain an isomorphism
		\begin{equation*}
			\Hup^p(\Gal (\Fbar /F),A)\cong\Hup^p(X_F,A).
		\end{equation*}
		\item Consider the spectral sequence (\ref{eqn:CLSSForOurTopologicalSpace}) for $A=\bbq$. All the cohomology groups $\Hup^q(X_{\Fbar},\bbq )$ are $\bbq$-vector spaces, hence have trivial Galois cohomology, so $E_2^{p,q}=0$ for $p\neq 0$. Again the spectral sequence (\ref{eqn:CLSSForOurTopologicalSpace}) degenerates at $E_2$ and we obtain isomorphisms
		\begin{equation*}
		\Hup^0(\Gal (\Fbar /F),\Hup^q(X_{\Fbar },\bbq )\cong\Hup^q(X_F,\bbq ).
		\end{equation*}
		Using Proposition~\ref{Prop:CohomologyOfXFBar}.(ii) we can rewrite this in the desired form. It is clear that $\Hup^0(X_F,\bbq )=\bbq$; for the calculation of $\Hup^1(X_F,\bbq)$ we need that the Galois invariants in $\Fbar^{\times }_{\mathrm{tf}}$ are isomorphic to $F^{\times }\otimes\bbq$. To see this consider the short exact sequence of Galois modules
		\begin{equation*}
		0\to \mu_{\infty }\to \Fbar^{\times }\to\Fbar^{\times }_{\mathrm{tf}}\to 0
		\end{equation*}
		and the associated long exact Galois cohomology sequence
		\begin{equation*}
		0\to\mu_{\infty }\to F^{\times }\to (\Fbar^{\times }_{\mathrm{tf}})^{\Gal (\Fbar /F)}\to\Hom (\Gal (\Fbar /F),\mu_{\infty })\to\cdots
		\end{equation*}
		which shows that the cokernel of the inclusion $F_{\mathrm{tf}}^{\times }\to (\Fbar^{\times }_{\mathrm{tf}})^{\Gal (\Fbar /F)}$ is a torsion group. Hence $F^{\times }\otimes\bbq\to (\Fbar^{\times }_{\mathrm{tf}})^{\Gal (\Fbar /F)}$ must be an isomorphism.
		\item 
		Consider the long exact cohomology sequence for the short exact sequence of coefficient groups $0\to\bbz \to\bbq \to\mu_{\infty }\to 0$:
		\begin{equation*}
		\cdots\to\bbq \twoheadrightarrow\mu_{\infty }\overset{0}{\to }\Hup^1(X_F,\bbz )\to\Hup^1(X_F,\bbq  )\to\Hup^1(X_F,\mu_{\infty })\to\cdots .
		\end{equation*}
		We see that $\Hup^1(X_F,\bbz )$ is the kernel of the map $\Hup^1(X_F,\bbq )\to\Hup^1(X_F,\mu_{\infty })$. By (ii) the domain of this map is isomorphic to $F^{\times }\otimes\bbq\cong (F^{\times}_{\mathrm{sat}})_{\mathrm{tf}}$, by (i) the target is isomorphic to $\Hup^1(\Gal (\Fbar /F),\mu_{\infty })=\Hom (\Gal (\Fbar /F),\mu_{\infty })$. A tedious but straightforward calculation shows that this map
		\begin{equation*}
		(F^{\times }_{\mathrm{sat}})_{\mathrm{tf}}\to\Hom (\Gal (\Fbar /F),\mu_{\infty })
		\end{equation*}
		is given by $\alpha\mapsto\langle -,\alpha\rangle$, where $\langle -,-\rangle$ is the Kummer pairing discussed in section~\ref{ssn:MultiplicativelyFreeFields}. Hence its kernel is precisely $F^{\times }_{\mathrm{tf}}$.\qedhere
	\end{altenumerate}
\end{proof}
\begin{Remark}
	Even for $A=\bbz$ the spectral sequence (\ref{eqn:CLSSForOurTopologicalSpace}) gets somewhat simplified, namely then $E_2^{p,q}=0$ whenever $p\neq 0$ and $q\neq 0$. This is because then $\Hup^q(X_{\Fbar },\bbz )$ is a $\bbq$-vector space by Proposition~\ref{Prop:CohomologyOfXFBar}.(iii), hence all higher Galois cohomology groups for this space vanish.
\end{Remark}

In a similar vein we can identify Galois cohomology with constant torsion coefficients with \'etale cohomology of~$\mathcal{X}_F$:
\begin{Theorem}\label{Thm:EtaleCohomologyOfXFIsGaloisCohomology}
	Let $F\supseteq\bbq (\zeta_{\infty })$ be a field and let $A$ be an abelian torsion group, viewed as a trivial Galois module. Then for every $m\ge 0$ there is a canonical isomorphism
	\begin{equation*}
	\Hup^m_{\et }(\mathcal{X}_F,A)\cong\Hup^m(F,A).
	\end{equation*}
\end{Theorem}
\begin{proof}
	Consider the Cartan--Leray spectral sequence as in Proposition~\ref{Prop:CartanLerayEtale}:
	\begin{equation*}
	E_2^{p,q}=\Hup^p(\Gal (\Fbar /F),\Hup^q_{\et }(\mathcal{X}_{\Fbar },A))\Rightarrow\Hup_{\et }^{p+q}(\mathcal{X}_F,A).
	\end{equation*}
	Since $\mathcal{X}_{\Fbar }\cong\Spec\bbc [\Fbar^{\times }_{\mathrm{tf}}]$ as a scheme, $\Hup^q(\mathcal{X}_{\Fbar},A)=0$ for all $q>0$ by Proposition~\ref{Prop:EtaleCohomologyOfSpecGroupAlgebra}. Hence the spectral sequence degenerates at $E_2$, and the claim follows.
\end{proof}

\subsubsection*{The Galois symbol} Recall that for a field $F$ and an integer $m\ge 0$ the $m$-th \emph{Milnor K-group} $\MK_m(F)$ is defined as the quotient of the $m$-th exterior power $\bigwedge_{\bbz }^mF^{\times }$ by the subgroup generated by all expressions of the form $\alpha\wedge (1-\alpha )\wedge \beta_2\wedge\dotsb\wedge \beta_m$ for $\alpha\in F\smallsetminus\{ 0,1\}$ and $\beta_i\in F^{\times }$. In other words, the graded algebra $\bigoplus_{m\ge 0}\MK_m(F)$ is the quotient of the exterior algebra $\bigwedge^{\bullet }F^{\times }$ by the two-sided homogeneous ideal generated by all $\alpha\wedge (1-\alpha )$ with $\alpha\in F\smallsetminus \{ 0,1 \}$. See \cite{MR0260844} for more information.

The Milnor K-groups are related to the more universal and well-known \emph{Quillen $K$-groups} $\mathrm{K}_m(F)$ as follows. There are canonical isomorphisms $\mathrm{K}_0(F)\cong\bbz$ and $\mathrm{K}_1(F)\cong F^{\times }$. There is therefore a unique multiplicative extension $\bigwedge_{\bbz }^{\bullet }F^{\times }\to\mathrm{K}_{\bullet }(F)$;  it factors degreewise through a homomorphism $\MK_m(F)\to\mathrm{K}_m(F)$. This is an isomorphism for $m=0,1$ by construction and for $m=2$ by Matsumoto~\cite{MR0240214}, see also~\cite[\S 12]{MR0260844}.

For $n$ prime to the characteristic of $F$ there is a canonical homomorphism
\begin{equation*}
\partial\colon F^{\times }\to\Hup^1(F,\bbz /n\bbz (1))=\Hup^1(\Gal (\Fbar /F),\mu_n)
\end{equation*}
that sends $\alpha\in F^{\times }$ to the cohomology class $\partial\alpha$ defined by the crossed homomorphism $\Gal (\Fbar /F)\to\mu_n$ sending $\sigma$ to $\sigma (\sqrt[n]{\alpha })/\sqrt[n]{\alpha }$. Taking cup products this extends to a homomorphism of graded rings
\begin{equation}\label{eqn:MapDefiningGaloisSymbol}
\bigoplus_{m\ge 0}\bigwedge\nolimits^{\!\! m}F^{\times }\to\bigoplus_{m\ge 0}\Hup^m(F,\bbz /n\bbz (m)),\quad \alpha_1\wedge\dotsb\wedge \alpha_m\mapsto \partial \alpha_1\cup\dots\cup\partial\alpha_m.
\end{equation}
For any $\alpha\in F\smallsetminus\{ 0,1\}$ the relation $\partial\alpha\cup\partial (1-\alpha )=0$ holds in $\Hup^2(F,\bbz /n\bbz (2))$. Hence (\ref{eqn:MapDefiningGaloisSymbol}) factors through the Milnor K-groups of $F$, defining the \emph{Galois symbols}
\begin{equation*}
\partial^m\colon\MK_m(F)\to\Hup^m(F,\bbz /n\bbz (m)).
\end{equation*}
The \emph{Bloch--Kato conjecture} \cite[p.~118]{MR849653}, now a theorem due to Voevodsky \cite[Theorem~6.1]{MR2811603}, asserts that for every field $F$, every integer $m\ge 0$ and every $n\in\bbn$ prime to the characteristic of $F$ the induced group homomorphism
\begin{equation*}
\MK_m(F)\otimes\bbz /n\bbz\to\Hup^m(F,\bbz /n\bbz (m))
\end{equation*}
is an isomorphism.

Assume now that $F$ contains~$\bbq (\zeta_{\infty })$; then we may ignore Tate twists. By Theorem~\ref{Thm:CechCohomologyOfXF}.(i) we obtain therefore an isomorphism
\begin{equation}\label{eqn:GaloisSymbolCohomologyXF}
\partial_n^m\colon \MK_m(F)\otimes\bbz /n\bbz\to\Hup^m(X_F,\bbz /n\bbz );
\end{equation}
taking an inductive limit over all $n\in\bbn$ we can also construct an isomorphism
\begin{equation*}
\partial^m_{\infty }\colon\MK_m(F)\otimes\bbq /\bbz\to\Hup^m(X_F,\bbq /\bbz ).
\end{equation*}
Note also that $\partial_n^1$ lifts canonically to an isomorphism $\partial^1\colon \MK_1(F)=F^{\times }\to\Hup^1(X_F,\bbz )$ by Theorem~\ref{Thm:CechCohomologyOfXF}.(iii). However, for $\alpha\in F\smallsetminus\{ 0,1\}$ the element $\partial^1(\alpha )\cup\partial^1(1-\alpha )\in\Hup^2(X_F,\bbz )$ is nonzero since its image in $\Hup^2(X_F,\bbq )$ is nonzero by Theorem~\ref{Thm:CechCohomologyOfXF}.(ii). In particular the resulting homomorphism
\begin{equation*}
\bigwedge\nolimits^{\!\! m}_{\bbz }F^{\times }\to\Hup^m(X_F,\bbz ),\quad \alpha_1\wedge\dotsb\wedge \alpha_m\mapsto\partial^1(\alpha_1)\cup\dotsb\cup\partial^1(\alpha_m)
\end{equation*}
does not factor through $\MK_m(F)$ for any $m\ge 2$ and any field $F\supseteq\bbq (\zeta_{\infty })$.
\begin{Proposition}\label{prop:corblochkato}
	Let $F\supseteq\bbq (\zeta_{\infty })$ be a field.
	\begin{enumerate}
		\item For every $m\ge 0$ the homomorphism
		\begin{equation*}
		\Hup^m(X_F,\bbq )\to\Hup^m(X_F,\bbq /\bbz )
		\end{equation*}
		is surjective.
		\item For every $m\ge 0$ the group $\Hup^m(X_F,\bbz )$ is torsion-free.
	\end{enumerate}
\end{Proposition}
\begin{proof}
	Consider the commutative diagram
	\begin{equation*}
	\xymatrix{
		\bigwedge\nolimits_{\bbq }^{\!\! m}(F^{\times }\otimes\bbq )\ar[r]\ar[d] & \MK_m(F)\otimes\bbq /\bbz \ar[d]\\
		\Hup^m(X_F,\bbq )\ar[r]& \Hup^m(X_F,\bbq /\bbz ).
		}
	\end{equation*}
	Here the upper horizontal map is surjective by construction, and the right vertical map is surjective by the Bloch--Kato conjecture. Hence the lower horizontal map has to be surjective as well, which proves~(i). The map from (i) is part of a long exact sequence
	\begin{equation*}
	\begin{split}
	\cdots\to &\Hup^m(X_F,\bbz )\to\Hup^m(X_F,\bbq )\to\Hup^m(X_F,\bbq /\bbz)\overset{\delta^m}{\to}\\
	&\Hup^{m+1}(X_F,\bbz )\to\Hup^{m+1}(X_F,\bbq )\to\Hup^{m+1}(X_F,\bbq /\bbz)\overset{\delta^{m+1}}{\to}\cdots .
	\end{split}
	\end{equation*}
	By (i) the connecting homomorphisms $\delta^m$ have to vanish, hence the short sequences
	\begin{equation*}
	0\to\Hup^m(X_F,\bbz )\to\Hup^m(X_F,\bbq )\to\Hup^m(X_F,\bbq /\bbz)\to 0
	\end{equation*}
	are also exact. Therefore $\Hup^m(X_F,\bbz)$ injects into a $\bbq$-vector space, which proves~(ii).
\end{proof}

\subsubsection*{The cup product} Consider the following axiomatisation of the situation in Theorem~\ref{Thm:CechCohomologyOfXF}.(i). Assume that $X$ is a connected compact Hausdorff space such that $\Hup^n(X,A)=0$ for all finite abelian groups $A$ and $n>0$, and that a profinite group $G$ operates continuously and freely on~$X$. Then arguing as above we obtain canonical isomorphisms $\Hup^n(G,A)\cong\Hup^n(X/G,A)$ for all $n\geq 0$. The existence of such a space $X$ can thus be used to relate continuous group cohomology with the cohomology of a compact Hausdorff space, and one may ask whether such a space always exists. The answer is no:

\begin{Proposition}\label{Prop:CompactCSImpliesTorsionFree}
Let $X$ be a connected compact Hausdorff space with $\Hup^n(X,A)$ $=0$ for $n>0$, for all finite abelian groups~$A$. Let $G$ be a profinite group acting continuously and freely on~$X$. Then $G$ is torsion-free.
\end{Proposition}

It would be interesting to see whether one can deduce further properties of $G$.

\begin{proof}
	Assume the contrary; by passing to a suitable subgroup we may then assume that $G$ is cyclic of order $p$ for some prime~$p$. By the Cartan--Leray spectral sequence we obtain isomorphisms $\Hup^n(X/G,\bbf_p)\cong\Hup^n(G,\bbf_p)$. It is not hard to see that these isomorphisms commute with the cup product.

By Lemma~\ref{Lem:CohomologyClassesOnCompactAreNilpotent} below, every class in $\Hup^2(X/G,\bbf_p)$ is nilpotent. However, there exists a non-nilpotent class in $\Hup^2(G,\bbf_p)$ (see e.g.\ \cite[Chapter~II, Corollary~4.2]{MR2035696}), a contradiction.
\end{proof}

\begin{Lemma}\label{Lem:CohomologyClassesOnCompactAreNilpotent}
Let $X$ be a compact Hausdorff space, let $A$ be a ring and let $n>0$. Then every class in $\Hup^n(X,A)$ is nilpotent.
\end{Lemma}

\begin{proof}
There is a canonical isomorphism of graded rings
\begin{equation*}
\Hup^{\ast }(X,A)\cong\varinjlim_{\mathfrak{U}}\check{\mathrm{H}}^{\ast }(\mathfrak{U},A)
\end{equation*}
where the limit runs over all \emph{finite} open covers $\mathfrak{U}$ of~$X$. By construction, every class of positive degree in $\check{\mathrm{H}}^{\ast }(\mathfrak{U},A)$ is nilpotent.
\end{proof}

\begin{Corollary}\label{Cor:AbsoluteGaloisIsTorsionFree}
Let $F\supseteq\bbq (\zeta_{\infty})$ be a field. Then the absolute Galois group of $F$ is torsion-free.\hfill $\square$
\end{Corollary}

Of course, Corollary~\ref{Cor:AbsoluteGaloisIsTorsionFree} also follows from the Theorem of Artin--Schreier (\cite[Satz~4]{MR3069467}, see also \cite[Theorem~11.14]{MR1009787}) which says that the absolute Galois group of a field is finite and nontrivial if and only if that field is real closed. Clearly, no field containing $\bbq (\zeta_{\infty})$ can be real closed.

\section{The cyclotomic character}

\noindent In this section we develop a variant of the preceding constructions that works also for fields which do not contain all roots of unity, provided their absolute Galois groups are pro-$\ell$-groups. We will summarise the necessary structural results, the proofs being very similar to the case treated before, and then discuss in more detail some actions on cohomology groups that only become nontrivial in this new case.

\subsection{A variant with (some) roots of unity} Throughout this section we fix a rational prime $\ell$ and a perfect field $F$ with algebraic closure $\Fbar$ whose characteristic is not equal to $\ell$ (but may well be positive) such that $\Gal (\Fbar /F)$ is a pro-$\ell$-group.

We write $\mu_n$ for the group of all $n$-th roots of unity in $\Fbar$, where $n\in\bbn$. We also set
\begin{equation*}
\mu_{\ell^{\infty }}=\bigcup_{n\in\bbn }\mu_{\ell^n}\qquad\text{and}\qquad\mu_{\ell '}=\bigcup_{n\in\bbn\smallsetminus\ell\bbn }\mu_n.
\end{equation*}
The groups $\mu_{\ell^n}$ for $n<\infty$ are cyclic of order~$\ell^n$. There is a continuous group homomorphism
\begin{equation*}
\chi_{\ell ,F}\colon\Gal (\Fbar /F)\to\Aut\mu_{\ell^{\infty }}\overset{\cong}{\to}\bbz_{\ell}^{\times }
\end{equation*}
called the \emph{$\ell$-adic cyclotomic character} and characterised by $\sigma (\zeta )=\zeta^{\chi_{\ell ,F}(\sigma )}$ for all $\sigma\in\Gal (\Fbar /F)$ and $\zeta\in\mu_{\ell^{\infty }}$. Its kernel is the group $\Gal (\Fbar /F(\zeta_{\ell^{\infty }}))$, and the possibilities for its image are rather restricted.
\begin{Proposition}\label{Prop:ImageOfLAdicCyclotomic} Let $n$ be maximal such that $\mu_{\ell^n}\subset F$.	If $\ell$ is odd or if $n>1$, the image of $\Gal (\Fbar /F)$ under the $\ell$-adic cyclotomic character is equal to the subgroup
	\begin{equation*}
	U_{\ell^n}\defined 1+\ell^n\bbz_{\ell }\subset\bbz_{\ell}^{\times }.
	\end{equation*}
	If $\ell =2$ and $n=1$, there is some $m\in\{ 2,3,4,\dots ,\infty \}$ such that the image is generated by $U_{2^m}\subset\bbz_2^{\times}$ and $-1\in\bbz_2^{\times }$, where we set $U_{2^{\infty }}=\{ 1\}$.
\end{Proposition}
\begin{proof}
	By assumption, this image $H=\operatorname{im}\chi_{\ell ,F}$ is a closed subgroup of $\bbz_{\ell }^{\times 
	}$ which is contained in $U_{\ell^n}$ but not contained in~$U_{\ell^{n+1}}$. If $\ell$ is odd or $n>1$, the $\ell$-adic exponential series defines an isomorphism of topological groups $\ell^n\bbz_{\ell }\to U_{\ell^n}$. Any closed subgroup of $\bbz_{\ell }$ is an ideal, hence of the form $\ell^m\bbz_{\ell }$ for $0\le m\le \infty$, and an ideal contained in $\ell^n\bbz_{\ell }$ but not contained in $\ell^{n+1}\bbz_{\ell }$ must be equal to~$\ell^n\bbz_{\ell }$.
	
	The case $\ell =2$ and $n=1$ remains. Here the exponential series defines an isomorphism $4\bbz_2\to U_4$, hence by the previous argument we find that $H\cap U_4=U_{2^m}$ for some $2\le m\le \infty$. From the short exact sequence
	\begin{equation*}
	0\to 4\bbz_2\overset{\exp}{\to}U_4\to\{\pm 1\}\to 0
	\end{equation*}
	and the assumption that $H\nsubseteq U_4$ we conclude that $H=\pm U_{2^m}$.
\end{proof}
By considering algebraic extensions of finite fields we see that all of these possibilities do occur. For another example, if $F$ is a real closed field, then $\ell =2$ and $n=1$, and the image of $\chi_{2,F}\colon\Gal (\Fbar /F)\hookrightarrow\bbz_2^{\times }$ is simply $\{\pm 1\}\subset\bbz_2^{\times }$.

We will only formulate our main results in this section for the case where $\operatorname{im}\chi_{\ell ,F}=U_{\ell^n}$, the case $\operatorname{im}\chi_{2,F}=\pm U_{2^m}$ being similar.

\subsubsection*{The spaces $Y_{\ell^n,F}$} For each $n\in\bbn\cup\{\infty\}$ and each perfect field $F$ of characteristic different from $\ell$ with $\mu_{\ell^n}\subset F$ and each injective character $\iota\colon\mu_{\ell^n}\to\bbs^1$ we will now define a topological space $Y_{\ell^n,F}(\iota )$. We start with the algebraically closed case and let
\begin{equation*}
Y_{\ell^n,\Fbar }(\iota )=\{ \chi\in\Hom (\Fbar^{\times }/\mu_{\ell '},\bbs^1)\mid\chi\rvert_{\mu_{\ell^n}}=\iota \}
\end{equation*}
where `$\Hom$' denotes group homomorphisms; this space is endowed with the compact-open topology.
The Galois group $\Gal (\Fbar /F)$ operates continuously on $Y_{\ell^n,\Fbar}(\iota )$, and we set
\begin{equation*}
Y_{\ell^n,F}(\iota )=\Gal (\Fbar /F)\backslash Y_{\ell^n,\Fbar }(\iota ).
\end{equation*}
\begin{Proposition}\label{Prop:SimplePropertiesOfYlnF}
	Let $\ell$ be a rational prime and $F$ a perfect field of characteristic different from $\ell$ such that $\Gal (\Fbar /F)$ is a pro-$\ell$-group. Let $n\in\bbn\cup\{\infty\}$ such that $\mu_{\ell^n}\subset F$ and let $\iota\colon\mu_{\ell^n}\to\bbs^1$ be an injective character.
	\begin{enumerate}
		\item $Y_{\ell^n,F}(\iota )$ is a nonempty compact Hausdorff space.
		\item $\Gal (\Fbar /F)$ operates freely and properly on $Y_{\ell^n,\Fbar}(\iota )$.
	\end{enumerate}
\end{Proposition}
\begin{proof}
	The proof is essentially analogous to that of Proposition~\ref{Prop:SimplePropertiesOfCHXF}.(i) and~(ii). For the freeness in (ii) we use Lemma~\ref{Lem:MitHilbert90UndL} below, which is similar to Lemma~\ref{Lem:MitHilbert90} above.
\end{proof}
\begin{Lemma}\label{Lem:MitHilbert90UndL}
	Let $\ell$ and $F$ be as in Proposition~\ref{Prop:SimplePropertiesOfYlnF}, and let $k$ be a field. Then $\Gal (\Fbar /F)$ operates freely on the set $\mathcal{I}_{\ell }(\Fbar ,k)$ of all group homomorphisms $\Fbar^{\times }\to k^{\times }$ which are injective on $\mu_{\ell^{\infty }}$ and trivial on~$\mu_{\ell '}$.
\end{Lemma}
\begin{proof}
	Let $\sigma\in\Gal (\Fbar /F)$ be different from the identity element, and let $\chi\in\mathcal{I}_{\ell }(\Fbar ,k)$. We need to show that $\sigma (\chi )\neq\chi$.
	
	Replacing $F$ by the fixed field of $\sigma$ we may assume that $\sigma$ topologically generates $\Gal (\Fbar /F)$. Since $\Gal (\Fbar /F)$ is a pro-$\ell$-group, it is then either finite cyclic or isomorphic to~$\bbz_{\ell}$.
	
	In the first case it has to be cyclic of order $2$ by the Theorem of Artin--Schreier \cite[Satz~4]{MR3069467}, and $F$ has to be real closed. Then $\chi (\zeta_4)$ is a primitive fourth root of unity in $k$, and $\sigma (\chi )(\zeta_4)=\chi (\sigma (\zeta_4))=\chi (\zeta_4^{-1})=\chi (\zeta_4)^{-1}\neq \chi (\zeta_4)$, hence $\chi\neq\sigma (\chi)$.
	
	In the second case we let $E$ be the fixed field of $\sigma^{\ell}$, so that $E/F$ is a cyclic extension of degree~$\ell$. By Hilbert's Theorem~90 we find some $\alpha\in E$ with $\sigma (\alpha )/\alpha=\zeta_{\ell}$, and since $\chi (\zeta_{\ell })\neq 1$ we then find $\sigma (\chi )(\alpha )\neq \chi (\alpha )$, i.e.\ $\chi\neq\sigma (\chi)$.
\end{proof}
It is here that the assumption that $\Gal (\Fbar /F)$ is a pro-$\ell$-group is used.

\begin{Lemma}\label{Lem:ConnectedComponentsOfYlnFj}
	Let $\ell$ and $F$ as before, and assume that the image of $\chi_{\ell ,F}\colon \Gal (\Fbar /F)\to\bbz_{\ell}^{\times }$ is equal to $U_{\ell^n}$ for some $n\in\bbn\cup\{\infty\}$. Let $\iota\colon\mu_{\ell^n}\to\bbs^1$ be an injective character and let $\tilde{\iota }\colon \mu_{\ell^{\infty }}\to\bbs^1$ be an injective character with $\tilde{\iota}\rvert_{\mu_{\ell^n}}=\iota$.
	
	Then the spaces $Y_{\ell^n,F}(\iota )$ and $Y_{\ell^{\infty },F(\zeta_{\ell^{\infty}})}(\tilde{\iota })$ are canonically homeomorphic.
\end{Lemma}
\begin{proof}
	There is a tautological inclusion $Y_{\ell^{\infty },\Fbar}(\tilde{\iota})\hookrightarrow Y_{\ell^n,\Fbar }(\iota )$, which is equivariant for the group inclusion $\Gal (\Fbar /F(\zeta_{\infty}))\hookrightarrow\Gal (\Fbar /F)$. It therefore descends to a continuous map $Y_{\ell^{\infty },F(\zeta_{\ell^{\infty }})}(\tilde{\iota })\to Y_{\ell^n,F}(\iota )$, and since the spaces under consideration are compact Hausdorff spaces, it suffices to show that this map is a bijection.
	
	Note that $Y_{\ell^n,\Fbar }(\iota )$ is the disjoint union (as sets, not as topological spaces!)
	\begin{equation*}
	\bigcup_{\substack{j\in\Hom (\mu_{\ell^{\infty }},\bbs^1)\\ j\rvert_{\mu_{\ell^n}}=\iota }}Y_{\ell^{\infty },\Fbar }(j),
	\end{equation*}
	and the subgroup $\Gal (\Fbar /F(\zeta_{\ell^{\infty }}))\subset\Gal (\Fbar /F)$ preserves each summand while the quotient $\Gal (F(\zeta_{\ell^{\infty }})/F)\cong U_{\ell^n}$ permutes the summands simply transitively. Therefore the quotient of each summand by its stabiliser $\Gal (\Fbar /F(\zeta_{\ell^{\infty }}))$ maps bijectively to the quotient of the whole set by $\Gal (\Fbar /F)$.
\end{proof}
\begin{Proposition}
	Let $\ell$ and $F$ be as before with $\operatorname{im}\chi_{\ell ,F}=U_{\ell^n}$ and let $\iota\colon\mu_{\ell^n}\hookrightarrow\bbs^1$. Then $Y_{\ell^n,F}(\iota )$ is connected.
\end{Proposition}
\begin{proof}
	By Lemma~\ref{Lem:ConnectedComponentsOfYlnFj} we may assume that $n=\infty$. Then $Y_{\ell^{\infty },\Fbar}(\iota )$ is homeomorphic to the Pontryagin dual of the torsion-free group $\Fbar^{\times }/\mu_{\ell '}\mu_{\ell^{\infty }}=\Fbar^{\times }/\mu_{\infty }$, hence it is connected. Therefore its quotient $Y_{\ell^{\infty },F}(\iota )$ is also connected.
\end{proof}
\begin{Proposition}\label{Prop:EtaleFGYellF}
	Let $\ell$, $F$, $n$, $\iota$ and $\tilde{\iota }$ as before.
	\begin{enumerate}
		\item The \'etale fundamental group of $Y_{\ell^{\infty },\Fbar }(\tilde{\iota })$ is trivial.
		\item The \'etale fundamental group of $Y_{\ell^n,F}(\iota )$ is isomorphic to $\Gal (\Fbar /F(\zeta_{\ell^{\infty } }))$; this isomorphism is canonical up to inner automorphisms.
	\end{enumerate}
\end{Proposition}
\begin{proof}
	For (i) note that $Y_{\ell^{\infty },\Fbar }(\tilde{\iota })$ is homeomorphic to the Pontryagin dual of $\Fbar^{\times }/\mu_{\infty }$, a $\bbq$-vector space. Hence $Y_{\ell^{\infty },\Fbar}(\tilde{\iota })$ is the universal profinite covering space of $Y_{\ell^{\infty },F(\zeta_{\ell^{\infty }})}$ by Proposition~\ref{Prop:SimplePropertiesOfYlnF}. It follows that the deck transformation group $\Gal (\Fbar /F(\zeta_{\ell^{\infty }}))$ is isomorphic to the \'etale fundamental group of $Y_{\ell^{\infty },F(\zeta_{\ell^{\infty }})}(\tilde{\iota  })$; by Lemma~\ref{Lem:ConnectedComponentsOfYlnFj} this space is canonically homeomorphic to $Y_{\ell^n,F}(\iota )$.
\end{proof}
\begin{Theorem}
	Let $\ell$, $F$, $n$ and $\iota$ as before, and let $A$ be an abelian torsion group. Then there are canonical isomorphisms
	\begin{equation*}
	\Hup^m(Y_{\ell^n,F},A)\cong\Hup^m(F(\zeta_{\ell^{\infty }}),A)
	\end{equation*}
	and
	\begin{equation*}
	\Hup^m(Y_{\ell^n,F},\bbq )\cong\left( \bigwedge\nolimits_{\bbq }^{\!\! m}\Fbar^{\times}_{\mathrm{tf}}\right)^{\Gal (\Fbar /F(\zeta_{\ell^{\infty }}))}
	\end{equation*}
	for each $m\ge 0$.
\end{Theorem}
\begin{proof}
	The proof is similar to that of Theorem~\ref{Thm:CechCohomologyOfXF}.
\end{proof}
There is also again a scheme-theoretic version of these constructions. For a perfect field $F$ of characteristic different from $\ell$ with algebraic closure $\Fbar$ such that $\Gal (\Fbar /F)$ is a pro-$\ell$-group, an $n\in\bbn\cup\{\infty\}$ with $\mu_{\ell^n}\subset F$ and an embedding $\iota\colon\mu_{\ell^n}(F)\hookrightarrow\bbs^1$ we set
\begin{equation*}
B_{\ell^n,F}=\left(\bbc [\Fbar^{\times }/\mu_{\ell^\prime}]/I\right)^{\Gal (\Fbar /F)},
\end{equation*}
where $I\subset\bbc [\Fbar{\times }/\mu_{\ell^\prime}]$ is the ideal generated by all $[\zeta ]-\iota (\zeta )\cdot [1]$ with $\zeta\in\mu_{\ell^n}(F)$. Note that if $n<\infty$ then this ideal is generated by a single element $[\zeta_{\ell^n}]-\iota (\zeta_{\ell^n})\cdot [1]$ with $\zeta_{\ell^n}$ a primitive $\ell^n$-th root of unity.

Then we set $\mathcal{Y}_{\ell^n,F}=\Spec B_{\ell^n,F}$. Note that this still depends on $\iota$, but we suppress this to lighten the notation. In complete analogy to the schemes $\mathcal{X}_F$ we obtain the following properties:
\begin{Theorem}\label{thm:cohomelladic}
	Let $\ell$ be a rational prime and let $F$ be a perfect field of characteristic other than~$\ell$. Assume that $\Gal (\Fbar /F)$ is a pro-$\ell$-group and that the image of $\chi_{\ell ,n}\colon\Gal (\Fbar /F)\to \bbz_{\ell}^{\times }$ is equal to $U_{\ell^n}$ for some $n\in\bbn\cup\{\infty\}$.
	
	Then the scheme $\mathcal{Y}_{\ell^n,F}$ is connected, and its \'etale fundamental group is isomorphic to $\Gal (\Fbar /F(\zeta_{\ell^{\infty }}))$, the isomorphism being canonical up to inner automorphisms. For abelian torsion groups $A$ we obtain natural isomorphisms
	\begin{equation*}
	\Hup^m_{\et }(\mathcal{Y}_{\ell^n,F},A)\cong\Hup^m(F(\zeta_{\ell^{\infty }}),A).
	\end{equation*}
	The space $Y_{\ell^n,F}$ can be identified with a subspace in $\mathcal{Y}_{\ell^n,F}(\bbc )$ with the complex topology; it is a strong deformation retract.
\end{Theorem}
\begin{proof}
	In analogy to Lemma~\ref{Lem:ConnectedComponentsOfYlnFj} there is a canonical isomorphism $\mathcal{Y}_{\ell^n,F}\cong\mathcal{Y}_{\ell^{\infty },F(\zeta_{\ell^{\infty }})}$, hence we may assume that $n=\infty$. The proof is then analogous to those of Corollary~\ref{Cor:PioneetOfXFIsAbsoluteGalois}, Theorem~\ref{Thm:CompatibilityBetweenXFSpaceAndXFScheme}.(ii) and Theorem~\ref{Thm:EtaleCohomologyOfXFIsGaloisCohomology}.
\end{proof}

\subsection{Three actions on cohomology}
We assume that $\ell$ is a prime, $F$ a perfect field of characteristic other than $\ell$ with algebraic closure $\Fbar$ such that $\Gal (\Fbar /F)$ is a pro-$\ell$-group and such that the image of $\chi_{\ell ,F}\colon\Gal (\Fbar /F)\to\bbz_{\ell}^{\times }$ is equal to $U_{\ell^n}$ for some (finite!) $n\in\bbn$. We have seen that then for each abelian torsion group $A$ and each $m\ge 0$ there are canonical isomorphisms
\begin{equation}\label{eqn:ThreeIncarnationsOfHmFA}
\Hup^m(Y_{\ell^n,F},A)\cong\Hup^m_{\et }(\mathcal{Y}_{\ell^n,F},A)\cong\Hup^m(F(\zeta_{\ell^{\infty }}),A).
\end{equation}
On each of the groups in (\ref{eqn:ThreeIncarnationsOfHmFA}) there is a natural action from the left by a certain group; we will show below that these three actions are compatible.

\subsubsection*{The topological action} By construction $Y_{\ell^n,\Fbar }$ is a closed subset of the Pontryagin dual $(\Fbar^{\times }/\mu_{\ell '})^{\vee }$, more precisely a translate of $(\Fbar^{\times}_{\mathrm{tf}})^{\vee }$. The group $\Fbar^{\times }/\mu_{\ell '}$ is divisible and has no torsion elements of order prime to~$\ell$, therefore it is a $\bbz_{(\ell )}$-module in a unique way (here $\bbz_{(\ell)}\subset\bbq$ is the ring of rational numbers whose denominators are prime to~$\ell$). Therefore the group of units $\bbz_{(\ell )}^{\times }$ acts on $(\Fbar^{\times }/\mu_{\ell '})$ by group automorphisms from the left: $u\cdot (\alpha\bmod\mu_{\ell '})=(\alpha^u\bmod\mu_{\ell '})$. Hence it acts on the Pontryagin dual $(\Fbar^{\times }/\mu_{\ell '})^{\vee}$ from the right: $(\chi\cdot u)(\alpha )=\chi (\alpha^u)$.

This action does not preserve the subspace $Y_{\ell^n,\Fbar }\subset (\Fbar^{\times}/\mu_{\ell '})^{\vee }$, but its restriction to the subgroup
\begin{equation*}
U_{(\ell^n )}=1+\ell^n\bbz_{(\ell)}\subset\bbz_{(\ell )}^{\times }
\end{equation*}
will, because elements of $U_{(\ell^n)}$ operate trivially on~$\mu_{\ell^n}$.

This action of $U_{(\ell^n)}$ on $Y_{\ell^n,F}$ commutes with that of $\Gal (\Fbar /F)$, hence it descends to a right action of $U_{(\ell^n)}$ on $Y_{\ell^n,F}$.

Then by functoriality we obtain a left action of $U_{(\ell^n)}$ on the cohomology group $\Hup^m(Y_{\ell^n,F},A)$ for any abelian torsion group $A$ and any $m\ge 0$.

\subsubsection*{The arithmetic action} The scheme $\mathcal{Y}_{\ell^n,F}$ admits a natural model over the ring of cyclotomic integers $\bbz [\zeta_{\ell^n}]$. To be precise, let \begin{equation*}\mathcal{Y}_{F,\,\mathrm{int}}=\Spec B_{F,\,\mathrm{int}}
\end{equation*}
with
\begin{equation*}
B_{F,\,\mathrm{int}}=\left(\bbz [\Fbar^{\times }/\mu_{\ell}^\prime][([\zeta_\ell]-1)^{-1}]\right)^{\Gal (\Fbar /F)}
\end{equation*}
where $\zeta_\ell\in F$ is an $\ell$-th root of unity. Each embedding $\iota\colon\mu_{\ell^n}(F)\hookrightarrow\bbs^1$ induces a ring embedding
\begin{equation*}
e_{\iota }\colon\bbz [\zeta_{\ell^n}]=\bbz [\mu_{\ell^n}(\bbc )]\hookrightarrow B_{F,\,\mathrm{int}}
\end{equation*}
with $e_{\iota }(\iota (\zeta ))=[\zeta ]$ for each $\zeta\in\mu_{\ell^n}(F)$. This ring embedding turns $\mathcal{Y}_{F,\,\mathrm{int}}$ into a $\bbz [\zeta_{\ell^n}]$-scheme, and there is a natural isomorphism
\begin{equation*}
\mathcal{Y}_{F,\,\mathrm{int}} \times_{e_{\iota },\,\Spec\bbz [\zeta_{\ell^n}]}\Spec\bbc\cong\mathcal{Y}_{\ell^n,F}(\iota ).
\end{equation*}
There are then also a natural isomorphisms
\begin{equation*}
\Hup^m_{\et }(\mathcal{Y}_{\ell^n,F},A)\cong\Hup^m_{\et }(\mathcal{Y}_{F,\,\mathrm{int}}\times_{e_{\iota },\,\Spec\bbz [\zeta_{\ell^n}]}\Spec\qbar ,A).
\end{equation*}
Now $\Gal (\qbar /\bbq (\zeta_{\ell^n}))$ operates from the left on $B_{F,\,\mathrm{int}}\otimes_{e_{\iota },\,\bbz [\zeta_{\ell^n}]}\qbar$ (trivially on the first factor and tautologically on the second factor), hence from the right on the spectrum of this algebra, hence from the left on the cohomology of the latter. Thus we obtain a left action of  $\Gal (\qbar /\bbq (\zeta_{\ell^n}))$ on $\Hup^m_{\et }(\mathcal{Y}_{\ell^n,F},A)$.

\subsubsection*{The group-theoretic action} This is the simplest to describe: from the short exact sequence of profinite groups
\begin{equation*}
1\to\Gal (\Fbar /F(\zeta_{\ell^{\infty }}))\to\Gal (\Fbar /F)\overset{\chi_{\ell ,F}}{\to }U_{\ell^n}\to 1
\end{equation*}
we obtain a right action of $\Gal (\Fbar /F)$ on its normal subgroup $\Gal (\Fbar /F(\zeta_{\ell^{\infty }}))$ by $h^g=g^{-1}hg$, hence a left action on $\Hup^m(F(\zeta_{\ell^{\infty }}),A)=\Hup^m(\Gal (\Fbar /F (\zeta_{\ell^{\infty }})),A)$. Since inner group automorphisms act trivially on cohomology, this descends to a left action by $U_{\ell^n}\cong\Gal (F(\zeta_{\ell^{\infty }})/F)$.

\begin{Theorem}\label{Thm:CompatibilityThreeActionsCohomology}
	Let $\ell$ be a rational prime, $F$ a perfect field of characteristic other than $\ell$ with algebraic closure $\Fbar$ such that $\Gal (\Fbar /F)$ is a pro-$\ell$-group and such that $\operatorname{im}\chi_{\ell ,F}=U_{\ell^n}$ for some $n\in\bbn$. Let $A$ be an abelian torsion group. Then the following claims hold for each $m\ge 0$:
	\begin{enumerate}
		\item The arithmetic action of $\Gal (\qbar /\bbq (\zeta_{\ell^n}))$ on $\Hup^m_{\et}(\mathcal{Y}_{\ell^n,F},A)$ factors through the $\ell$-adic cyclotomic character
		\begin{equation*}
		\chi_{\ell}=\chi_{\ell ,\bbq (\zeta_{\ell^n})}\colon\Gal (\qbar /\bbq (\zeta_{\ell^n}))\twoheadrightarrow U_{\ell^n}\subset\bbz_{\ell}^{\times }.
		\end{equation*}
		Hence it defines an action by $U_{\ell^n}$ on $\Hup^m_{\et }(\mathcal{Y}_{\ell^n,F},A)$ which we also call arithmetic.
		\item The topological action of $U_{(\ell^n)}$ on $\Hup^m(Y_{\ell^n,F},A)$ extends uniquely to a continuous action of $U_{\ell^n}=1+\ell^n\bbz_{\ell }$ on the same space. Here continuity refers to the $\ell$-adic topology on~$U_{\ell^n}$.
	\end{enumerate}
	Moreover the isomorphisms in (\ref{eqn:ThreeIncarnationsOfHmFA})	are $U_{\ell^n}$-equivariant up to a sign. More precisely, the isomorphism
	\begin{equation*}
	\Hup^m_{\et }(\mathcal{Y}_{\ell^n,F},A)\cong\Hup^m(F(\zeta_{\ell^{\infty }}),A)
	\end{equation*}
	is equivariant for the identity $U_{\ell^n}\to U_{\ell^n}$, while the other two isomorphisms
	\begin{equation*}
	\Hup^m(Y_{\ell^n,F},A)\cong\Hup^m_{\et }(\mathcal{Y}_{\ell^n,F},A)\qquad\text{and}\qquad\Hup^m(Y_{\ell^n,F},A)\cong\Hup^m(F(\zeta_{\ell^{\infty }}),A)
	\end{equation*}
	are equivariant for the inverse map $U_{\ell^n}\to U_{\ell^n}$, $u\mapsto u^{-1}$.
\end{Theorem}
\begin{proof}
	Broken up into smaller pieces, this is proved below in Proposition~\ref{Prop:FinalStepArithmeticGrouptheoretic}, Corollary~\ref{Cor:FinalStepTopologicalGroupTheoretic} and Proposition~\ref{Prop:FinalStepCompatibilityTopologicalArithmetic} below.
\end{proof}

The proof of Theorem~\ref{Thm:CompatibilityThreeActionsCohomology} will fill up the remainder of this section. More precisely, for each two of the three actions we will establish equivariance for these two actions, and prove along the way that the actions factor through~$U_{\ell^n}$.  From a strictly logical point of view this is redundant, but we believe each of the three proofs reveals something particular about the objects under consideration.

\subsection{Compatibility of the arithmetic and group-theoretic actions} 
 For each field $F$ satisfying the conditions of Theorem~\ref{Thm:CompatibilityThreeActionsCohomology} and each `coefficient field' $k$ we set $\mathcal{Y}_{F,k}=\mathcal{Y}_{F,\,\mathrm{int}}\times_{\Spec\bbz }\Spec k$. The scheme $\mathcal{Y}_{F,k}$ is not necessarily connected, but for $k=\bbq$ it is. We will now determine the \'etale fundamental group of $\mathcal{Y}_{F,\bbq}$.

There is a pro-\'etale but possibly disconnected normal covering $\mathcal{Y}_{\Fbar ,\qbar }$. Its deck transformation group can be identified with $\Gal (\Fbar /F)\times\Gal (\qbar /\bbq )$ which operates in the obvious way on $\mathcal{Y}_{\Fbar ,\qbar}$ (from the right, however). The space of connected components $\pi_0(\mathcal{Y}_{\Fbar ,\qbar })$ is canonically homeomorphic to $\operatorname{Isom}(\mu_{\ell^{\infty }}(\Fbar ),\mu_{\ell^{\infty }}(\qbar))$, with its obvious left action by $\Gal (\qbar /\bbq )$ and its obvious right action by $\Gal (\Fbar /F)$. The space $\operatorname{Isom}(\mu_{\ell^{\infty }}(\Fbar ),\mu_{\ell^{\infty }}(\qbar ))$ is a two-sided principal homogeneous space for the abelian group $\bbz_{\ell}^{\times }$, and the Galois actions respect this structure.
\begin{Lemma}
	Let $\Gal (\Fbar /F)\times \Gal (\qbar /\bbq )$ act on $\operatorname{Isom}(\mu_{\ell^{\infty }}(\Fbar ),\mu_{\ell^{\infty }}(\qbar ))$ by $(\sigma_F,\sigma_{\bbq })\cdot\iota =\chi_{\ell }(\sigma_F)\chi_{\ell}(\sigma_{\bbq })^{-1}\cdot \iota$. Then the homeomorphism $\pi_0(\mathcal{Y}_{\Fbar ,\qbar })\to\operatorname{Isom}(\mu_{\ell^{\infty }}(\Fbar ),\mu_{\ell^{\infty }}(\qbar ))$ is equivariant for $\Gal (\Fbar /F)\times\Gal (\qbar /\bbq )$.
\end{Lemma}
\begin{proof}
	This follows from the previous discussion. As to the different signs, note that $\Gal (\qbar /\bbq)$ operates most naturally from the left on $\operatorname{Isom}(\mu_{\ell^{\infty }}(\Fbar ),\mu_{\ell^{\infty }}(\qbar ))$ and from the right on $\mathcal{Y}_{\Fbar ,\bbq }$, whereas $\Gal (\Fbar /F)$ operates most naturally from the right on both spaces.
\end{proof}
\begin{Corollary}
	Fix some group isomorphism $\tilde{\iota }\colon\mu_{\ell^{\infty }}(\Fbar )\to\mu_{\ell^{\infty }}(\qbar )$, and denote the corresponding component of $\mathcal{Y}_{\Fbar ,\qbar }$ by $\mathcal{Y}_{\Fbar ,\qbar }^{\circ}$. Then $\mathcal{Y}_{\Fbar ,\qbar }^{\circ }\to\mathcal{Y}_{F,\bbq }$ is a universal profinite covering space, and its deck transformation group is
	\begin{equation*}
	G_{F,\bbq }\defined \{ (\sigma_F,\sigma_{\bbq })\in\Gal (\Fbar /F)\times\Gal (\qbar /\bbq )\mid \chi_{\ell,F}(\sigma_F)=\chi_{\ell }(\sigma_{\bbq }) \} .
	\end{equation*}
	Hence there is a natural isomorphism $\pioneet (\mathcal{Y}_{F,\bbq })\cong G_{F,\bbq }$, canonical up to inner automorphisms.\hfill $\square$
\end{Corollary}
Now take an isomorphism $\iota\colon\mu_{\ell^n}(F)\to\mu_{\ell^n}(\qbar )$, and extend it to an isomorphism $\tilde{\iota }\colon\mu_{\ell^{\infty }}(\Fbar )\to\mu_{\ell^{\infty }}(\qbar )$. We wish to find our space $\mathcal{Y}_{\ell^n,F}$ and its model over $\bbq (\zeta_{\ell^n})$ back as a quotient of $\mathcal{Y}_{\Fbar ,\qbar }^{\circ }$.

We construct two intermediate coverings of $\mathcal{Y}_{\Fbar ,\qbar }^{\circ }\to\mathcal{Y}_{F,\bbq }$.
\begin{itemize}
	\item First note that there is a continuous epimorphism
	\begin{equation*}
	\chi_{\ell}\colon G_{F,\bbq }\to\bbz_{\ell}^{\times },\quad (\sigma_F,\sigma_{\bbq } )\mapsto\chi_{\ell}(\sigma_F)=\chi_{\ell}(\sigma_{\bbq }).
	\end{equation*}
	Then $G_{F,\bbq }(\ell^n)=\chi_{\ell}^{-1}(U_{\ell^n})$ is an open normal subgroup of $G_{F,\bbq }\cong\pioneet (\mathcal{Y}_{F,\bbq })$, and the corresponding intermediate covering is equal to $\mathcal{Y}_{F,\bbq (\zeta_{\ell^n})}^{\circ }$, the connected component of $\mathcal{Y}_{F,\bbq (\zeta_{\ell^n})}$ determined by~$\iota$. This is precisely the model of $\mathcal{Y}_{\ell^n,F}$ over $\bbq (\zeta_{\ell^n})$ used to define the $\Gal (\qbar /\bbq (\zeta_{\ell^n }))$-action on its \'etale cohomology.
	\item The closed normal subgroup $\Gal (\Fbar /F(\zeta_{\ell^{\infty }}))\times\{ 1\}\subset G_{F,\bbq }\subset\Gal (\Fbar /F)\times\Gal (\qbar /\bbq)$ defines the normal profinite covering space 
	\begin{equation*}
		\mathcal{Y}_{F(\zeta_{\ell^{\infty }}),\qbar}^{\circ }\to\mathcal{Y}_{F,\bbq (\zeta_{\ell^n})}^{\circ }\to\mathcal{Y}_{F,\bbq };
	\end{equation*}
	as a $\mathcal{Y}_{F,\bbq (\zeta_{\ell^n})}^{\circ}$-scheme this is isomorphic to $\mathcal{Y}_{F,\bbq (\zeta_{\ell^n})}\otimes_{\bbq (\zeta_{\ell^n})}\qbar$.
\end{itemize}
\begin{Lemma}\label{Lem:IntermediateActionArithmeticGroupTheoretical}
	There are natural isomorphisms of cohomology groups
	\begin{equation*}
	\Hup^m(F(\zeta_{\ell^{\infty }}),A)\overset{\cong}{\leftarrow}\Hup^m_{\et}(\mathcal{Y}_{F(\zeta_{\ell^\infty }),\qbar }^{\circ },A)\overset{\cong}{\to}\Hup^m_{\et }(\mathcal{Y}_{\ell^n,F,\qbar},A)
	\end{equation*}
	equivariant for the group homomorphisms (natural projections)
	\begin{equation*}
	\Gal (\Fbar /F)\leftarrow G_{F,\bbq }(\ell^n)\to\Gal (\qbar /\bbq (\zeta_{\ell^n})).
	\end{equation*}
\end{Lemma}
\begin{proof}
The isomorphism $\Hup^m_{\et }(\mathcal{Y}_{F(\zeta_{\ell^{\infty }}),\qbar }^{\circ },A)\to\Hup^m(F(\zeta_{\ell^{\infty }}),A)$ is obtained from the Cartan--Leray spectral sequence applied to the universal covering $\mathcal{Y}_{\Fbar ,\qbar }^{\circ}\to\mathcal{Y}_{F(\zeta_{\ell^{\infty }}),\qbar }^{\circ }$; note that all the higher cohomology groups of $\mathcal{Y}_{\Fbar ,\qbar }^{\circ}$ with torsion coefficients vanish.

The rest follows from the preceding discussion.
\end{proof}
\begin{Proposition}\label{Prop:FinalStepArithmeticGrouptheoretic}
	Let $\sigma_F\in\Gal (\Fbar /F)$ and $\sigma_{\bbq }\in\Gal (\qbar /\bbq (\zeta_{\ell^n}))$ be such that $\chi_{\ell }(\sigma_F)=\chi_{\ell }(\sigma_{\bbq })$. Then under the isomorphism $\Hup^m(F(\zeta_{\ell^{\infty }}),A)\cong\Hup^m_{\et}(\mathcal{Y}_{\ell^n,F,\qbar },A)$ in (\ref{eqn:ThreeIncarnationsOfHmFA}) the actions of $\sigma_F$ and $\sigma_{\bbq }$ correspond to each other.
\end{Proposition}
\begin{proof}
	This follows from Lemma~\ref{Lem:IntermediateActionArithmeticGroupTheoretical}: just note how the element $(\sigma_F,\sigma_{\bbq })\in G_{F,\bbq }(\ell^n)$ acts.
\end{proof}

\subsection{Compatibility of the group-theoretic and topological actions} We shall consider diverse Galois categories and exact functors between them:
\begin{equation}\label{eqn:DiagramGaloisCategoriesYellnF}
\xymatrix{
	\UellnSet\ar[d]\ar[r]&\mathbf{FEt}(\Spec F)\ar[d]\ar[r]&\mathbf{FEt}(\Spec F(\zeta_{\ell^{\infty }}))\ar[d]\\
	\UBrellnSet\ar[r]&\mathbf{FCov}_{U_{(\ell^n)}}(Y_{\ell^n,F})\ar[r]&\mathbf{FCov}(Y_{\ell^n,F})
}
\end{equation}
Though we have suppressed this in the notation, this diagram will depend on a choice of $\iota\colon\mu_{\ell^n}\hookrightarrow\bbs^1$ and of an extension $\tilde{\iota}\colon\mu_{\ell^{\infty }}\hookrightarrow\bbs^1$.
The functors in (\ref{eqn:DiagramGaloisCategoriesYellnF}) are as follows:
\begin{altitemize}
	\item $\UellnSet\to\mathbf{FEt}(\Spec F)$ is the composition of two functors
	\[
	\UellnSet\to\Gal (\Fbar /F)\text{-\textbf{\textup{FSet}}}\to\mathbf{FEt}(\Spec F),
	\]
	the first of which is induced by the group homomorphism $\chi_{\ell,F}\colon\Gal (\Fbar /F)\to U_F$ and the second of which is `Grothendieck's Galois theory'. For an explicit description, let $S$ be a finite set with a continuous left action of $\Gal (\Fbar /F)$, then $\Gal (\Fbar /F)$ acts from the left on the $F$-algebra $\Fbar^S$ by
	\begin{equation*}
	\sigma ((\alpha_s)_{s\in S})=(\alpha_{\sigma^{-1}s})_{s\in S},
	\end{equation*} 
	and the ring of invariants $E(S)=(\Fbar^S)^{\Gal (\Fbar /F)}$ is a finite \'etale $F$-algebra. Then the functor can be described as $S\mapsto\Spec E(S)$.
	\item $\mathbf{FEt}(\Spec F)\to\mathbf{FEt}(\Spec F(\zeta_{\ell^{\infty }}))$ is the functor $X\mapsto X\times_{\Spec F}\Spec F(\zeta_{\ell^{\infty }})$.
	\item The two functors in the lower horizontal line are obtained from the $U_{(\ell^n)}$-action on $Y_{\ell^n,F}$ as on page~\pageref{page:SesFCovForStackyQuotient}, that is, the first one sends a finite $U_{(\ell^n)}$-set $S$ to the product $Y_{\ell^n,F}\times S$ with the diagonal $U_{(\ell^n)}$-action, and the second one is the obvious forgetful functor.
	\item $\UellnSet\to\UBrellnSet$ is induced by the $\ell$-adic completion map $U_{(\ell^n)}\to U_{\ell^n}$.
	\item $\mathbf{FEt}(\Spec F)\to\mathbf{FCov}_{U_{(F)}}(Y_{\ell^n,F})$ sends $\Spec E$ for a finite extension $E/F$ to $Y_{\ell^n,E}$, and more generally for an \'etale $F$-algebra $E$ we set
	\begin{equation*}
	Y_{\ell^n,E}=\coprod_{\mathfrak{p}\in\Spec E}Y_{\ell^n,E/\mathfrak{p}}.
	\end{equation*}
	Note that in the basic case where $E$ is a field we need to choose an algebraic closure $\Ebar /E$ to even define $Y_{\ell^n,E}$, and an isomorphism $\Ebar\to\Fbar$ to obtain a map $Y_{\ell^n,E}\to Y_{\ell^n,F}$. However, as for the spaces $Y_F$ we check that $Y_{\ell^n,F}$ and $Y_{\ell^n,E}\to Y_{\ell^n,F}$ are independent up to canonical isomorphism from the choices of $\Fbar$, $\Ebar$ and $\Ebar\to\Fbar$.
	\item $\mathbf{FEt}(\Spec F(\zeta_{\ell^{\infty }}))\to\mathbf{FCov}(Y_{\ell^n,F})$ is the composition
	\[
	\mathbf{FEt}(\Spec F(\zeta_{\ell^{\infty }}))\to\mathbf{FCov}(Y_{\ell^{\infty },F(\zeta_{\ell^{\infty }})})\to\mathbf{FCov}(Y_{\ell^n,F})
	\]
	where the first functor sends $\Spec E$ to $Y_{\ell^{\infty },E}$ and the second functor is induced by the homeomorphism $Y_{\ell^n,F}\cong Y_{\ell^{\infty },F(\zeta_{\ell^{\infty }})}$.
\end{altitemize}
\begin{Lemma}\label{Lem:DiagramOfGaloisCategoriesCommutes}
	The diagram of exact functors between Galois categories (\ref{eqn:DiagramGaloisCategoriesYellnF})
	commutes up to isomorphism of functors. 
\end{Lemma}
\begin{proof}
	The commutativity of the right square is straightforward but tedious.
	
	For the commutativity of the left hand side, let $S$ be a finite set endowed with a continuous left action by~$U_{\ell^n}$. We will construct a natural isomorphism
	\begin{equation*}
	Y_{\ell^n,F}\times S\cong Y_{\ell^n,E(S)}
	\end{equation*}
	of $U_{(\ell^n)}$-equivariant finite covering spaces of~$Y_{\ell^n,F}$.
	
	First we may assume that the $U_{\ell^n}$-action on $S$ is transitive, because all the functors involved respect finite direct sums. Hence $E(S)$ is a finite field extension of~$F$. As explained above, to construct the covering $Y_{\ell^n,E(S)}\to Y_{\ell^n,F}$ we need to choose an embedding of $E(S)$ into~$\Fbar$, and later check that the choice of this embedding changes everything by canonical isomorphisms only (we omit that later part). By definition, $E(S)=(\Fbar^S)^{\Gal (\Fbar /F)}$, and hence
	\begin{equation*}
	\Hom_F(E(S),\Fbar )\cong\Hom_{\Fbar }(\Fbar^S,\Fbar )\cong S.
	\end{equation*}
	Therefore the choice we need to make is that of a particular element $s_0\in S$, which then allows us to trivialise the $\Gal (\Fbar /F)$-set $S$ as $\Gal (\Fbar /F)/H$, where $H=\Gal (\Fbar /E(S))$ is the stabiliser of~$s_0$.
	
	Note that
	\begin{equation*}
	\Hom_{\iota}(\mu_{\ell^{\infty }},\bbs^1)=\{ j\colon\mu_{\ell^{\infty }}\to\bbs^1\mid j\rvert_{\mu_{\ell^n}}=\iota  \}
	\end{equation*}
	is a $U_{\ell^n}$-torsor, and it is trivialised by the choice of $\tilde{\iota }\in\Hom_{\iota }(\mu_{\ell^{\infty }},\bbs^1)$ (which is implicit in the construction of the rightmost vertical functor in~(\ref{eqn:DiagramGaloisCategoriesYellnF})). There is then a unique $U_{\ell{}^n}$-equivariant map
	\begin{equation*}
	q\colon\Hom_{\iota }(\mu_{\ell^{\infty }},\bbs^1)\twoheadrightarrow S
	\end{equation*}
	with $q(\tilde{\iota })=s_0$. We define a continuous map
	\begin{equation*}
	Y_{\ell^n,\Fbar }\to Y_{\ell^n,\Fbar }\times S,\quad \chi\mapsto (\chi, q(\chi\rvert_{\mu_{\ell^{\infty }}}))
	\end{equation*}
	which is equivariant for the group inclusion $\Gal (\Fbar /E(S))\hookrightarrow\Gal (\Fbar /F)$ (trivial Galois action on $S$), hence it descends to a continuous map
	\begin{equation*}
	Y_{\ell^n,E(S)}\to Y_{\ell^n,F}\times S.
	\end{equation*}
	It is straightforward to check that this last map is a bijection, hence a homeomorphism, and that it is $U_{(\ell^n)}$-equivariant.
\end{proof}
Choosing a point in $Y_{\ell^{\infty },\Fbar }(\tilde{\iota })$ we obtain a compatible family of fibre functors on all the categories in (\ref{eqn:DiagramGaloisCategoriesYellnF}), hence a commutative diagram of profinite groups.
\begin{Proposition}\label{Prop:SesDiagEtaleFGs}
	The diagram of profinite groups and continuous group homomorphisms
	\begin{equation*}
	\xymatrix{
		1\ar[r]&\pioneet (Y_{\ell^n,F})\ar[r]\ar[d]_-{\cong }&\pioneet ([Y_{\ell^n,F}/U_{(\ell^n)}])\ar[r]\ar@{->>}[d]&\hat{U}_{(\ell^n)}\ar[r]\ar@{->>}[d]&1\\
		1\ar[r]&\Gal (\Fbar /F(\zeta_{\ell^{\infty }}))\ar[r]&\Gal (\Fbar /F)\ar[r]_-{\chi_{\ell ,F}}&U_{\ell^n}\ar[r]&1
		}
	\end{equation*}
	(where $\hat{U}_{(\ell^n)}$ denotes the profinite completion of the abstract group~$U_{(\ell^n)})$ commutes, and the rows are exact. The leftmost vertical map is an isomorphism, the other two vertical maps are surjective but not injective.
\end{Proposition}
\begin{proof}
	The commutativity follows from Lemma~\ref{Lem:DiagramOfGaloisCategoriesCommutes}. The lower row is exact by Galois theory and our assumptions on~$F$; the upper row is exact except possibly at $\pioneet (Y_{\ell^n,F})$ by Proposition~\ref{Prop:SESStackyFibration}. The map $\pioneet (Y_{\ell^n,F})\to\Gal (\Fbar /F(\zeta_{\ell^{\infty }}))$ is an isomorphism by Proposition~\ref{Prop:EtaleFGYellF}.(ii). Combining the last two observations we also obtain exactness of the upper row at~$\pioneet (Y_{\ell^n,F})$. Finally, the rightmost vertical map is clearly surjective but not injective, hence the same holds for the middle vertical map.
\end{proof}
\begin{Corollary}\label{Cor:FinalStepTopologicalGroupTheoretic}
	For each abelian torsion group $A$ the natural isomorphisms of cohomology groups
	\begin{equation*}
	\Hup^m(Y_{\ell^n,F},A)\overset{\cong }{\to }\Hup^m(\pioneet (Y_{\ell^n,F}),A)\overset{\cong }{\to }\Hup^m(F(\zeta_{\ell^{\infty }}),A)
	\end{equation*}
	are equivariant for the canonical group homomorphisms
	\begin{equation*}
	U_{(\ell^n)}\to\hat{U}_{(\ell^n)}\to U_{\ell^n},
	\end{equation*}
	up to an exponent which is $1$ for the first map and $-1$ for the second map.
	
	In particular the $U_{(\ell^n)}$-action on $\Hup^m(Y_{\ell^n,F})$ extends uniquely to a continuous $U_{\ell^n}$-action.
\end{Corollary}
\begin{proof}
	This follows from Proposition~\ref{Prop:SesDiagEtaleFGs}. The exponent $-1$, i.e.\ inversion, on the acting groups for $\Hup^m(\pioneet (Y_{\ell^n,F}),A)\to\Hup^m(F (\zeta_{\ell^{\infty }}),A)$ occurs because $U_{(\ell^n)}$ operates on $Y_{\ell^n,F}$ from the right, see Remark~\ref{Rmk:RightActionSignInSes}.(ii). 
\end{proof}

\subsection{Compatibility of the topological and arithmetic actions} This argument is technically much more difficult than the other two, so considering its logical redundancy we will be very sketchy here.

The essential ingredient is the fact that $\mathcal{Y}_{F,\mathrm{int}}$ is a $\Lambda$-scheme in the sense of Borger \cite{MR2833791}. The general definition of $\Lambda$-schemes is rather involved, but a na\"{\i}ve variant suffices for our purposes.
\begin{Definition} Let $\mathcal{X}$ be a scheme which is flat over $\bbz$. A $\Lambda$-structure on $\mathcal{X}$ is a family of mutually commuting endomorphisms $\varphi_p\colon\mathcal{X}\to\mathcal{X}$, indexed by the rational primes $p=2,3,5,7,\dots$, such that for all $p$ the base change
	\begin{equation*}
	\varphi_p\times\mathrm{id}\colon\mathcal{X}\times\Spec\bbf_p\to\mathcal{X}\times\Spec\bbf_p
	\end{equation*}
	is the absolute Frobenius of $\mathcal{X}_p=\mathcal{X}\times\Spec\bbf_p$, i.e.\ the morphism which is the identity on the topological space underlying $\mathcal{X}_p$ and which is $f\mapsto f^p$ on sections of the structure sheaf.
\end{Definition}
For example let $M$ be an abelian group in multiplicative notation. Then there is a canonical $\Lambda$-structure on $\Spec\bbz [M]$ given by 
\begin{equation}\label{eqn:LambdaStructureOnSpecGroupRing}
\varphi_p^{\sharp}\colon\bbz [M]\to\bbz [M],\qquad [m]\mapsto [m^p].
\end{equation}
We will now construct a $\Lambda$-structure on $\mathcal{Y}_{F,\mathrm{int}}=\Spec B_{F,\mathrm{int}}$ by writing down the corresponding endomorphisms $\varphi_p^{\sharp }$ of $B_{F,\mathrm{int}}$. Recall that
\begin{equation*}
B_{F,\mathrm{int}}=\left( \bbz [\Fbar^{\times }/\mu_{\ell '}] [([\zeta_\ell]-1)^{-1}]\right)^{\Gal (\Fbar /F)}
\end{equation*}
We begin by setting $\varphi_{\ell}^{\sharp }=\mathrm{id}$ (or anything)\footnote{Note that inverting $[\zeta_\ell]-1$ in particular inverts $\ell$, so the condition that $\varphi_\ell^\sharp$ lifts Frobenius is vacuous.}, and next consider the case $p\neq\ell$.

Here we define $\varphi_p^{\sharp}$ on $\bbz [\Fbar^{\times }/\mu_{\ell '}]$ by (\ref{eqn:LambdaStructureOnSpecGroupRing}). This descends to an endomorphism
\begin{equation*}
\varphi_p^{\sharp }\colon \bbz [\Fbar^{\times }/\mu_{\ell '}] [([\zeta_\ell]-1)^{-1}]\to \bbz [\Fbar^{\times }/\mu_{\ell '}] [([\zeta_\ell]-1)^{-1}].
\end{equation*}
\begin{Lemma}
	Let $p\neq\ell$. The canonical ring homomorphism
	\begin{equation*}
	B_{F,\mathrm{int}}\otimes\bbf_p\to\left(\bbf_p[\Fbar^{\times }/\mu_{\ell '}] [([\zeta_\ell]-1)^{-1}]\right)^{\Gal (\Fbar /F)}
	\end{equation*}
	is an isomorphism.
\end{Lemma}
\begin{proof}
	For $F=\Fbar$ this is clear. For the general case, we use that $B_{F,\mathrm{int}}\to B_{\Fbar,\mathrm{int}}$ is a pro-finite \'etale $\Gal(\Fbar/F)$-cover (cf.~Theorem~\ref{thm:Wratfiniteetale}), so that forming quotients under $\Gal(\Fbar/F)$ commutes with base change.
\end{proof}

\begin{Corollary}
	The family of endomorphisms $\varphi_p\colon\mathcal{Y}_{F,\mathrm{int}}\to\mathcal{Y}_{F,\mathrm{int}}$ given on rings by the $\varphi_p^{\sharp}$ as constructed above define a $\Lambda$-structure on~$\mathcal{Y}_{F,\mathrm{int}}$.\hfill $\square$
\end{Corollary}
Note that for $p\neq\ell$ the $\varphi_p$ are even automorphisms of~$\mathcal{Y}_{F,\mathrm{int}}$. They can be described in a different way as follows:

The group $\Fbar^{\times }/\mu_{\ell '}$ is a $\bbz_{(\ell)}$-module, hence $\bbz_{(\ell)}^{\times }$ acts by automorphisms from the left on $\Fbar^{\times }/\mu_{\ell '}$ and also on $\bbz [\Fbar^{\times }/\mu_{\ell '}][([\zeta_\ell]-1)^{-1}]$. Since this action commutes with $\Gal (\Fbar /F)$, it restricts to a left $\bbz_{(\ell)}^{\times }$-action on $B_{F,\mathrm{int}}$; hence a right $\bbz_{(\ell)}^{\times }$-action on $\mathcal{Y}_{F,\mathrm{int}}$.
\begin{Lemma}\label{Lem:TopologicalActionOnSchemeDefinedByLambdaStructure}
	Consider the $\bbz_{(\ell)}^{\times }$-action on $\mathcal{Y}_{F,\mathrm{int}}$ as just described. For each prime $p\neq\ell$ the element $p\in\bbz_{(\ell)}^{\times }$ acts by the automorphism $\varphi_p$, and the element $-1\in\bbz_{(\ell)}^{\times }$ acts by the automorphism $[a]\mapsto [a^{-1}]$ on $\bbz [\Fbar^{\times }/\mu_{\ell '}]$.\hfill $\square$
\end{Lemma}
Note that $\bbz_{(\ell)}^{\times }$ is generated by $-1$ and the primes different from $\ell$, hence the $\bbz_{(\ell)}^{\times }$-action on $\mathcal{Y}_{F,\mathrm{int}}$ is uniquely determined by Lemma~\ref{Lem:TopologicalActionOnSchemeDefinedByLambdaStructure}.

Extending scalars from $\bbz$ to $\bbc$ we then obtain a right $\bbz_ {(\ell)}^{\times }$-action on
\begin{equation*}
\mathcal{Y}_{F,\mathrm{int}}\times\Spec\bbc\cong\coprod_{\iota\colon\mu_{\ell^n}\hookrightarrow\bbs^1}\mathcal{Y}_{\ell^n,F}(\iota ).
\end{equation*}
The stabiliser of each component is $U_{(\ell^n)}$, and the quotient $\bbz_{(\ell)}^{\times }/U_{(\ell^n)}\cong (\bbz /\ell^n\bbz )^{\times }$ operates simply transitively on the set of components. Hence there is a canonical $\bbz_{(\ell)}^\times$-equivariant isomorphism of cohomology groups
\begin{equation*}
\Hup^m(\mathcal{Y}_{F,\mathrm{int}}\otimes\qbar ,A)\cong\Hup^m(\mathcal{Y}_{F,\mathrm{int}}\otimes\bbc ,A)\cong\bigoplus_{\iota }\Hup^m(\mathcal{Y}_{\ell^n,F}(\iota ),A),
\end{equation*}
and each summand on the right hand side is stable under $U_{(\ell^n)}$. By unravelling definitions we see that the $U_{(\ell^n)}$-action on each summand corresponds to the `topological action' described above under the canonical isomorphism $\Hup^m(\mathcal{Y}_{\ell^n,F}(\iota),A)\cong\Hup^m(Y_{\ell^n,F},A)$.

We now draw some consequences from the fact that the $\varphi_p$ constitute a $\Lambda$-structure on~$\mathcal{Y}_{F,\mathrm{int}}$.
\begin{Proposition}\label{Prop:LambdaStructureInverseArithmeticFrobenius}
	For each $p\neq\ell$ and each abelian torsion group $A$ the automorphism of $\Hup^m_{\et }(\mathcal{Y}_{F,\mathrm{int}}\times\Spec\widebar{\bbf }_p,A)$ induced by $\varphi_p\bmod p$ is the inverse of the `arithmetic' automorphism induced by the canonical generator $\sigma_p\in\Gal (\widebar{\bbf}_p/\bbf_p)$ with $\sigma_p(a)=a^p$.
\end{Proposition}
\begin{proof}[Sketch of proof]
	This follows from the fact that $\sigma_p^{\flat }$ and $\varphi_p\bmod p$ commute, their product is the absolute Frobenius endomorphism of the $\bbf_p$-scheme $\mathcal{Y}_{F,\mathrm{int}}\times\Spec\widebar{\bbf}_p$, and this absolute Frobenius acts trivially on \'etale cohomology. The latter is clear for $\Hup^0_{\et }$ and then follows formally for $\Hup^m_{\et }$ by the universal property of sheaf cohomology. See \cite[Rapport, section~1.8]{MR0463174} for a detailed discussion in the finite type case.
\end{proof}
\begin{Proposition}\label{Prop:ActionsByGalFpQpQ}
	Fix an algebraic closure $\qbar_p$ of $\bbq_p$ and an embedding $\qbar\hookrightarrow\qbar_p$. Then for any abelian torsion group $A$ and any $m\ge 0$ the natural maps
	\begin{equation*}
	\Hup^m_{\et }(\mathcal{Y}_{F,\mathrm{int}}\otimes\widebar{\bbf}_p,A)\leftarrow\Hup^m_{\et }(\mathcal{Y}_{F,\mathrm{int}}\otimes\mathcal{O}_{\qbar_p},A)\to\Hup^m_{\et }(\mathcal{Y}_{F,\mathrm{int}}\otimes\qbar_p,A)\leftarrow\Hup^m_{\et }(\mathcal{Y}_{F,\mathrm{int}}\otimes\qbar ,A)
	\end{equation*}
	are isomorphisms, and they are equivariant for the action of $\bbz_{(\ell)}^\times$ and the homomorphisms of absolute Galois groups
	\begin{equation*}
	\Gal (\widebar{\bbf}_p/\bbf_p)\leftarrow \Gal (\qbar_p/\bbq_p)=\Gal (\qbar_p/\bbq_p)\to\Gal (\qbar /\bbq ).
	\end{equation*}
\end{Proposition}
\begin{proof}[Sketch of proof]
	This can be shown identifying each of the \'etale cohomology groups with $\Hup^m(F(\zeta_{\ell^{\infty }}),A)$.
\end{proof}

The next proposition is a variant of a result of Borger, \cite[Theorem 6.1]{borger}.

\begin{Proposition}\label{Prop:GaloisOnEtaleCohoFactThroughChiEll}
	The action of $\Gal (\qbar /\bbq )$ on $\Hup^m_{\et }(\mathcal{Y}_{F,\mathrm{int}}\otimes\qbar ,A)$ factors through the $\ell$-adic cyclotomic character $\chi_{\ell ,\bbq}\colon\Gal (\qbar /\bbq )\to\bbz_{\ell}^{\times }$.
\end{Proposition}

\begin{proof} By Proposition~\ref{Prop:ActionsByGalFpQpQ} and Proposition~\ref{Prop:LambdaStructureInverseArithmeticFrobenius}, the action of any two Frobenius elements commute, as they can be identified with the action of the commuting operators $\varphi_p$. By Chebotarev, we see that the action of $\Gal(\qbar/\bbq)$ factors through its maximal abelian quotient; also, it is unramified at all primes different from~$\ell$ by Proposition~\ref{Prop:ActionsByGalFpQpQ}. This implies that it factors through the quotient $\Gal (\bbq (\zeta_{\ell^{\infty }})/\bbq)$, which is precisely the quotient defined by the $\ell$-adic cyclotomic character.
\end{proof}

\begin{Proposition}\label{Prop:ActionByGlobalFrobeniusAndComplexConjugation}
	\begin{enumerate}
		\item Let $p\neq\ell$ be a rational prime, and let $\sigma_p\in\Gal (\qbar /\bbq )$ be such that $\chi_{\ell }(\sigma_p)=p\in\bbz_{\ell}^{\times }$. Then $\sigma_p$ operates on $\Hup^m_{\et }(\mathcal{Y}_{F,\mathrm{int}}\otimes\qbar ,A)$ through the inverse of~$\varphi_p$.
		\item Let $\sigma_{-1}\in\Gal (\qbar /\bbq )$ be such that $\chi_{\ell }(\sigma_{-1})=-1\in\bbz_{\ell}^{\times }$. Then $\sigma_{-1}$ operates on $\Hup^m_{\et }(\mathcal{Y}_{F,\mathrm{int}}\otimes\qbar ,A)$ through the involution induced by the involution $[a]\mapsto [a^{-1}]$ on $\bbz [\Fbar^{\times }/\mu_{\ell '}]$.
	\end{enumerate}
\end{Proposition}
\begin{proof}
	(i) follows from the conjunction of Propositions \ref{Prop:LambdaStructureInverseArithmeticFrobenius}, \ref{Prop:ActionsByGalFpQpQ} and~\ref{Prop:GaloisOnEtaleCohoFactThroughChiEll}. For (ii), by Proposition~\ref{Prop:GaloisOnEtaleCohoFactThroughChiEll} we may assume that $\sigma_{-1}$ is complex conjugation, whose action on cohomology is easily determined by contemplating the isomorphisms
	\begin{equation*}
	\Hup^m_{\et }(\mathcal{Y}_{F,\mathrm{int}}\otimes\qbar ,A)\cong\Hup^m_{\et }(\mathcal{Y}_{F,\mathrm{int}}\otimes\bbc ,A)\cong\Hup^m(\mathcal{Y}_{F,\mathrm{int}}(\bbc ),A)\cong\bigoplus_{\iota }\Hup^m(Y_{\ell^n,F}(\iota),A).\qedhere
	\end{equation*}
\end{proof}
\begin{Proposition}\label{Prop:FinalStepCompatibilityTopologicalArithmetic}
	The canonical isomorphism
	\begin{equation*}
	\Hup^m(Y_{\ell^n,F}(\iota ),A)\to\Hup^m_{\et }(\mathcal{Y}_{F,\mathrm{int}}\times_{e_{\iota },\,\Spec\bbz [\zeta_{\ell^n}]}\Spec\qbar ,A)
	\end{equation*}
	is equivariant for the \emph{inverse} inclusion $U_{(\ell^n)}\to U_{\ell^n}$, $u\mapsto u^{-1}$, in the sense that for any $u\in U_{(\ell^n)}$ and any $\sigma\in\Gal (\qbar /\bbq (\zeta_{\ell^n}))$ with $\chi_{\ell ,\bbq }(\sigma )=u^{-1}$ the action by $u$ on the left hand side agrees with the action by $\sigma$ on the right hand side.
\end{Proposition}
\begin{proof}
	It suffices to show that the canonical isomorphism
	\begin{equation*}
	\bigoplus_{\iota\colon\mu_{\ell^n}\hookrightarrow\bbs^1}\Hup^m(Y_{\ell^n,F}(\iota ),A)\to\Hup^m_{\et }(\mathcal{Y}_{F,\mathrm{int}}\times_{\Spec\bbz }\Spec\qbar ,A)
	\end{equation*}
	is equivariant for the inverse inclusion $\bbz_{(\ell)}^{\times }\to\bbz_{\ell}^{\times }$, $u\mapsto u^{-1}$, where the $\bbz_{\ell }^{\times }$-action on the right hand side is given by Proposition~\ref{Prop:GaloisOnEtaleCohoFactThroughChiEll}. Since $\bbz_{(\ell)}^{\times }$ is generated by $-1$ and the primes different from $\ell$ it is sufficient to check this for these elements. The combination of Lemma~\ref{Lem:TopologicalActionOnSchemeDefinedByLambdaStructure} and Proposition~\ref{Prop:ActionByGlobalFrobeniusAndComplexConjugation} yields the desired result.
\end{proof}

\providecommand{\bysame}{\leavevmode\hbox to3em{\hrulefill}\thinspace}
\providecommand{\MR}{\relax\ifhmode\unskip\space\fi MR }
\providecommand{\MRhref}[2]{%
	\href{http://www.ams.org/mathscinet-getitem?mr=#1}{#2}
}
\providecommand{\href}[2]{#2}


\begin{thebibliography}{10}
	
	\bibitem{StacksProject}
	\emph{The {S}tacks {P}roject}, website, available at
	http://stacks.math.columbia.edu/.
	
	\bibitem{MR2017446}
	\emph{Rev\^etements \'{e}tales et groupe fondamental ({SGA} 1)}, Documents
	Math\'{e}matiques (Paris) [Mathematical Documents (Paris)], 3,
	Soci\'{e}t\'{e} Math\'{e}matique de France, Paris, 2003, S\'{e}minaire de
	g\'{e}om\'{e}trie alg\'{e}brique du Bois Marie 1960--61. [Algebraic Geometry
	Seminar of Bois Marie 1960-61], Directed by A. Grothendieck, With two papers
	by M. Raynaud, Updated and annotated reprint of the 1971 original [Lecture
	Notes in Math., 224, Springer, Berlin; MR0354651 (50 \#7129)]. \MR{2017446}
	
	\bibitem{MR2035696}
	Alejandro Adem and R.~James Milgram, \emph{Cohomology of finite groups}, second
	ed., Grundlehren der Mathematischen Wissenschaften [Fundamental Principles of
	Mathematical Sciences], vol. 309, Springer-Verlag, Berlin, 2004. \MR{2035696}
	
	\bibitem{MR0424786}
	Gert Almkvist, \emph{Endomorphisms of finitely generated projective modules
		over a commutative ring}, Ark. Mat. \textbf{11} (1973), 263--301.
	\MR{0424786}
	
	\bibitem{MR523461}
	\bysame, \emph{{$K$}-theory of endomorphisms}, J. Algebra \textbf{55} (1978),
	no.~2, 308--340. \MR{523461}
	
	\bibitem{MR3069467}
	Emil Artin and Otto Schreier, \emph{Algebraische {K}onstruktion reeller
		{K}\"{o}rper}, Abh. Math. Sem. Univ. Hamburg \textbf{5} (1927), no.~1,
	85--99. \MR{3069467}
	
	\bibitem{MR3379634}
	Bhargav Bhatt and Peter Scholze, \emph{The pro-\'{e}tale topology for schemes},
	Ast\'{e}risque (2015), no.~369, 99--201. \MR{3379634}
	
	\bibitem{MR849653}
	Spencer Bloch and Kazuya Kato, \emph{{$p$}-adic \'{e}tale cohomology}, Inst.
	Hautes \'{E}tudes Sci. Publ. Math. (1986), no.~63, 107--152. \MR{849653}
	
	\bibitem{borger}
	James Borger, \emph{{$\Lambda$}-rings and the field with one element},  (2009),
	arXiv:0906.3146.
	
	\bibitem{MR2833791}
	\bysame, \emph{The basic geometry of {W}itt vectors, {I}: {T}he affine case},
	Algebra Number Theory \textbf{5} (2011), no.~2, 231--285. \MR{2833791}
	
	\bibitem{bostconnes}
	J.-B. Bost and A.~Connes, \emph{Hecke algebras, type {III} factors and phase
		transitions with spontaneous symmetry breaking in number theory}, Selecta
	Math. (N.S.) \textbf{1} (1995), no.~3, 411--457. \MR{1366621}
	
	\bibitem{MR2954666}
	Jeremy Brazas, \emph{Semicoverings: a generalization of covering space theory},
	Homology Homotopy Appl. \textbf{14} (2012), no.~1, 33--63. \MR{2954666}
	
	\bibitem{MR2995090}
	\bysame, \emph{The fundamental group as a topological group}, Topology Appl.
	\textbf{160} (2013), no.~1, 170--188. \MR{2995090}
	
	\bibitem{MR0346781}
	Joel~M. Cohen, \emph{Homotopy groups of inverse limits}, Proceedings of the
	{A}dvanced {S}tudy {I}nstitute on {A}lgebraic {T}opology ({A}arhus {U}niv.,
	{A}arhus, 1970), {V}ol. {I}, Mat. Inst., Aarhus Univ., Aarhus, 1970,
	pp.~29--43. Various Publ. Ser., No. 13. \MR{0346781}
	
	\bibitem{MR0463174}
	Pierre Deligne, \emph{Cohomologie \'{e}tale}, Lecture Notes in Mathematics,
	Vol. 569, Springer-Verlag, Berlin-New York, 1977, S\'{e}minaire de
	G\'{e}om\'{e}trie Alg\'{e}brique du Bois-Marie SGA 4${1\over 2}$, avec la
	collaboration de J. F. Boutot, A. Grothendieck, L. Illusie et J. L. Verdier.
	\MR{0463174 (57 \#3132)}
	
	\bibitem{MR0050886}
	Samuel Eilenberg and Norman Steenrod, \emph{Foundations of algebraic topology},
	Princeton University Press, Princeton, New Jersey, 1952. \MR{0050886}
	
	\bibitem{MR2810974}
	Paul Fabel, \emph{Multiplication is discontinuous in the {H}awaiian earring
		group (with the quotient topology)}, Bull. Pol. Acad. Sci. Math. \textbf{59}
	(2011), no.~1, 77--83. \MR{2810974}
	
	\bibitem{FontaineWintenberger}
	Jean-Marc Fontaine and Jean-Pierre Wintenberger, \emph{Extensions
		alg\'{e}brique et corps des normes des extensions {APF} des corps locaux}, C.
	R. Acad. Sci. Paris S\'{e}r. A-B \textbf{288} (1979), no.~8, A441--A444.
	\MR{527692}
	
	\bibitem{MR0289476}
	L.~Fuchs and F.~Loonstra, \emph{On the cancellation of modules in direct sums
		over {D}edekind domains.}, Nederl. Akad. Wetensch. Proc. Ser. A 74 = Indag.
	Math. \textbf{33} (1971), 163--169. \MR{0289476}
	
	\bibitem{MR0345092}
	Roger Godement, \emph{Topologie alg\'{e}brique et th\'{e}orie des faisceaux},
	Hermann, Paris, 1973, Troisi\`{e}me \'{e}dition revue et corrig\'{e}e,
	Publications de l'Institut de Math\'{e}matique de l'Universit\'{e} de
	Strasbourg, XIII, Actualit\'{e}s Scientifiques et Industrielles, No. 1252.
	\MR{0345092}
	
	\bibitem{MR0238860}
	A.~Grothendieck, \emph{\'{E}l\'{e}ments de g\'{e}om\'{e}trie alg\'{e}brique.
		{IV}. \'{E}tude locale des sch\'{e}mas et des morphismes de sch\'{e}mas
		{IV}}, Inst. Hautes \'{E}tudes Sci. Publ. Math. (1967), no.~32, 361.
	\MR{0238860}
	
	\bibitem{MR1867354}
	Allen Hatcher, \emph{Algebraic topology}, Cambridge University Press,
	Cambridge, 2002. \MR{1867354}
	
	\bibitem{MR2553661}
	Michiel Hazewinkel, \emph{Witt vectors. {I}}, Handbook of algebra. {V}ol. 6,
	Handb. Algebr., vol.~6, Elsevier/North-Holland, Amsterdam, 2009,
	pp.~319--472. \MR{2553661}
	
	\bibitem{Zahlbericht}
	David Hilbert, \emph{Die {T}heorie der algebraischen {Z}ahlk\"{o}rper},
	Jahresber. Deutsch. Math.-Verein. \textbf{4} (1897), 175--546.
	
	\bibitem{Hirschhorn2015}
	Philip~S. Hirschhorn, \emph{The homotopy groups of the inverse limit of a tower
		of fibrations}, 2015, preprint,
	\texttt{http://www-math.mit.edu/\~{}psh/notes/limfibrations.pdf}.
	
	\bibitem{MR1009787}
	Nathan Jacobson, \emph{Basic algebra. {II}}, second ed., W. H. Freeman and
	Company, New York, 1989. \MR{1009787}
	
	\bibitem{MR0260842}
	J.~L. Kelley and E.~H. Spanier, \emph{Euler characteristics}, Pacific J. Math.
	\textbf{26} (1968), 317--339. \MR{0260842}
	
	\bibitem{Klevdal2015}
	Christian Klevdal, \emph{A {G}alois {C}orrespondence with {G}eneralized
		{C}overing {S}paces}, 2015, Undergraduate Honors Thesis, University of
	Colorado, Boulder.
	
	\bibitem{MR1878556}
	Serge Lang, \emph{Algebra}, third ed., Graduate Texts in Mathematics, vol. 211,
	Springer-Verlag, New York, 2002. \MR{1878556}
	
	\bibitem{MR0240214}
	Hideya Matsumoto, \emph{Sur les sous-groupes arithm\'{e}tiques des groupes
		semi-simples d\'{e}ploy\'{e}s}, Ann. Sci. \'{E}cole Norm. Sup. (4) \textbf{2}
	(1969), 1--62. \MR{0240214}
	
	\bibitem{MR0258786}
	Warren May, \emph{Unit groups of infinite abelian extensions}, Proc. Amer.
	Math. Soc. \textbf{25} (1970), 680--683. \MR{0258786}
	
	\bibitem{MR559531}
	James~S. Milne, \emph{\'{E}tale cohomology}, Princeton Mathematical Series,
	vol.~33, Princeton University Press, Princeton, N.J., 1980. \MR{559531}
	
	\bibitem{MR0260844}
	John Milnor, \emph{Algebraic {$K$}-theory and quadratic forms}, Invent. Math.
	\textbf{9} (1969/1970), 318--344. \MR{0260844}
	
	\bibitem{Munkres2000}
	James~R. Munkres, \emph{Topology}, second ed., Prentice Hall, Upper Saddle
	River, NJ, 2000.
	
	\bibitem{MR1697859}
	J\"{u}rgen Neukirch, \emph{Algebraic number theory}, Grundlehren der
	Mathematischen Wissenschaften [Fundamental Principles of Mathematical
	Sciences], vol. 322, Springer-Verlag, Berlin, 1999, Translated from the 1992
	German original and with a note by Norbert Schappacher, With a foreword by G.
	Harder. \MR{1697859}
	
	\bibitem{MR1503168}
	L.~Pontrjagin, \emph{The theory of topological commutative groups}, Ann. of
	Math. (2) \textbf{35} (1934), no.~2, 361--388. \MR{1503168}
	
	\bibitem{MR0152834}
	Walter Rudin, \emph{Fourier analysis on groups}, Interscience Tracts in Pure
	and Applied Mathematics, No. 12, Interscience Publishers (a division of John
	Wiley and Sons), New York-London, 1962. \MR{0152834}
	
	\bibitem{MR1620705}
	Peter Schneider, \emph{Equivariant homology for totally disconnected groups},
	J. Algebra \textbf{203} (1998), no.~1, 50--68. \MR{1620705}
	
	\bibitem{scholzethesis}
	Peter Scholze, \emph{Perfectoid spaces}, Publ. Math. Inst. Hautes \'{E}tudes
	Sci. \textbf{116} (2012), 245--313. \MR{3090258}
	
	\bibitem{MR0082175}
	Jean-Pierre Serre, \emph{G\'{e}om\'{e}trie alg\'{e}brique et g\'{e}om\'{e}trie
		analytique}, Ann. Inst. Fourier, Grenoble \textbf{6} (1955--1956), 1--42.
	\MR{0082175 (18,511a)}
	
	\bibitem{MR0180551}
	\bysame, \emph{Cohomologie galoisienne}, Cours au Coll\`{e}ge de France, vol.
	1962, Springer-Verlag, Berlin-Heidelberg-New York, 1962/1963. \MR{0180551 (31
		\#4785)}
	
	\bibitem{MR0357114}
	Saharon Shelah, \emph{Infinite abelian groups, {W}hitehead problem and some
		constructions}, Israel J. Math. \textbf{18} (1974), 243--256. \MR{0357114}
	
	\bibitem{MR0389579}
	\bysame, \emph{A compactness theorem for singular cardinals, free algebras,
		{W}hitehead problem and transversals}, Israel J. Math. \textbf{21} (1975),
	no.~4, 319--349. \MR{0389579}
	
	\bibitem{MR507446}
	Lynn~Arthur Steen and J.~Arthur Seebach, Jr., \emph{Counterexamples in
		topology}, second ed., Springer-Verlag, New York-Heidelberg, 1978.
	\MR{507446}
	
	\bibitem{MR0043219}
	Karl Stein, \emph{Analytische {F}unktionen mehrerer komplexer
		{V}er\"{a}nderlichen zu vorgegebenen {P}eriodizit\"{a}tsmoduln und das zweite
		{C}ousinsche {P}roblem}, Math. Ann. \textbf{123} (1951), 201--222.
	\MR{0043219}
	
	\bibitem{MR2811603}
	Vladimir Voevodsky, \emph{On motivic cohomology with {$\bold
			Z/l$}-coefficients}, Ann. of Math. (2) \textbf{174} (2011), no.~1, 401--438.
	\MR{2811603}
	
	\bibitem{WeinsteinGeomGal}
	Jared Weinstein, \emph{{$\mathrm{Gal}(\qbar_p/\bbq_p)$} as
		a geometric fundamental group},  (2014), arxiv:1404.7192.
	
\end{thebibliography}
\end{document}